\documentclass[11pt,reqno]{amsart}
\usepackage{latexsym,amsmath}
\usepackage{amsfonts}
\usepackage{amssymb}
\usepackage{caption}
\usepackage{latexsym}
\usepackage{tikz}
\usepackage{caption}
\usepackage{float}
\usepackage{graphicx}
\usepackage{enumerate}

\allowdisplaybreaks

\def\E{\mathbb{E}}

\def\E{\mathbb{E}}

\def\eps{\epsilon}

\def\1{\mathbf{1}}

\def\lam {\lambda}
\def\Lam{\Lambda}
\def\tce{t_c + \eps}
\def\tce2{t_c + \frac{\eps}{2}}

\newtheorem*{theorem*}{Theorem}
\newtheorem{theorem}{Theorem}
\newtheorem{lemma}[theorem]{Lemma}
\newtheorem{cor}[theorem]{Corollary}

\newtheorem*{prop*}{Proposition}
\newtheorem{conj}{Conjecture}
\newtheorem{claim}[theorem]{Claim}
\newtheorem{question}{Question}

\begin{document}
\title{Counting independent sets in cubic graphs of given girth}
\author{Guillem Perarnau}
\author{Will Perkins}
\thanks{WP supported in part by EPSRC grant EP/P009913/1.}
\address{University of Birmingham}
\email{g.perarnau@bham.ac.uk,math@willperkins.org}
\date{\today}


\begin{abstract}
We prove a tight upper bound on the independence polynomial (and total number of independent sets) of cubic graphs of girth at least $5$.  The bound is achieved by unions of the Heawood graph, the point/line incidence graph of the Fano plane.  

We also give a tight lower bound on the total number of independent sets of triangle-free cubic graphs. This bound is achieved by unions of the Petersen graph.

We conjecture that in fact all Moore graphs are extremal for the scaled number of independent sets in regular graphs of a given minimum girth, maximizing this quantity if their girth is even and minimizing if odd. The Heawood and Petersen graphs are instances of this conjecture, along with complete graphs, complete bipartite graphs, and cycles.  
\end{abstract}

\maketitle

\section{Independent sets in regular graphs}

A classic theorem of Kahn~\cite{kahn2001entropy} states that a union of $n/2d$ copies of the complete $d$-regular bipartite graph ($K_{d,d}$) has the most independent sets of all $d$-regular bipartite graphs on $n$ vertices.  Zhao~\cite{zhao2010number} extended this to all $d$-regular graphs.  A result of Galvin and Tetali~\cite{galvin2004weighted} for bipartite graphs combined with Zhao's result shows that maximality of $K_{d,d}$ holds at the level of the {\em independence polynomial},
\begin{equation}
P_G(\lam) = \sum_{I \in \mathcal I(G)} \lam^{|I|},
\end{equation}
where $\mathcal I(G)$ is the set of all independent sets of $G$.
\begin{theorem}[Kahn, Galvin--Tetali, Zhao~\cite{kahn2001entropy,galvin2004weighted,zhao2010number}]
\label{thm:KddPart}
For every $d$-regular graph $G$ and all $\lam>0$,
\begin{equation}
\label{eq:kahn}
\frac{1}{|V(G)|} \log P_G(\lam) \le \frac{1}{2d} \log P_{K_{d,d}}(\lam).
\end{equation}
\end{theorem}
 The result on the number of independent sets in a regular graph is recovered by taking $\lam=1$ and noting that the independence polynomial is multiplicative over taking disjoint unions of graphs. 

The function $P_G(\lam)$ is also known as the partition function (the normalizing constant) of the {\em hard-core model} from statistical physics.  The hard-core model is a probability distribution over the independent sets of a graph $G$, parametrized by a positive real number $\lam$, the {\em fugacity}.  The distribution is given by:
\[ \Pr [I] = \frac{\lam^{|I|}}{ P_G(\lam)} .\]

The  derivative of $\frac{1}{|V(G)|} \log P_G(\lam) $ has a nice probabilistic interpretation: it is the {\em occupancy fraction}, $\alpha_G(\lam)$, the expected fraction of vertices of $G$ in the random independent set drawn from the hard-core model:
\begin{align*}
\alpha_G(\lam) &:= \frac{1}{|V(G)|} \E |I| \\
&= \frac{ \lam P_G'(\lam)  }{|V(G)| \cdot P_G(\lam)} \\
&= \frac{\lam }{|V(G)|}  (\log P_G(\lam))'.
\end{align*}
Davies, Jenssen, Perkins and Roberts recently gave a strengthening of Theorem~\ref{thm:KddPart}, showing that~\eqref{eq:kahn} holds at the level of the occupancy fraction.  
\begin{theorem}[Davies, Jenssen, Perkins, Roberts~\cite{davies2015independent}]
\label{thm:KddOcc} 
For every $d$-regular graph $G$ and all $\lam>0$,
\begin{equation}
\label{eq:occupancy}
\alpha_G(\lam) \le \alpha_{K_{d,d}}(\lam).
\end{equation}
\end{theorem}
Theorem~\ref{thm:KddPart} can be recovered from Theorem~\ref{thm:KddOcc} by noting that $\log P_G(0)= 0$ for all $G$, and then integrating $\frac{\alpha_G(t)}{t}$ from $0$ to $\lam$.

Now $K_{d,d}$ contains many copies of the $4$-cycle, $C_4$, as subgraphs (in fact the highest possible $C_4$ density of a $d$-regular triangle-free graph). Heuristically we might imagine that having many short even cycles increases the independent set density, while having odd cycles decreases it.  So what happens if we forbid $4$-cycles?

Similarly, Cutler and Radcliffe~\cite{cutler2014maximum} show that $\frac{1}{|V(G)|} \log P_G(\lam)$ is {\em minimized} over all $d$-regular graphs by a union of copies of $K_{d+1}$, the complete graph on $d+1$ vertices.  But $K_{d+1}$ has (many) triangles.  So what happens if we forbid triangles?

\begin{question}
\noindent
\begin{itemize}
\item Which $d$-regular graphs of girth at least $4$ have the fewest independent sets? 
\item Which $d$-regular graphs of girth at least $5$ have the most independent sets? 
\item More generally, for $g$ even, which $d$-regular graphs of girth at least $g$ have the fewest independent sets, and for $g$ odd, which $d$-regular graphs of girth at least $g$ have the most independent sets?
\end{itemize}
\end{question} 
Here we answer the first two questions for the class of cubic ($3$-regular) graphs: the triangle-free cubic graphs with the fewest independent sets are copies of  the Petersen graph $P_{5,2}$, and the cubic graphs of girth at least $5$ with the most independent sets are copies of the Heawood graph $H_{3,6}$.

\begin{center}
\begin{figure} 
\caption{}
\label{fig:P52_H36}
\begin{tabular}{ c c }
  
\begin{tikzpicture}
    \node[shape=circle,fill=black] (A1) at (0,1) {};
     \node[shape=circle,fill=black] (A2) at (0,2) {};
    \node[shape=circle,fill=black] (B1) at (0.5877,-0.809) {};
    \node[shape=circle,fill=black] (B2) at (1.1754,-1.618) {};
    \node[shape=circle,fill=black] (C1) at (-0.5877,-0.809)  {};
    \node[shape=circle,fill=black] (C2) at (-1.1754,-1.618) {};
    \node[shape=circle,fill=black] (D1) at (0.951,0.309) {};
    \node[shape=circle,fill=black] (D2) at (1.902,0.618) {};
    \node[shape=circle,fill=black] (E1) at (-0.951,0.309) {};
    \node[shape=circle,fill=black] (E2) at (-1.902,0.618) {};

\draw plot [smooth, tension=2] coordinates { (A1) (A2) };
\draw plot [smooth, tension=2] coordinates { (B1) (B2) };
\draw plot [smooth, tension=2] coordinates { (C1) (C2) };
\draw plot [smooth, tension=2] coordinates { (D1) (D2) };
\draw plot [smooth, tension=2] coordinates { (E1) (E2) };

\draw plot [smooth, tension=2] coordinates { (A1) (B1) };
\draw plot [smooth, tension=2] coordinates { (B1) (E1) };
\draw plot [smooth, tension=2] coordinates { (E1) (D1) };
\draw plot [smooth, tension=2] coordinates { (D1) (C1) };
\draw plot [smooth, tension=2] coordinates { (C1) (A1) };

\draw plot [smooth, tension=2] coordinates { (A2) (D2) };
\draw plot [smooth, tension=2] coordinates { (D2) (B2) };
\draw plot [smooth, tension=2] coordinates { (B2) (C2) };
\draw plot [smooth, tension=2] coordinates { (C2) (E2) };
\draw plot [smooth, tension=2] coordinates { (E2) (A2) };

\end{tikzpicture}
 
  &

 \hspace{0.5cm}
\begin{tikzpicture}
    \node[shape=circle,fill=black] (A1) at (0,2) {};
    \node[shape=circle,fill=black] (A2) at (0,-2) {};
   \node[shape=circle,fill=black] (C1) at (0.868,1.802) {};
    \node[shape=circle,fill=black] (C2) at (0.868,-1.802) {};
   \node[shape=circle,fill=black] (C3) at (-0.868,1.802) {};
    \node[shape=circle,fill=black] (C4) at (-0.868,-1.802) {};
    
    \node[shape=circle,fill=black] (D1) at (1.564,1.247) {};
    \node[shape=circle,fill=black] (D2) at (1.564,-1.247) {};
   \node[shape=circle,fill=black] (D3) at (-1.564,1.247) {};
    \node[shape=circle,fill=black] (D4) at (-1.564,-1.247) {};
    \node[shape=circle,fill=black] (E1) at (1.95,0.445) {};
    \node[shape=circle,fill=black] (E2) at (1.95,-0.445) {};
    \node[shape=circle,fill=black] (E3) at (-1.95,0.445) {};
    \node[shape=circle,fill=black] (E4) at (-1.95,-0.445) {};

\draw plot [smooth, tension=2] coordinates { (A1) (C1) };
\draw plot [smooth, tension=2] coordinates { (C1) (D1) };
\draw plot [smooth, tension=2] coordinates { (D1) (E1) };
\draw plot [smooth, tension=2] coordinates { (E1) (E2) };
\draw plot [smooth, tension=2] coordinates { (E2) (D2) };
\draw plot [smooth, tension=2] coordinates { (D2) (C2) };
\draw plot [smooth, tension=2] coordinates { (C2) (A2) };
\draw plot [smooth, tension=2] coordinates { (A2) (C4) };
\draw plot [smooth, tension=2] coordinates { (C4) (D4) };
\draw plot [smooth, tension=2] coordinates { (D4) (E4) };
\draw plot [smooth, tension=2] coordinates { (E4) (E3) };
\draw plot [smooth, tension=2] coordinates { (E3) (D3) };
\draw plot [smooth, tension=2] coordinates { (D3) (C3) };
\draw plot [smooth, tension=2] coordinates { (C3) (A1) };

\draw plot [smooth, tension=2] coordinates { (E4) (C1) };
\draw plot [smooth, tension=2] coordinates { (E3) (C2) };
\draw plot [smooth, tension=2] coordinates { (E1) (D3) };
\draw plot [smooth, tension=2] coordinates { (E2) (D4) };
\draw plot [smooth, tension=2] coordinates { (C3) (C4) };
\draw plot [smooth, tension=2] coordinates { (A1) (D2) };
\draw plot [smooth, tension=2] coordinates { (A2) (D1) };
\end{tikzpicture}
\\
\\
Petersen Graph $P_{5,2}$
&
\hspace{0.5cm}
Heawood Graph $H_{3,6}$
  \end{tabular}
  \end{figure}
  \end{center}

  Notably, in all of the cases that we know ($K_{d+1}, K_{d,d}$, the cycles $C_n$, and the Petersen and Heawood graphs), the  optimizing graph is a  {\em Moore graph}.  A $(d,g)$-Moore graph, for $g$ odd, is a $d$-regular graph with girth $g$, diameter $(g-1)/2$ and exactly
\[ 1 + d \sum_{j=0}^{(g-3)/2} (d-1)^j \]
 vertices.  If $g$ is even, then a $(d,g)$-Moore graph is $d$-regular, has girth $g$  and exactly
 \[  1 + (d-1)^{g/2 -1} + d \sum_{j=0}^{(g-4)/2} (d-1)^j  \]
vertices.   Moore graphs are necessarily {\em cages}: regular graphs with the fewest number of vertices for their girth.  It is natural to consider the maximization problem for graphs of girth at least an odd integer $g$ and the minimization problem for graphs of girth at least an even integer $g$ as we expect short even cycles to encourage more independent sets and short odd cycles to suppress independent sets, and thus the solution to the maximization problem for even $g$ will be the same as the solution of the maximization problem for $g-1$, and likewise for minimization.  This intuition is borne out in all of the above examples.

Moore graphs do not exist for every pair $d,g$.  But we conjecture that if such a Moore graph exists, then it is extremal for the scaled number of independent sets in a $d$-regular graph of girth at least $g-1$ (and of course extremal for graphs of girth at least $g$ as well).
 
 \begin{conj}
 \label{conj:Moore}
 Suppose $g$ is odd and there exists a $(d,g)$-Moore graph, $G^*_{d,g}$.  Then for every $d$-regular graph $G$ of girth at least $g-1$,
 \[  \frac{1}{|V(G)|} \log | \mathcal I(G)| \ge \frac{1}{|V(G^*_{d,g})|} \log | \mathcal I(G^*_{d,g}) |     .\] 
 Suppose $g$ is even and there exists a $(d,g)$-Moore graph $G^*_{d,g}$.  Then for every $d$-regular graph $G$ of girth at least $g-1$,
 \[  \frac{1}{|V(G)|} \log | \mathcal I(G)| \le \frac{1}{|V(G^*_{d,g})|} \log | \mathcal I(G^*_{d,g}) |     .\] 
 \end{conj}

\subsection{The Petersen graph}

The {\em Petersen graph}, $P_{5,2}$, has $10$ vertices, is $3$-regular and vertex-transitive, has girth $5$ and is a $(3,5)$-Moore graph (see Figure~1).  Its independence polynomial is
\[ P_{P_{5,2}}(\lam) = 1+ 10\lam +30 \lam^2 + 30 \lam^3 + 5 \lambda ^4,\]
and its occupancy fraction is
\begin{equation}
\alpha_{P_{5,2}}(\lam) = \frac{\lambda  \left(1+ 6 \lam + 9\lam^2 +2 \lambda ^3\right)}{P_{P_{5,2}}(\lam)}.
\end{equation}

Our first result provides a tight lower bound on the occupancy fraction of triangle-free cubic graphs for every $\lam\in (0,1]$:
\begin{theorem}
\label{thm:d3g4}
For any triangle-free, cubic graph $G$, and for every $\lam \in (0,1]$,
\[ \alpha_G(\lam) \ge \alpha_{P_{5,2}} (\lam),\]
with equality if and only if $G$ is a union of copies of $P_{5,2}$. 
\end{theorem}
By integrating $\frac{\alpha_G(\lam)}{\lam}$ from $\lam=0$ to $1$ we obtain the corresponding counting result:
\begin{cor}
\label{cor:Peter}
For any triangle-free, cubic graph $G$, and any $\lam \in (0,1]$,
\[ \frac{1}{|V(G)|} \log P_G(\lam) \ge  \frac{1}{10} \log P_{P_{5,2}}(\lam), \] 
and in particular,
\[  \frac{1}{|V(G)|} \log | \mathcal I(G)| \ge \frac{1}{10} \log | \mathcal I(P_{5,2}) | ,\]
with equality if and only if $G$ is a union of copies of $P_{5,2}$. 
\end{cor}

The inequality of course holds for $\lam =0$ but we do not have uniqueness, as $P_G(0) =1$ for all $G$.

Recently, Cutler and Radcliffe~\cite{cutler2016minimum} proved a lower bound on the occupancy fraction of triangle-free cubic graphs that gave $\frac{1}{|V(G)|} \log | \mathcal I(G)| \ge 0.430703$, while the tight bound above is $\frac{1}{10} \log | \mathcal I(P_{5,2}) | = \frac{\log (76  )}{10} \approx 0.43307 $.  Cutler and Radcliffe also conjectured that for every cubic triangle-free graph $G$ on $n$ vertices and for every $\lam>0$
\[\frac{1}{n} \log P_G(\lam) \geq\min \left\{ \frac{1}{10} \log P_{P_{5,2}}(\lam),\frac{1}{14} \log P_{P_{7,2}}(\lam)\right\},\]
where $P_{7,2}$ is the (7,2)-Generalized Petersen graph (see Figure~2).
Theorem~\ref{thm:d3g4} proves this conjecture for $\lam \in (0,1]$.

\begin{center}
\begin{figure}[h!]\label{fig:P72} 
\caption{}
\begin{tabular}{ c }
  
 \hspace{0.5cm}
\begin{tikzpicture}
    \node[shape=circle,fill=black] (A1) at (0,1) {};
     \node[shape=circle,fill=black] (A2) at (0,2) {};
    \node[shape=circle,fill=black] (B1) at (0.439,-0.901) {};
    \node[shape=circle,fill=black] (B2) at (0.878,-1.802) {};
     \node[shape=circle,fill=black] (C1) at (-0.439,-0.901) {};
     \node[shape=circle,fill=black] (C2) at (-0.878,-1.802) {};
    \node[shape=circle,fill=black] (D1) at (0.782,0.623) {};
     \node[shape=circle,fill=black] (D2) at (1.564,1.246) {};
     \node[shape=circle,fill=black] (E1) at (-0.782,0.623) {};
     \node[shape=circle,fill=black] (E2) at (-1.564,1.246) {};
    \node[shape=circle,fill=black] (F1) at (0.975,-0.226) {};
    \node[shape=circle,fill=black] (F2) at (1.950,-0.452) {};
     \node[shape=circle,fill=black] (G1) at (-0.975,-0.226) {};
     \node[shape=circle,fill=black] (G2) at (-1.950,-0.452) {};

\draw plot [smooth, tension=2] coordinates { (A1) (A2) };
\draw plot [smooth, tension=2] coordinates { (B1) (B2) };
\draw plot [smooth, tension=2] coordinates { (C1) (C2) };
\draw plot [smooth, tension=2] coordinates { (D1) (D2) };
\draw plot [smooth, tension=2] coordinates { (E1) (E2) };
\draw plot [smooth, tension=2] coordinates { (F1) (F2) };
\draw plot [smooth, tension=2] coordinates { (G1) (G2) };

\draw plot [smooth, tension=2] coordinates { (A1) (F1) };
\draw plot [smooth, tension=2] coordinates { (F1) (C1) };
\draw plot [smooth, tension=2] coordinates { (C1) (E1) };
\draw plot [smooth, tension=2] coordinates { (E1) (D1) };
\draw plot [smooth, tension=2] coordinates { (D1) (B1) };
\draw plot [smooth, tension=2] coordinates { (B1) (G1) };
\draw plot [smooth, tension=2] coordinates { (G1) (A1) };

\draw plot [smooth, tension=2] coordinates { (A2) (D2) };
\draw plot [smooth, tension=2] coordinates { (D2) (F2) };
\draw plot [smooth, tension=2] coordinates { (F2) (B2) };
\draw plot [smooth, tension=2] coordinates { (B2) (C2) };
\draw plot [smooth, tension=2] coordinates { (C2) (G2) };
\draw plot [smooth, tension=2] coordinates { (G2) (E2) };
\draw plot [smooth, tension=2] coordinates { (E2) (A2) };

\end{tikzpicture}

\\
\\
(7,2)-Generalized Petersen $P_{7,2}$
  \end{tabular}
  \end{figure}
  \end{center}

\subsection{The Heawood graph}

 The {\em Heawood graph},  $H_{3,6}$, has $14$ vertices, is $3$-regular and vertex-transitive, has girth $6$, and is a $(3,6)$-Moore graph~(see Figure~1).  It can be constructed as the point-line incidence graph of the Fano plane.  Its independence polynomial is 
\[ P_{H_{3,6}}(\lam) = 1+14\lam + 70\lam^2 + 154\lam^3 +147 \lam^4 + 56 \lam^5 +14\lam^6 + 2\lam^7,\]
and its occupancy fraction is
\begin{equation}
\label{eq:HeaOcc}
\alpha_{H_{3,6}}(\lam) =  \frac{ \lam ( 1+ 10\lam +33 \lam^2 + 42 \lam^3 +20 \lam^4 + 6\lam^5 + \lam^6  ) }{ P_{H_{3,6}}(\lam)  }  . 
\end{equation}

Our second result provides a tight upper bound on the occupancy fraction of cubic graphs with girth at least $5$:
\begin{theorem}
\label{thm:d3g6}
For any cubic graph $G$ of girth at least $5$, and for every $\lam>0$,
\[ \alpha_G(\lam) \le \alpha_{H_{3,6}} (\lam),\]
with equality if and only if $G$ is a union of copies of $H_{3,6}$. 
\end{theorem}
And by integrating $\frac{\alpha_G(t)}{t}$ from $t=0$ to $\lam$ we obtain the corresponding counting results.
\begin{cor}
\label{cor:Hea}
For any cubic graph $G$ of girth at least $5$, and for every $\lam >0$,
\[ \frac{1}{|V(G)|} \log P_G(\lam) \le  \frac{1}{14} \log P_{H_{3,6}}(\lam), \] 
and in particular,
\[  \frac{1}{|V(G)|} \log | \mathcal I(G)| \le \frac{1}{14} \log | \mathcal I(H_{3,6}) | ,\]
with equality if and only if $G$ is a union of copies of $H_{3,6}$. 
\end{cor}

Note that Theorem~\ref{thm:d3g6} applies to all positive $\lam$, while Theorem~\ref{thm:d3g4} requires $\lam \in (0,1]$.  Some bound on the interval for which $P_{5,2}$ minimizes the occupancy fraction is necessary: for large $\lam$, $P_{7,2}$ has a smaller occupancy fraction, and in fact in the limit as $\lam \to \infty$, it is minimal:  Staton~\cite{staton1979some} proved  the independence ratio of any triangle-free cubic graph is at least $5/14$ and this is achieved by $P_{7,2}$.

Corollaries~\ref{cor:Peter} and \ref{cor:Hea} answer the $3$-regular case of a question of Zhao~\cite{yufeiSurvey} and confirm his conjecture in these cases that minimum and maximum normalized number of independent sets under local constraints are attained by finite graphs.

We prove Theorems~\ref{thm:d3g4} and \ref{thm:d3g6} by introducing several local constraints that the hard-core model induces on any $3$-regular graph of a given girth.  We then relax the optimization problem to the set of all probability distributions on local configurations that satisfy these constraints and solve the relaxation using linear programming.

We begin in Section~\ref{sec:related} with an overview of the method  that we use to prove Theorems~\ref{thm:d3g4} and \ref{thm:d3g6}. 

In Section~\ref{sec:girth6} we prove Theorem~\ref{thm:d3g6} under the additional assumption that $G$ has girth at least $6$.  This illustrates the method and introduces the type of local constraints we will use in the other proofs.  In Section~\ref{sec:g5} we extend the proof to all graphs of girth at least $5$ by introducing new variables into our maximization problem.  In Section~\ref{sec:g4} we prove Theorem~\ref{thm:d3g4} by switching from maximization to minimization and by introducing further variables corresponding to graphs containing $4$-cycles.  Nonetheless, the constraints remain unchanged from those in Sections~\ref{sec:girth6} and \ref{sec:g5}.

We conclude in Section~\ref{sec:extensions} with some further questions on the number of independent sets and graph homomorphisms in regular graphs with girth constraints.

\section{On the method and related work}
\label{sec:related}
The method we use is an extension of the  method used in~\cite{davies2015independent} to prove Theorem~\ref{thm:KddOcc}  and the analogous theorem for random matchings in regular graphs.  At a high level, the method works as follows.  To bound the occupancy fraction of the hard-core model on a graph $G$, we consider the experiment of drawing an independent set $I$ from the hard-core model, then independently choosing a vertex $v$ uniformly from the graph.  We then record a {\em local view } of both $G$ and $I$ from the perspective of $v$. The depth-$t$ local view from $v$ includes both the information of the graph structure of the depth-$t$ neighborhood of $v$ as well as the boundary conditions the independent set $I$ induces on this neighborhood (that is, which vertices at the boundary are blocked from being in the independent set by some external vertex).   In~\cite{davies2015independent}, the local view considered was of the first neighborhood of $v$, with each neighbor labeled according to whether or not it had an occupied neighbor among the second neighbors of $v$.  

In this paper, we extend the local view to include the first and second neighborhood of $v$.  We call a realization of the local view a {\em configuration}.  The probabilistic experiment of drawing $I$ and $v$ at random induces a probability distribution on the set of all possible configurations.  Not all probability distributions over the configuration set are attainable by graphs; certain consistency conditions must hold. For instance, here we use the fact that the probability that $v$ has $t$ occupied neighbors in this experiment must equal the probability that a random neighbor $u$ of $v$ has $t$ occupied neighbors.  Such consistency conditions serve as constraints in an optimization problem in which the variables are the individual probabilities of each possible configuration.

The art in applying the method is choosing the right local view and which consistency conditions to impose.  Enriching the local view as we have done here adds power to the optimization program, but comes at the cost of increasing the  complexity of the resulting linear program.  As an example, compare the upper bound on the independence polynomial in $d$-regular triangle-free graphs~\cite{davies2015independent} with the lower bound for $3$-regular triangle-free graphs given by Theorem~\ref{thm:d3g4}: the proof of the first is short and elementary, while the proof of the second requires (at least in this iteration) a large mass of calculations given in the appendix and in the ancillary files.  

This suggests several directions for further inquiry into this method.
\begin{enumerate}
\item Is there a general theory of which problems can be solved using this method and is there an underlying principle that indicates which distributions and graphs are extremal?
\item Can the proof procedure be efficiently automated, in a way that given the definition of the local view, the allowed configurations, and the objective function, a computer outputs a bound along with a proof certificate?
\item Is there a more analytic and less computational analysis of the linear programs than the proofs we give here?
\end{enumerate}

The method has to this point been used for upper and lower bounds on independent sets and matchings in regular graphs~\cite{davies2015independent, davies2016averagens, cutler2016minimum}, as well as the Widom-Rowlinson model~\cite{cohen2015widom}, another statistical physics model with hard constraints.  In~\cite{PottsExtremes}, the method is applied to models with soft constraints, namely the Ising and Potts models on regular graphs, which in the `zero-temperature limit' yield extremal bounds on the number of $q$-colorings of cubic graphs. See the survey of Zhao~\cite{yufeiSurvey} for more on extremal problems for regular graphs.

\section{Proof of Theorem \ref{thm:d3g6} for girth at least $6$}
\label{sec:girth6}
Let $G$ be a $3$-regular graph of girth at least $6$.  Since $G$ has girth greater than $5$, every vertex $v \in V(G)$ has $6$ distinct second neighbors, and its second neighbors form an independent set.

Draw an independent set $I\in \mathcal{I}(G)$ from the hard-core distribution on $G$ with fugacity $\lambda>0$. We say that a vertex is \emph{occupied} if it is in $I$.  
Pick a vertex $v$ uniformly at random from $V(G)$. We say that a vertex $u$ of the second neighborhood of $v$ is \emph{externally uncovered} if none of its neighbors at distance $3$ from $v$ are in $I$.
Order the neighbors of $v$ arbitrarily, $u_1, u_2,u_3$.
Describe the local view of $v$ with respect to $I$ by  $C= (c_1,c_2,c_3)$ with $c_i$ being the number of externally uncovered second neighbors joined to $u_i$.

Let 
\begin{align*}
\mathcal C_6 &= \{ (0,0,0), (0,0,1), (0,0,2), (0,1,1), (0,1,2), \\
 &(0,2,2), (1,1,1), (1,1,2), (1,2,2),(2,2,2) \}
\end{align*}
be the set of all possible configurations $C$ that can arise from cubic graphs of girth at least $6$.  As the functions that appear in the optimization problem below do not depend on the ordering of the $c_i$'s we can restrict ourselves to multisets; that is, the configuration $(1,1,2)$ is equivalent $(2,1,1)$.  We abuse notation and let $C$ refer to the vector $(c_1, c_2,c_3)$ as well as the graph formed by the configuration: $v$ joined to its three neighbors $u_1,u_2,u_3$, each joined to $c_1, c_2, c_3$ second neighbors of $v$ respectively.

Consider the following quantities,
\begin{align*}
Z_-(C) &= \prod_{i=1}^3 \left( \lam + (1+\lam)^{c_i} \right )  \\
Z_+(C) &= \lam  (1+\lam)^{\sum_{i=1}^3 c_i } \\
Z(C) &= Z_-(C) + Z_+(C) .
\end{align*}
Here, $Z_-(C)$ is the partition function of $C$ restricted to independent sets with $v$  unoccupied; $Z_+$ is the same restricted to $v$ occupied, and $Z(C)$ is the total partition function of $C$. We reserve the letter $P$ for a partition function of a full cubic graph, as in $P_{H_{3,6}}$, and use $Z$ for the partition function of configuration. Then the probability that $v$ is occupied given configuration $C$ is
\begin{align*}
\alpha_{C,v}(\lam)&:= \Pr[ v \in I | C] = \frac{Z_+(C)}{Z(C)}. 
\end{align*}

Observe that the occupancy fraction of $G$ can be written as
$$
\alpha_G(\lam) = \frac{1}{n}\sum_{v\in V(G)} \Pr[v\in I] = \sum_{C\in \mathcal C_6}\alpha_{C,v}(\lam) \Pr[C],
$$
where $\Pr[C]$ is the probability that the configuration $C$ is observed when $I\in \mathcal I (G)$ is chosen according to the hard-core model and $v$ is chosen uniformly at random. 

Given the choice of $v$, let $u$ be a vertex chosen uniformly at random from the neighbors of $v$. 
We can write down formulae for the conditional probabilities that either $v$ or $u$ has a given   the number of occupied neighbors.  Let 
\begin{align*}
\gamma_t^v(C) &= \Pr[v \text{ has } t \text{ occupied neighbors}|C]  \quad \text{ and }\\
\gamma_t^u(C) &= \Pr[u \text{ has } t \text{ occupied neighbors}|C].
\end{align*}
\begin{lemma}
\label{lem:g6functions}
\begin{align*}
\gamma^v_0(C) &=  \frac{1+\lam}{\lam} \alpha_{C,v}(\lam) \\
\gamma^v_1(C) &=  \alpha_{C,v}(\lam) \cdot \sum_{i=1}^3 (1+\lam)^{-c_i} \\
\gamma^v_2(C) &=   \frac{\lam^2}{Z(C)} \cdot \sum_{i=1}^3 (1+\lam)^{c_i} \\
\gamma^u_0(C) &=   (1- \alpha_{C,v}(\lam)) \frac{1+\lam}{3} \sum_{i=1}^3 \frac{1}{\lam + (1+\lam)^{c_i}}     \\
\gamma^u_1(C) &=    \alpha_{C,v}(\lam) \cdot \frac{1}{3} \sum_{i=1}^3 (1+\lam)^{-c_i} + (1- \alpha_{C,v}(\lam)) \cdot \frac{1}{3} \sum_{i=1}^3 \frac{c_i \lam}{\lam + (1+ \lam)^{c_i}}     \\
\gamma^u_2(C) &=  \alpha_{C,v}(\lam) \cdot \frac{1}{3} \sum_{i=1}^3 c_i \lam (1+\lam)^{-c_i} + (1- \alpha_{C,v}(\lam)) \cdot \frac{1}{3} \sum_{i\in\{1,2,3\}\atop c_i=2} \frac{\lam^2}{\lam + (1+ \lam)^{c_i}} .
\end{align*}
\end{lemma}
\begin{proof}
Here we show how to obtain the expressions for $\gamma^v_2(C)$ and $\gamma^u_2(C)$; the others  follow similarly. For $\gamma^v_2(C)$ we need to select a vertex $u_i$ in the neighborhood of $v$ that will not be included in $I$, and insist that the two other neighbors of $v$ are in $I$. Given this choice, we may additionally include any subset of the externally uncovered neighbors of $u_i$ (different from $v$), contributing the factor $(1+\lam)^{c_i}$. The term $\lam^2$ accounts for the contribution of $u_j$ with $j\neq i$. It follows that 
$$
\gamma^v_2(C) =   \frac{\lam^2}{Z(C)} \cdot \sum_{i=1}^3 (1+\lam)^{c_i}
$$
For $\gamma^u_2(C)$ we need to distinguish between the case where $v$ is occupied and the case where it is unoccupied. Let us assume first that $v$ is occupied, then it contributes with a factor of $\lam$, and accounts for $1$ of the two required occupied neighbors of $u$. Select a vertex $u=u_i$ in the neighborhood of $v$ uniformly at random. Since $v\in I$, we need to include exactly one of its externally uncovered neighbors, contributing a factor of $c_i\lam$. The neighbors $u_j$ with $j\neq i$ cannot be included in $I$ since $v$ is occupied, thus we can include any subset of their externally uncovered neighbors, contributing with a factor of $\frac{(1+\lam)^{c_1+c_2+c_3}}{(1+\lam)^{c_i}}= \frac{Z_+(C)(1+\lam)^{-c_i}}{\lam}$. The probability that $u$ has $2$ occupied neighbors and that $v$ is occupied is
$$
\frac{1}{Z(C)} \cdot \lam \cdot \frac{1}{3} \sum_{i=1}^3 c_i \lam \cdot  \frac{Z_+(C)(1+\lam)^{-c_i}}{\lam} =  \alpha_{C,v}(\lam) \cdot \frac{1}{3} \sum_{i=1}^3 c_i \lam (1+\lam)^{-c_i}\;.
$$
The probability that  $u$ has $2$ occupied neighbors and that $v$ is unoccupied can be computed in a similar way, giving that
$$
\gamma^u_2(C) =  \alpha_{C,v}(\lam) \cdot \frac{1}{3} \sum_{i=1}^3 c_i \lam (1+\lam)^{-c_i} + (1- \alpha_{C,v}(\lam)) \cdot \frac{1}{3} \sum_{i\in\{1,2,3\}\atop c_i=2} \frac{\lam^2}{\lam + (1+ \lam)^{c_i}} .
$$
\end{proof}

Now we can form the following linear program with decision variables $p(C)$ with $ C \in \mathcal C_6$ corresponding to $\Pr[C]$:
\begin{align*}
\alpha_{\mathrm{max}}(\lam) = \max \,    &\sum_{C \in \mathcal C_6} \alpha_{C,v}(\lam) p(C)    \text{ subject to } \\
&\sum_{C \in \mathcal C_6} p(C) = 1 \\
&\sum_{C \in \mathcal C_6} p(C) \cdot  ( \gamma^v_t(C) - \gamma^u_t(C)) =0 \, \, \,  \text{ for } t=0,1,2  \\
&p(C) \ge 0 \, \, \,  \forall C \in \mathcal C_6
\end{align*}

The dual program with decision variables $\Lam_p, \Lam_0, \Lam_1, \Lam_2$ is:

\begin{align*}
&\alpha_{\mathrm{max}}(\lam) = \min \,   \Lam_p   \text{ subject to  } \\
& \Lam_p  +\sum_{t=0}^2 \Lam_t \left [ \gamma^v_t(C) - \gamma^u_t(C)   \right ]   \ge \alpha_{C,v}(\lam) \,\,\,\, \forall C \in \mathcal C_6.
\end{align*}

To show that $\alpha_G(\lam) \le \alpha_{H_{3,6}} (\lam)$ for all cubic $G$ of girth at least $6$, we need to show that $\alpha_{\mathrm{max}}(\lam) \le  \alpha_{H_{3,6}}(\lam)$ (the reverse inequality is immediate since the distribution induced by $H_{3,6}$ is a feasible solution).  To prove this using linear programming duality, it is enough to find a feasible solution to the dual program with $\Lam_p =  \alpha_{H_{3,6}}$.  We define the {\em slack function} of a configuration $C$ as:
\begin{align}\label{eq:slack1}
\text{SLACK}_{\mathrm{max}}( \lam,\Lam_0, \Lam_1, \Lam_2, C) &=  \alpha_{H_{3,6}} -\alpha_{C,v}(\lam) +\sum_{t=0}^2 \Lam_t \left [ \gamma^v_t(C) - \gamma^u_t(C)   \right ].
\end{align}

Our goal is now to find values for the dual variables $\Lam_0^*, \Lam_1^*, \Lam_2^*$ so that 
\begin{align*}
\text{SLACK}_{\mathrm{max}}(\lam, \Lam_0^*, \Lam_1^*, \Lam_2^*,C) \ge 0
\end{align*}
for all configurations $C\in \mathcal C_6$ and $\lam>0$.

Our candidate solution is the Heawood graph $H_{3,6}$ (see Figure~2). There are only $4$ possible configurations arising from $H_{3,6}$: $(0,0,0), (1,0,0), (1,1,1)$, and  $(2,2,2)$.  These correspond respectively to having $3$ or $4$, $2$, $1$ and $0$ vertices from $I$ in the third neighborhood of $v$.

\begin{center}
\begin{figure}[h!]\label{fig:Hea} 
\caption{}
\begin{tabular}{ c}
\begin{tikzpicture}
    \node[shape=circle,fill=black] (A1) at (0,0) {};
     \node[shape=circle,fill=black] (B1) at (-2,1.5) {};
     \node[shape=circle,fill=black] (B2) at (0,1.5) {};
     \node[shape=circle,fill=black] (B3) at (2,1.5) {};
    \node[shape=circle,fill=black] (C1) at (-2.5,3) {};
     \node[shape=circle,fill=black] (C2) at (-1.5,3) {};
     \node[shape=circle,fill=black] (C3) at (-0.5,3) {};
     \node[shape=circle,fill=black] (C4) at (0.5,3) {};
     \node[shape=circle,fill=black] (C5) at (1.5,3) {};
    \node[shape=circle,fill=black] (C6) at (2.5,3) {};
    
    \node[shape=circle,fill=black] (D1) at (-2.1,4.5) {};
     \node[shape=circle,fill=black] (D2) at (-0.7,4.5) {};
     \node[shape=circle,fill=black] (D3) at (0.7,4.5) {};
     \node[shape=circle,fill=black] (D4) at (2.1,4.5) {};
     
     \draw plot [smooth, tension=2] coordinates { (A1) (B1) };
     \draw plot [smooth, tension=2] coordinates { (A1) (B2) };
     \draw plot [smooth, tension=2] coordinates { (A1) (B3) };
     
     \draw plot [smooth, tension=2] coordinates { (B1) (C1) };
     \draw plot [smooth, tension=2] coordinates { (B1) (C2) };
     \draw plot [smooth, tension=2] coordinates { (B2) (C3) };
     \draw plot [smooth, tension=2] coordinates { (B2) (C4) };
     \draw plot [smooth, tension=2] coordinates { (B3) (C5) };
     \draw plot [smooth, tension=2] coordinates { (B3) (C6) };
     
     \draw plot [smooth, tension=2] coordinates { (C1) (D1) };
     \draw plot [smooth, tension=2] coordinates { (C1) (D2) };
     \draw plot [smooth, tension=2] coordinates { (C2) (D3) };
     \draw plot [smooth, tension=2] coordinates { (C2) (D4) };
     
     \draw plot [smooth, tension=2] coordinates { (C3) (D1) };
     \draw plot [smooth, tension=2] coordinates { (C3) (D3) };
     \draw plot [smooth, tension=2] coordinates { (C4) (D2) };
     \draw plot [smooth, tension=2] coordinates { (C4) (D4) };
     
     \draw plot [smooth, tension=2] coordinates { (C5) (D1) };
     \draw plot [smooth, tension=2] coordinates { (C5) (D4) };
     \draw plot [smooth, tension=2] coordinates { (C6) (D2) };
     \draw plot [smooth, tension=2] coordinates { (C6) (D3) };
\end{tikzpicture}
\\
\\
Heawood Graph $H_{3,6}$ viewed from a vertex
  \end{tabular}
  \end{figure}
  \end{center}

If we set the dual constraints to hold with equality for the configurations $(0,0,0), (1,0,0)$, and $(1,1,1)$,  and set $\Lam_p = \alpha_{H_{3,6}}$, we get the following system of equations.
{\smaller
\begin{align*}
&\frac{\lam}{\lam + (1+\lam)^3}= \alpha_{H_{3,6}} + \Lam_0 \left [ \frac{1+ 2\lam}{\lam + (1+\lam)^3} -1  \right ] + \Lam_1 \cdot \frac{2 \lambda }{(\lambda +1)^3+\lambda }  +\Lam_2 \frac{3 \lam^2}{\lam + (1+\lam)^3}      \\
&\frac{\lambda }{2 \lambda ^2+4 \lambda +1}= \alpha_{H_{3,6}}- \Lam_0 \frac{5 \lambda  (\lambda +1)}{6 \lambda ^2+12 \lambda +3}   - \Lam_1  \frac{\lambda  \left(\lambda ^2-2 \lambda -5\right)}{3 (\lambda +1) \left(2 \lambda ^2+4 \lambda +1\right)} 
 + \Lam_2 \frac{\lambda ^2 (3 \lambda +8)}{3 \left(2 \lambda ^3+6 \lambda ^2+5 \lambda +1\right)}   \\
&\frac{\lambda  (\lambda +1)^3}{\lambda  (\lambda +1)^3+(2 \lambda +1)^3}=  \\
&\alpha_{H_{3,6}}+  \Lam_0 \frac{\lambda  \left(\lambda ^3-2 \lambda -1\right)}{\lambda ^4+11 \lambda ^3+15 \lambda ^2+7 \lambda +1} + \Lam_1  \frac{\lambda -2 \lambda ^3}{\lambda ^4+11 \lambda ^3+15 \lambda ^2+7 \lambda +1}   + \Lam_2 \frac{\lambda ^2 \left(-\lambda ^2+\lambda +2\right)}{\lambda ^4+11 \lambda ^3+15 \lambda ^2+7 \lambda +1}  
\end{align*}
}

Solving these equations give candidate values for the dual variables.
\begin{claim}
\label{claim:dualFeas}
With the following assignments to the dual variables, 
\[ \text{SLACK}_{\mathrm{max}}( \lam,\Lam_0^*, \Lam_1^*, \Lam_2^*, C) \ge 0 \]
 for all configurations $C\in \mathcal C_6$. 
\begin{align*}
\Lam_0^* & = \frac{- 3 -27 \lam - 94 \lam^2  -139 \lam^3 -20 \lam^4 +139 \lam^5 +124 \lam^6 + 45 \lam^7 + 9 \lam^8 + \lam^9 }{(1+\lam) (1+2\lam) \cdot P_{H_{3,6}}(\lambda)} \\
\Lam_1^* &= \frac{- 3 -24 \lam - 73 \lam^2  -99 \lam^3 -25 \lam^4 +63 \lam^5 +55 \lam^6 +15 \lam^7 + \lam^8 }{(1+2\lam) \cdot P_{H_{3,6}}(\lambda)} \\
\Lam_2^* &=  \frac{- 3 -27 \lam - 94 \lam^2  -160 \lam^3 -132 \lam^4 -46 \lam^5 -3 \lam^6 + \lam^7 }{(1+2\lam) \cdot P_{H_{3,6}}(\lambda)}  
\end{align*}
\end{claim}
In particular, Claim~\ref{claim:dualFeas} shows that in the primal $\alpha_{\mathrm{max}}(\lam) = \alpha_{H_{3,6}}(\lam)$.  To prove Claim~\ref{claim:dualFeas} we will show that for all $C \in \mathcal C_6$, the following scaling of the slack function
\begin{align}
\label{eq:g5Ffunc}
F_{\mathrm{max}}(C):= 3 (\lam +2) \cdot P_{H_{3,6}}(\lam) \cdot Z(C) \cdot \text{SLACK}_{\mathrm{max}}( \lam,\Lam_0^*, \Lam_1^*, \Lam_2^*,C) ,
\end{align}
is identically $0$ if $C$ is in the support of the Heawood graph and a polynomial in $\lam$ with positive coefficients otherwise. This suffices to prove  Claim~\ref{claim:dualFeas} since $ 3 (\lam +2) \cdot P_{H_{3,6}}(\lam) \cdot Z(C) $ is itself a polynomial in $\lam$ with positive coefficients. 

 Using \eqref{eq:HeaOcc} and Lemma~\ref{lem:g6functions}, we calculate:
\begin{align*}
F_{\mathrm{max}}((0,0,0))& = 0   \\
F_{\mathrm{max}}((0,0,1))& =  0 \\
F_{\mathrm{max}}((0,0,2))& =  \lambda ^5 (\lambda +1) (\lambda +2) \left(\lambda ^5+13 \lambda ^4+47 \lambda ^3+69 \lambda ^2+36 \lambda +6\right) \\
F_{\mathrm{max}}((0,1,1))& = \lambda ^6 \left(2 \lambda ^5+11 \lambda ^4+30 \lambda ^3+41 \lambda ^2+20 \lambda +3\right)   \\
F_{\mathrm{max}}((0,1,2))& = \lambda ^5 \left(4 \lambda ^7+47 \lambda ^6+209 \lambda ^5+458 \lambda ^4+523 \lambda ^3+303 \lambda ^2+84 \lambda +9\right)  \\
F_{\mathrm{max}}((0,2,2))& = \lambda ^5 (\lambda +2) \left(4 \lambda ^7+49 \lambda ^6+218 \lambda ^5+463 \lambda ^4+502 \lambda ^3+269 \lambda ^2+66 \lambda +6\right)  \\
F_{\mathrm{max}}((1,1,1))& =  0  \\
F_{\mathrm{max}}((1,1,2))& = \lambda ^5 (\lambda +1)^3 \left(3 \lambda ^4+37 \lambda ^3+77 \lambda ^2+39 \lambda +6\right)  \\
F_{\mathrm{max}}((1,2,2))& = 2 \lambda ^5 (\lambda +1)^3 \left(\lambda ^5+15 \lambda ^4+52 \lambda ^3+62 \lambda ^2+24 \lambda +3\right) \\
F_{\mathrm{max}}((2,2,2))& = 0.  
\end{align*}
Indeed for all $C$ in the support of the Heawood graph $F_{\mathrm{max}}(C)=0$, and for all other $C$, $F_{\mathrm{max}}(C)$ is a polynomial in $\lam$ with positive coefficients. This proves Claim~\ref{claim:dualFeas} and thus shows that $\alpha_G(\lam) \le \alpha_{H_{3,6}}(\lam)$ for all $\lam >0$ and all cubic $G$ of girth at least $6$.  Uniqueness follows from complementary slackness and the fact that we have $4$ linearly independent constraints; therefore the only feasible distribution whose support is contained in the support of the Heawood graph is the Heawood graph.

\section{Girth at least $5$}
\label{sec:g5}
Now we extend the proof to include graphs of girth $5$.  
If $G$ is cubic and has girth at least than $5$, then every vertex $v \in V(G)$ has $6$ distinct second neighbors, but now its second neighborhood  may contain some edges.

Draw an independent set $I\in \mathcal{I}(G)$ from the hard-core distribution on $G$ with fugacity $\lambda>0$. Recall that a vertex $u$ of the second neighborhood of $v$ is externally uncovered if none of its neighbors at distance $3$ from $v$ are in $I$. For $i\in\{1,2,3\}$, let $u_i$ be the neighbors of $v$ and for $j\in\{1,2\}$ let $w_{ij}$ be the neighbors of $u_i$ that are second neighbors of $v$.  Let $C=(W, E_{12}, E_{22})$ be a configuration where $W$ is the set of second neighbors of $v$ that are externally uncovered, $E_{12}$ is the set of edges between first neighbors of $v$ and $W$, and $E_{22}$ is the set of edges within $W$ in the configuration $C$.

Let $\mathcal C_5$ be the set of all possible configurations $C$ that can arise from a cubic graph of girth at least $5$. The possible configurations are the following:

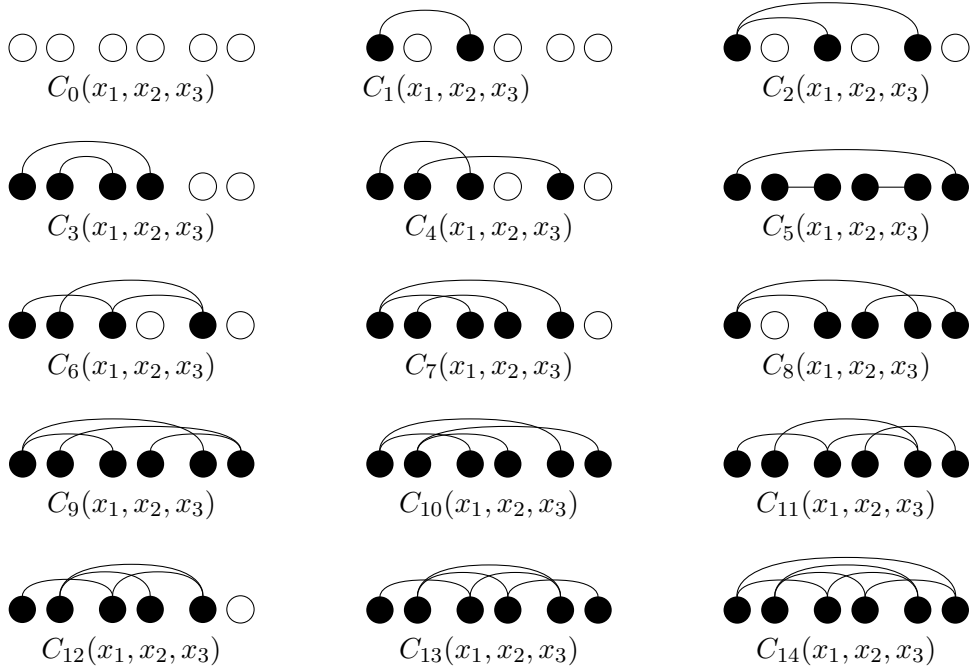
\begin{center}
\begin{tabular}{ c c c }
  \begin{tikzpicture}
    \node[shape=circle,draw=black] (A) at (0,0) {};
    \node[shape=circle,draw=black] (B) at (0.5,0) {};
    \node[shape=circle,draw=black] (C) at (1.2,0) {};
    \node[shape=circle,draw=black] (D) at (1.7,0) {};
    \node[shape=circle,draw=black] (E) at (2.4,0) {};
    \node[shape=circle,draw=black] (F) at (2.9,0) {} ;
\end{tikzpicture}
  &
  \hspace{1cm}
\begin{tikzpicture}
    \node[shape=circle,fill=black] (A) at (0,0) {};
    \node[shape=circle,draw=black] (B) at (0.5,0) {};
    \node[shape=circle,fill=black] (C) at (1.2,0) {};
    \node[shape=circle,draw=black] (D) at (1.7,0) {};
    \node[shape=circle,draw=black] (E) at (2.4,0) {};
    \node[shape=circle,draw=black] (F) at (2.9,0) {} ;

\draw plot [smooth, tension=2] coordinates { (0,0.15) (0.6,0.5) (1.2,0.15)};
\end{tikzpicture}
  & 
\hspace{1cm}
\begin{tikzpicture}
    \node[shape=circle,fill=black] (A) at (0,0) {};
    \node[shape=circle,draw=black] (B) at (0.5,0) {};
    \node[shape=circle,fill=black] (C) at (1.2,0) {};
    \node[shape=circle,draw=black] (D) at (1.7,0) {};
    \node[shape=circle,fill=black] (E) at (2.4,0) {};
    \node[shape=circle,draw=black] (F) at (2.9,0) {} ;

\draw plot [smooth, tension=2] coordinates { (0,0.15) (0.6,0.4) (1.2,0.15)};
\draw plot [smooth, tension=2] coordinates { (0,0.15) (1.2,0.6) (2.4,0.15)};
\end{tikzpicture}  
 \\
$C_0(x_1,x_2,x_3)$ & $C_1(x_1,x_2,x_3)$ & \hspace{1cm} $C_2(x_1,x_2,x_3)$\\
 \\
  \begin{tikzpicture}
    \node[shape=circle,fill=black] (A) at (0,0) {};
    \node[shape=circle,fill=black] (B) at (0.5,0) {};
    \node[shape=circle,fill=black] (C) at (1.2,0) {};
    \node[shape=circle,fill=black] (D) at (1.7,0) {};
    \node[shape=circle,draw=black] (E) at (2.4,0) {};
    \node[shape=circle,draw=black] (F) at (2.9,0) {} ;
    
\draw plot [smooth, tension=2] coordinates { (0,0.15) (0.85 ,0.6) (1.7,0.15)};
\draw plot [smooth, tension=2] coordinates { (0.5,0.15) (0.85,0.4) (1.2,0.15)};
\end{tikzpicture}
  &
  \hspace{1cm}
\begin{tikzpicture}
    \node[shape=circle,fill=black] (A) at (0,0) {};
    \node[shape=circle,fill=black] (B) at (0.5,0) {};
    \node[shape=circle,fill=black] (C) at (1.2,0) {};
    \node[shape=circle,draw=black] (D) at (1.7,0) {};
    \node[shape=circle,fill=black] (E) at (2.4,0) {};
    \node[shape=circle,draw=black] (F) at (2.9,0) {} ;

\draw plot [smooth, tension=2] coordinates { (0,0.15) (0.6 ,0.6) (1.2,0.15)};
\draw plot [smooth, tension=2] coordinates { (0.5,0.15) (1.45,0.4) (2.4,0.15)};
\end{tikzpicture}
  & 
\hspace{1cm}
\begin{tikzpicture}
    \node[shape=circle,fill=black] (A) at (0,0) {};
    \node[shape=circle,fill=black] (B) at (0.5,0) {};
    \node[shape=circle,fill=black] (C) at (1.2,0) {};
    \node[shape=circle,fill=black] (D) at (1.7,0) {};
    \node[shape=circle,fill=black] (E) at (2.4,0) {};
    \node[shape=circle,fill=black] (F) at (2.9,0) {} ;

\draw plot [smooth, tension=2] coordinates { (0.5,0) (1.2,0)};
\draw plot [smooth, tension=2] coordinates { (1.7,0) (2.4,0)};
\draw plot [smooth, tension=2] coordinates { (0,0.15) (1.45,0.5)  (2.9,0.15)};
\end{tikzpicture}  
 \\ 
 $C_3(x_1,x_2,x_3)$ & \hspace{1cm} $C_4(x_1,x_2,x_3)$ & \hspace{1cm} $C_5(x_1,x_2,x_3)$\\
 \\
  \begin{tikzpicture}
    \node[shape=circle,fill=black] (A) at (0,0) {};
    \node[shape=circle,fill=black] (B) at (0.5,0) {};
    \node[shape=circle,fill=black] (C) at (1.2,0) {};
    \node[shape=circle,draw=black] (D) at (1.7,0) {};
    \node[shape=circle,fill=black] (E) at (2.4,0) {};
    \node[shape=circle,draw=black] (F) at (2.9,0) {} ;
    
\draw plot [smooth, tension=2] coordinates { (0,0.15) (0.6 ,0.4) (1.2,0.15)};
\draw plot [smooth, tension=2] coordinates { (0.5,0.15) (1.45,0.6) (2.4,0.15)};
\draw plot [smooth, tension=2] coordinates { (1.2,0.15) (1.8,0.4) (2.4,0.15)};

\end{tikzpicture}
  &
  \hspace{1cm}
\begin{tikzpicture}
    \node[shape=circle,fill=black] (A) at (0,0) {};
    \node[shape=circle,fill=black] (B) at (0.5,0) {};
    \node[shape=circle,fill=black] (C) at (1.2,0) {};
    \node[shape=circle,fill=black] (D) at (1.7,0) {};
    \node[shape=circle,fill=black] (E) at (2.4,0) {};
    \node[shape=circle,draw=black] (F) at (2.9,0) {} ;

\draw plot [smooth, tension=2] coordinates { (0,0.15) (0.6 ,0.4) (1.2,0.15)};
\draw plot [smooth, tension=2] coordinates { (0.5,0.15) (1.1,0.4) (1.7,0.15)};
\draw plot [smooth, tension=2] coordinates { (0,0.15) (1.2,0.6) (2.4,0.15)};
\end{tikzpicture}
  & 
\hspace{1cm}
\begin{tikzpicture}
    \node[shape=circle,fill=black] (A) at (0,0) {};
    \node[shape=circle,draw=black] (B) at (0.5,0) {};
    \node[shape=circle,fill=black] (C) at (1.2,0) {};
    \node[shape=circle,fill=black] (D) at (1.7,0) {};
    \node[shape=circle,fill=black] (E) at (2.4,0) {};
    \node[shape=circle,fill=black] (F) at (2.9,0) {} ;

\draw plot [smooth, tension=2] coordinates { (0,0.15) (0.6 ,0.4) (1.2,0.15)};
\draw plot [smooth, tension=2] coordinates { (1.7,0.15) (2.3,0.4) (2.9,0.15)};
\draw plot [smooth, tension=2] coordinates { (0,0.15) (1.2,0.6) (2.4,0.15)};
\end{tikzpicture}  
  \\ 
$C_6(x_1,x_2,x_3)$ & \hspace{1cm} $C_7(x_1,x_2,x_3)$ & \hspace{1cm} $C_8(x_1,x_2,x_3)$\\
 \\
  \begin{tikzpicture}
    \node[shape=circle,fill=black] (A) at (0,0) {};
    \node[shape=circle,fill=black] (B) at (0.5,0) {};
    \node[shape=circle,fill=black] (C) at (1.2,0) {};
    \node[shape=circle,fill=black] (D) at (1.7,0) {};
    \node[shape=circle,fill=black] (E) at (2.4,0) {};
    \node[shape=circle,fill=black] (F) at (2.9,0) {} ;
    
\draw plot [smooth, tension=2] coordinates { (0,0.15) (0.6 ,0.4) (1.2,0.15)};
\draw plot [smooth, tension=2] coordinates { (0,0.15) (1.2 ,0.6) (2.4,0.15)};
\draw plot [smooth, tension=2] coordinates { (0.5,0.15) (1.7,0.5) (2.9,0.15)};
\draw plot [smooth, tension=2] coordinates { (1.7,0.15) (2.3,0.4) (2.9,0.15)};

\end{tikzpicture}
  &
  \hspace{1cm}
\begin{tikzpicture}
    \node[shape=circle,fill=black] (A) at (0,0) {};
    \node[shape=circle,fill=black] (B) at (0.5,0) {};
    \node[shape=circle,fill=black] (C) at (1.2,0) {};
    \node[shape=circle,fill=black] (D) at (1.7,0) {};
    \node[shape=circle,fill=black] (E) at (2.4,0) {};
    \node[shape=circle,fill=black] (F) at (2.9,0) {} ;

\draw plot [smooth, tension=2] coordinates { (0,0.15) (0.6 ,0.4) (1.2,0.15)};
\draw plot [smooth, tension=2] coordinates { (0,0.15) (1.2 ,0.6) (2.4,0.15)};
\draw plot [smooth, tension=2] coordinates { (0.5,0.15) (1.7,0.5) (2.9,0.15)};
\draw plot [smooth, tension=2] coordinates { (0.5,0.15) (1.1,0.4) (1.7,0.15)};
\end{tikzpicture}
  & 
\hspace{1cm}
\begin{tikzpicture}
    \node[shape=circle,fill=black] (A) at (0,0) {};
    \node[shape=circle,fill=black] (B) at (0.5,0) {};
    \node[shape=circle,fill=black] (C) at (1.2,0) {};
    \node[shape=circle,fill=black] (D) at (1.7,0) {};
    \node[shape=circle,fill=black] (E) at (2.4,0) {};
    \node[shape=circle,fill=black] (F) at (2.9,0) {} ;

\draw plot [smooth, tension=2] coordinates { (0,0.15) (0.6 ,0.4) (1.2,0.15)};
\draw plot [smooth, tension=2] coordinates { (0.5,0.15) (1.45 ,0.6) (2.4,0.15)};
\draw plot [smooth, tension=2] coordinates { (1.2,0.15) (1.8,0.4) (2.4,0.15)};
\draw plot [smooth, tension=2] coordinates { (1.7,0.15) (2.3,0.5) (2.9,0.15)};
\end{tikzpicture}  
 \\ 
$C_9(x_1,x_2,x_3)$ & \hspace{1cm} $C_{10}(x_1,x_2,x_3)$ & \hspace{1cm} $C_{11}(x_1,x_2,x_3)$\\
 \\
  \begin{tikzpicture}
    \node[shape=circle,fill=black] (A) at (0,0) {};
    \node[shape=circle,fill=black] (B) at (0.5,0) {};
    \node[shape=circle,fill=black] (C) at (1.2,0) {};
    \node[shape=circle,fill=black] (D) at (1.7,0) {};
    \node[shape=circle,fill=black] (E) at (2.4,0) {};
    \node[shape=circle,draw=black] (F) at (2.9,0) {} ;
    
\draw plot [smooth, tension=2] coordinates { (0,0.15) (0.6 ,0.4) (1.2,0.15)};
\draw plot [smooth, tension=2] coordinates { (0.5,0.15) (1.1 ,0.5) (1.7,0.15)};
\draw plot [smooth, tension=2] coordinates { (0.5,0.15) (1.45 ,0.6) (2.4,0.15)};
\draw plot [smooth, tension=2] coordinates { (1.2,0.15) (1.8,0.5) (2.4,0.15)};
\end{tikzpicture}
  &
  \hspace{1cm}
\begin{tikzpicture}
    \node[shape=circle,fill=black] (A) at (0,0) {};
    \node[shape=circle,fill=black] (B) at (0.5,0) {};
    \node[shape=circle,fill=black] (C) at (1.2,0) {};
    \node[shape=circle,fill=black] (D) at (1.7,0) {};
    \node[shape=circle,fill=black] (E) at (2.4,0) {};
    \node[shape=circle,fill=black] (F) at (2.9,0) {} ;

\draw plot [smooth, tension=2] coordinates { (0,0.15) (0.6 ,0.4) (1.2,0.15)};
\draw plot [smooth, tension=2] coordinates { (0.5,0.15) (1.1 ,0.5) (1.7,0.15)};
\draw plot [smooth, tension=2] coordinates { (0.5,0.15) (1.45 ,0.6) (2.4,0.15)};
\draw plot [smooth, tension=2] coordinates { (1.2,0.15) (1.8,0.5) (2.4,0.15)};
\draw plot [smooth, tension=2] coordinates { (1.7,0.15) (2.3 ,0.4) (2.9,0.15)};
\end{tikzpicture}
  & 
\hspace{1cm}
\begin{tikzpicture}
    \node[shape=circle,fill=black] (A) at (0,0) {};
    \node[shape=circle,fill=black] (B) at (0.5,0) {};
    \node[shape=circle,fill=black] (C) at (1.2,0) {};
    \node[shape=circle,fill=black] (D) at (1.7,0) {};
    \node[shape=circle,fill=black] (E) at (2.4,0) {};
    \node[shape=circle,fill=black] (F) at (2.9,0) {} ;

\draw plot [smooth, tension=2] coordinates { (0,0.15) (0.6 ,0.4) (1.2,0.15)};
\draw plot [smooth, tension=2] coordinates { (0.5,0.15) (1.1 ,0.5) (1.7,0.15)};
\draw plot [smooth, tension=2] coordinates { (0.5,0.15) (1.45 ,0.6) (2.4,0.15)};
\draw plot [smooth, tension=2] coordinates { (1.2,0.15) (1.8,0.5) (2.4,0.15)};
\draw plot [smooth, tension=2] coordinates { (1.7,0.15) (2.3 ,0.4) (2.9,0.15)};
\draw plot [smooth, tension=2] coordinates { (0,0.15) (1.45 ,0.7) (2.9,0.15)};

\end{tikzpicture}  
 \\ 
 $C_{12}(x_1,x_2,x_3)$ & \hspace{1cm} $C_{13}(x_1,x_2,x_3)$ & \hspace{1cm} $C_{14}(x_1,x_2,x_3)$\\
\end{tabular}
\end{center}

\noindent where $x_i\in\{0,1,2\}$ indicates the number of externally uncovered neighbors of $u_i$ that are second neighbors of $v$ and that are included in the configuration as isolated vertices  in the graph induced by the second neighborhood. The \emph{type} of a configuration is the configuration up to the number of free second neighbors of $v$ that are isolated in the graph induced by the second neighborhood. For instance, the type of $C_1(1,1,1)$ is $C_1$. In Appendix~\ref{app:conf} we give a proof that, up to symmetries, these $15$ configuration types are all the possible ones for cubic graphs with girth at least $5$.

Up to relabeling of the first and second neighbors of $v$, these $14$ types give rise to $46$ configurations to consider. 
For instance, the type $C_0$ gives rise to $10$ configurations: $C_0(0,0,0)$, $C_0(1,0,0)$, $C_0(2,0,0)$ , $C_0(1,1,0)$, $C_0(2,1,0)$, $C_0(2,2,0)$, $C_0(1,1,1)$, $C_0(2,1,1)$, $C_0(2,2,1)$ and $C_0(2,2,2)$ and they correspond exactly to the set configurations $\mathcal C_6$ considered above. As another example, there are $3$ configurations corresponding to the type $C_6$: $C_6(0,0,0)$, $C_6(0,1,0)$ and $C_6(0,1,1)$. The configurations $C_6(0,1,0)$ and $C_6(0,0,1)$ are equivalent with respect to the functions in the optimization problem and so we include just one in our set $\mathcal C_5$.

We defer the formulae for $\alpha_{C,v}(\lam), \gamma^v_t(C), \gamma^u_t(C)$ until Section~\ref{sec:computFunc} where we give the formulae in a more general case.

We define the same linear program, dual program, and slack function as in Section~\ref{sec:girth6}. The optimality of $\alpha_{H_{3,6}}$ now follows from the following claim:
\begin{claim}
\label{claim:g5ub}
With the assignments to the dual variables given as in Claim~\ref{claim:dualFeas}, \\ $\text{SLACK}_{\mathrm{max}}(\lam, \Lam_0^*, \Lam_1^*, \Lam_2^*,C) \ge 0$ for all configurations $C\in \mathcal C_5$.
\end{claim}

Again we  scale the slack function by the same positive polynomial used in the previous section:
\begin{align}\label{eq:F}
F_{\mathrm{max}}(C):= 3 (\lam +2) \cdot P_{H_{3,6}}(\lam) \cdot Z(C) \cdot \text{SLACK}_{\mathrm{max}}(\lam, \Lam_0^*, \Lam_1^*, \Lam_2^*,C) ,
\end{align}

One can verify that $F_{\mathrm{max}}(C)$ is either identically $0$ or a polynomial in $\lam$ with positive coefficients. We have verified this by having a computer program compute $F_{\mathrm{max}}(C)$ for all $C \in \mathcal C_5$ and collect coefficients. The computer code and printout is included as an ancillary file.

For all $C\in \mathcal C_5$ in the support of the Heawood graph $F_{\mathrm{max}}(C)\equiv0$. For all other $C\in \mathcal C_5$ different than $C_1(1,1,0)$, $F_{\mathrm{max}}(C)$ is a polynomial in $\lam$ with positive coefficients. However, we also find that $F_{\mathrm{max}}(C_1(1,1,0))\equiv0$. This proves that the dual solution is feasible but does not prove uniqueness of the solution.

In order to prove that unions of copies of $H_{3,6}$ are the only graphs that maximize $\alpha_G(\lam)$ among all $3$-regular graphs $G$ of girth at least $5$, we first need to exclude the configuration $C_1(1,1,0)$. Let $C$ be the random configuration obtained by choosing $I\in \mathcal{I}(G)$ according to the hard-core model and $v\in V(G)$ uniformly at random.
\begin{claim}
 Let $G$ be a $3$-regular graph of girth at least $5$ with $\alpha_G(\lam) = \alpha_{H_{3,6}}(\lam)$. 
 Then $\Pr[C=C_1(1,1,0)]=0$.
\end{claim}
\begin{proof}
 Suppose that $\Pr[C=C_1(1,1,0)]>0$. Let $v^*\in V(G)$ be such that the second neighborhood of $v^*$ in $G$ contains at least one edge. Since $G$ attains the maximum occupancy fraction, complementary slackness tell us that the only configurations that can appear with positive probability are $C_0(0,0,0)$, $C_0(1,0,0)$, $C_0(1,1,1)$, $C_0(2,2,2)$ and $C_1(1,1,0)$. Consider the empty independent set $I_0=\emptyset$. The configuration induced by $I_0$ in the second neighborhood of $v^*$ is of the form $C_i(2,2,2)$ for some $i\neq 0$, but all these configurations have probability $0$ to appear, leading to a contradiction. 
\end{proof}

Therefore, any maximizer has support in $C_0(0,0,0)$, $C_0(1,0,0)$, $C_0(1,1,1)$ and $C_0(2,2,2)$. It suffices to prove that $H_{3,6}$ is the only graph with this support. Fix $v\in V(G)$. First observe that there are no edges within the second neighborhood of $v$. The fact that $\Pr[C=C_0(2,2,1)]=\Pr[C=C_0(2,1,1)]=\Pr[C=C_0(2,2,0)]=0$, implies that every vertex in the third neighborhood of $v$ is adjacent to $3$ vertices in the second neighborhood of $v$. Thus, $G$ is the disjoint union of $3$-regular graphs of girth $5$ and order $14$. Uniqueness now follows from the well-known fact that $H_{3,6}$ is the only such graph (see e.g.~\cite{wong1982cages}).

\section{Proof of Theorem~\ref{thm:d3g4}}
\label{sec:g4}

Here we prove Theorem~\ref{thm:d3g4} by showing that an appropriate linear programming relaxation shows that for all $3$-regular $G$ of girth at least $4$ (triangle-free) and all $\lam \in (0,1]$, 
\[ \alpha_G(\lam) \ge \alpha_{P_{5,2}}(\lam) .\]

We remark that the statement of Theorem~\ref{thm:d3g4} may still be true for some $\lam>1$, but it is not true for $\lam > \lam_3^* \approx 1.84593$, since for $\lam > \lam_3^*$ the occupancy fraction of the $(7,2)$-Generalized Petersen graph is smaller than that of the Petersen graph. In Section~\ref{sec:extensions} we present a conjecture that extends Theorem~\ref{thm:d3g4} for every $\lambda>0$.

\begin{center}

\begin{figure}[h!]\label{fig:Pet} 
\caption{}
\begin{tabular}{ c}
\begin{tikzpicture}
    \node[shape=circle,fill=black] (A1) at (0,0) {};
     \node[shape=circle,fill=black] (B1) at (-2,1.5) {};
     \node[shape=circle,fill=black] (B2) at (0,1.5) {};
     \node[shape=circle,fill=black] (B3) at (2,1.5) {};
    \node[shape=circle,fill=black] (C1) at (-2.5,3) {};
     \node[shape=circle,fill=black] (C2) at (-1.5,3) {};
     \node[shape=circle,fill=black] (C3) at (-0.5,3) {};
     \node[shape=circle,fill=black] (C4) at (0.5,3) {};
     \node[shape=circle,fill=black] (C5) at (1.5,3) {};
    \node[shape=circle,fill=black] (C6) at (2.5,3) {};
     
     \draw plot [smooth, tension=2] coordinates { (A1) (B1) };
     \draw plot [smooth, tension=2] coordinates { (A1) (B2) };
     \draw plot [smooth, tension=2] coordinates { (A1) (B3) };
     
     \draw plot [smooth, tension=2] coordinates { (B1) (C1) };
     \draw plot [smooth, tension=2] coordinates { (B1) (C2) };
     \draw plot [smooth, tension=2] coordinates { (B2) (C3) };
     \draw plot [smooth, tension=2] coordinates { (B2) (C4) };
     \draw plot [smooth, tension=2] coordinates { (B3) (C5) };
     \draw plot [smooth, tension=2] coordinates { (B3) (C6) };

     \draw plot [smooth, tension=2] coordinates { (C2) (C3) };
     \draw plot [smooth, tension=2] coordinates { (C4) (C5) };
     \draw plot [smooth, tension=2] coordinates { (C1) (0,4.5) (C6) };
     \draw plot [smooth, tension=2] coordinates { (C1) (-1,4) (C4) };
     \draw plot [smooth, tension=2] coordinates { (C3) (1,4) (C6) };
      \draw plot [smooth, tension=2] coordinates { (C2) (0,3.5) (C5) };
     
\end{tikzpicture}
\\
\\
Petersen Graph $P_{5,2}$ viewed from a vertex
  \end{tabular}
  \end{figure}
  \end{center}

Since $G$ has girth at least than $4$, each vertex $v \in V(G)$ has no edges in its neighborhood, but now its second neighborhood can contain anywhere from $2$ to $6$ vertices and may also contain edges.

Similarly as in the previous section, we let $C=(W, E_{12},E_{22})$ be a configuration where $W$ is the set of second neighbors of $v$ that are externally uncovered (free second neighborhood of $v$), $E_{12}$ is the set of edges between the first neighborhood of $v$ and the free second  neighborhood of $v$ and $E_{22}$ is the set of edges among the free second neighbors of $v$.  Let $\mathcal C_4$ be the set of all possible configurations $C$ that can arise from a cubic graph of girth at least $4$. Up to symmetries, the different possible configurations are displayed in Appendix~\ref{app:C_4}. For each configuration $C=C^j_i(x_1,\dots,x_s)$, the variables $x_k$ determine if the $k$-th second neighbor is externally covered. This gives a total of $207$ configurations.

The local view of the Petersen Graph $P_{5,2}$ only has one possible configuration: $C^1_{29}(0,0,0,0,0,0)$~(see Figure~4).

As before, we consider $\gamma_t^v(C) = \Pr[v \text{ has } t \text{ occupied neighbors}|C] $ and $\gamma_t^u(C) = \Pr[u \text{ has } t \text{ occupied neighbors}|C]$, where $u$ is chosen uniformly at random from $\{u_1,u_2,u_3\}$. 

We will use the following linear program  for $\lam \in (0, 1]$:
\begin{align*}
\alpha_{\mathrm{min}}(\lam) = \min \,    &\sum_{C \in \mathcal C_4} p(C) \alpha_{C,v}(\lam)   \text{ subject to } \\
&\sum_{C \in \mathcal C_4} p(C) = 1 \\
&\sum_{C \in \mathcal C_4} p(C) \cdot  ( \gamma^v_t(C) - \gamma^u_t(C)) =0 \, \, \,  \text{ for } t=0,1,2 \\
&p(C) \ge 0 \, \, \,  \forall C \in \mathcal C_4.
\end{align*}

The respective dual program is:
\begin{align}
\label{dualMin2}
&\alpha_{\mathrm{min}}(\lam) = \max \,   \Lam_p   \text{ subject to  } \\
\nonumber
& \Lam_p  + \sum_{t=0}^2\Lam_t \left [ \gamma^v_t(C) - \gamma^u_t(C)   \right ]  \le \alpha_{C,v}(\lam) \,\,\,\, \forall C \in \mathcal C_4.
\end{align}

Our goal is to show that $\alpha_{\mathrm{min}}(\lam) = \alpha_{P_{5,2}}(\lam)$ for all $\lam \in (0,1]$. 

If we take $\Lambda_p = \alpha_{P_{5,2}} $, then we can define the slack function of the configuration $C\in\mathcal C_4$ as a function of the dual variables $\Lam_0, \Lam_1, \Lam_2$:
\begin{align}
\label{eq:g4slack}
\text{SLACK}_{\mathrm{min}}( \lam,\Lam_0, \Lam_1, \Lam_2,C) &= \alpha_{C,v}(\lam) - \alpha_{P_{5,2}} - \sum_{t=0}^2 \Lam_t \left [ \gamma^v_t(C) - \gamma^u_t(C)   \right ].
\end{align}
Our goal is now to find values for the dual variables $\Lam_0^*, \Lam_1^*, \Lam_2^*$ so that, for all configurations $C\in \mathcal C_4$, 
\begin{align*}
\text{SLACK}_{\mathrm{min}}(\lam, \Lam_0^*, \Lam_1^*, \Lam_2^*,C) \ge 0
\end{align*}
In this case, we need to divide the interval $(0,1]$ into four intervals, and select different functions for the dual variables depending on which interval $\lam$ is in. We will not be able to show that the slack functions are positive polynomials in $\lam$: instead we will perform four different substitutions, writing $\lam$ as function of an auxiliary variable $t$ and show that the slack functions are the ratio of positive polynomials in $t$.

Note that for any $a\leq b$, the function $\lambda(t)= \frac{b(a/b +t)}{1+t}$ maps $[0,\infty)$ to $[a,b)$. Since the function $\text{SLACK}_{\mathrm{min}}(\lam, \Lam^*_0(\lam),\Lam^*_1(\lam),\Lam^*_2(\lam),C)$ is a continuous function of $\lam$, if we can show 
\[ \text{SLACK}_{\mathrm{min}}(\lam(t), \Lam^*_0(\lam(t)),\Lam^*_1(\lam(t)),\Lam^*_2(\lam(t)),C) \ge 0 \]
 for all $t \ge 0$, then we have 
 \[ \text{SLACK}_{\mathrm{min}}(\lam, \Lam^*_0(\lam),\Lam^*_1(\lam),\Lam^*_2(\lam),C) \ge 0 \]
  for all $\lam \in [a,b]$. 
 
In each of the four claims below, we will assign values to the dual variables, $\Lam_0^*, \Lam_1^*, \Lam_2^*$.  We arrived at these values by solving the dual constraints to hold with equality for a given subset of the configurations; we determined these subsets by solving instances of the dual program for fixed values of $\lam$ and observing which constraints were tight.  
  
\begin{claim}\label{cla:31}
Let $\lam_1(t) = \frac{3}{16} \frac{t}{1+t}$.  Let
\begin{align*}
\Lam_0^*(\lam) & = 0 \\
\Lam_1^* (\lam) &= \frac{3 \left(4 \lambda ^8+21 \lambda ^7+57 \lambda ^6+67 \lambda ^5+38 \lambda ^4+10 \lambda ^3+\lambda ^2\right)}{\left(4 \lambda ^4+10 \lambda ^3+11 \lambda ^2+7 \lambda +1\right) \left(5 \lambda ^4+30 \lambda ^3+30 \lambda ^2+10 \lambda +1\right)} \\
\Lam_2^*(\lam) &= \frac{3 \left(4 \lambda ^8+31 \lambda ^7+68 \lambda ^6+64 \lambda ^5+33 \lambda ^4+9 \lambda ^3+\lambda ^2\right)}{\left(4 \lambda ^4+10 \lambda ^3+11 \lambda ^2+7 \lambda +1\right) \left(5 \lambda ^4+30 \lambda ^3+30 \lambda ^2+10 \lambda +1\right)} .
\end{align*}
Then, for every $C\in \mathcal C_4$: 
\begin{align*}
\text{SLACK}_{\mathrm{min}}( \lam_1(t), 0, \Lam_1^*(\lam_1(t)), \Lam_2^*(\lam_1(t)),C) 
\end{align*}
is either identically $0$ or  the ratio of two polynomials in $t$ with all positive coefficients.

This implies $\alpha_{\mathrm{min}}(\lam) = \alpha_{P_{5,2}}(\lam)$  for $\lam \in (0, 3/16]$.
\end{claim}

\begin{claim}\label{cla:32}
Let $\lam_2(t) = \frac{11}{20} \cdot \frac{\frac{3}{16} \frac{20}{11}   +t}{ 1+t}$.  Let
\begin{align*}
\Lam_0^*(t) & = 0 \\
\Lam_1^*(\lam) &= \frac{4 \lambda ^7-13 \lambda ^6-64 \lambda ^5-75 \lambda ^4-36 \lambda ^3-6 \lambda ^2}{2 \left(5 \lambda ^7+40 \lambda ^6+100 \lambda ^5+135 \lambda ^4+111 \lambda ^3+52 \lambda ^2+12 \lambda +1\right)}\\
\Lam_2^*(\lam) &= \frac{-6 \lambda ^6-45 \lambda ^5-49 \lambda ^4-21 \lambda ^3-3 \lambda ^2}{2 \left(\lambda ^2+\lambda +1\right) \left(5 \lambda ^4+30 \lambda ^3+30 \lambda ^2+10 \lambda +1\right)}.
\end{align*}
Then, for every $C\in \mathcal C_4$: 
\begin{align*}
\text{SLACK}_{\mathrm{min}}( \lam_2(t), 0, \Lam_1^* (\lam_2(t)), \Lam_2^*(\lam_2(t)),C) 
\end{align*}
is either identically $0$ or  the ratio of two polynomials in $t$ with all positive coefficients.

This implies $\alpha_{\mathrm{min}}(\lam) = \alpha_{P_{5,2}}(\lam)$ for $\lam \in [3/16, 11/20]$.
\end{claim}

\begin{claim}\label{cla:33}
Let $\lam_3(t) = \sqrt{\frac{3}{5}} \cdot \frac{\frac{11}{20} \sqrt{\frac{5}{3}}  + t}{1+t}  $.  Let
\begin{align*}
\Lam_0^* & = 0 \\
\Lam_1^*(\lam) &= \frac{5 \lambda ^7-41 \lambda ^6-125 \lambda ^5-111 \lambda ^4-42 \lambda ^3-6 \lambda ^2}{2 \left(10 \lambda ^7+75 \lambda ^6+165 \lambda ^5+205 \lambda ^4+152 \lambda ^3+63 \lambda ^2+13 \lambda +1\right)} \\
\Lam_2^*(\lam) &=  \frac{-12 \lambda ^7-96 \lambda ^6-140 \lambda ^5-83 \lambda ^4-24 \lambda ^3-3 \lambda ^2}{2 \left(2 \lambda ^3+3 \lambda ^2+3 \lambda +1\right) \left(5 \lambda ^4+30 \lambda ^3+30 \lambda ^2+10 \lambda +1\right)} .
\end{align*}
Then, for every $C\in \mathcal C_4$: 
\begin{align*}
\text{SLACK}_{\mathrm{min}}( \lam_3(t), 0, \Lam_1^*(\lam_3(t)), \Lam_2^*(\lam_3(t)),C) 
\end{align*}
is either identically $0$ or  the ratio of two polynomials in $t$ with all positive coefficients.

This implies $\alpha_{\mathrm{min}}(\lam) = \alpha_{P_{5,2}}(\lam)$ for $\lam \in [11/20, \sqrt{3/5}]$.
\end{claim}

\begin{claim}\label{cla:34}
Let $\lam_4(t) =  \frac{ t+ \sqrt{3/5}}{t+1}$.  Let
\begin{align*}
\Lam_0^*(\lam) & = \frac{12 \lambda ^7+66 \lambda ^6+86 \lambda ^5-88 \lambda ^4-196 \lambda ^3-117 \lambda ^2-30 \lambda -3}{6 \lambda  (\lambda +1)^2 \left(5 \lambda ^4+30 \lambda ^3+30 \lambda ^2+10 \lambda +1\right)} \\
\Lam_1^*(\lam) &= \frac{-18 \lambda ^6-3 \lambda ^5-54 \lambda ^4-118 \lambda ^3-87 \lambda ^2-27 \lambda -3}{6 \lambda  (\lambda +1) \left(5 \lambda ^4+30 \lambda ^3+30 \lambda ^2+10 \lambda +1\right)}\\
\Lam_2^*(\lam) &= \frac{-6 \lambda ^6-18 \lambda ^5-100 \lambda ^4-160 \lambda ^3-105 \lambda ^2-30 \lambda -3}{6 \lambda  (\lambda +1) \left(5 \lambda ^4+30 \lambda ^3+30 \lambda ^2+10 \lambda +1\right)}.
\end{align*}
Then, for every $C\in \mathcal C_4$: 
\begin{align*}
\text{SLACK}_{\mathrm{min}}( \lam_4(t), \Lam_0^*(\lam_4(t)), \Lam_1^*(\lam_4(t)), \Lam_2^*(\lam_4(t)),C) 
\end{align*}
is either identically $0$ or  the ratio of two polynomials in $t$ with all positive coefficients.

This implies $\alpha_{\mathrm{min}}(\lam) = \alpha_{P_{5,2}}(\lam)$ for $\lam \in [\sqrt{3/5}, 1]$.
\end{claim}

The proof of these claims again proceeds by computing 
\[ \text{SLACK}_{\mathrm{min}}( \lam_i(t), \Lam_0^*(\lam_i(t)), \Lam_1^*(\lam_i(t)), \Lam_2^*(\lam_i(t)),C) \]
 in each of the above intervals of $\lam$ and for each $C \in \mathcal C_4$.  These computations are done using the Mathematica software program, and the code as well as the output of the program can be found in the ancillary files. The necessary functions $\alpha_{C,v}(\lam), \gamma_t^v(C), \gamma_t^u(C)$, are written below in Section~\ref{sec:computFunc}.

We defer the proof of uniqueness to Appendix~\ref{sec:Unqiue52}. The proof proceeds via complementary slackness and showing that the only graph whose distribution has support contained in the set of configurations whose slack function is identically $0$ is a union of copies of $P_{5,2}$ (in each of the four intervals).   

\subsection{Computing the functions of configurations}
\label{sec:computFunc}

Let us describe here how we compute the functions $\alpha_{C,v}(\lam)$, $\gamma_t^v(C)$, and $\gamma_t^u(C)$.  We aim to write the functions in a naive way that can easily be implemented with a FOR loop in a computer program.

Recall that a configuration is given by $C=(W, E_{12}, E_{22})$.  Let $U = \{ u_1, u_2, u_3 \} $ represent the neighbors of $v$.  

Then for any $S_2 \subseteq W$, we can compute the contribution to the partition function  of the configuration $C$ from an independent set that contains $v$ and such that $I \cap W = S_2$ (where we drop $\lam$ and $C$ from the functional notation):
\begin{align}
\label{eq:zplusS}
Z_{+}(S_2) &= \lam ^{1+ |S_2|} \cdot \prod_{w_1, w_2 \in S_2} \mathbf 1_{ (w_1 ,w_2) \notin E_{22}} .
\end{align}
Similarly for $S_1 \subseteq U, S_2 \subseteq W$, we can write the contribution to the partition function from $I$ that does not contain $v$ and such that $I \cap U = S_1$ and $I \cap W=S_2$.  
\begin{align}
\label{eq:zminS1S2}
Z_{-}(S_1,S_2) &= \lam ^{|S_1|+ |S_2|}\cdot \prod_{u \in S_1, w \in S_2} \mathbf 1_{ (u ,w) \notin E_{12}} \cdot \prod_{w_1, w_2 \in S_2} \mathbf 1_{ (w_1 ,w_2) \notin E_{22}} .
\end{align}
Then we can sum over all possible sets to compute the partition function:
\begin{align}
\label{eq:Zcpart}
Z(C) &= \sum_{S_2 \subseteq W} Z_{+}(S_2) + \sum_{S_1 \subseteq U, S_2 \subseteq W} Z_{-}(S_1,S_2).
\end{align}
The probability $v $ is in $I$ is then:
\begin{align}
\label{eqAlpv}
\alpha_{C,v}(\lam) &= \frac{ \sum_{S_2 \subseteq W} Z_{+}(S_2)  }{ Z(C)} .
\end{align}
We  compute the $\gamma^v_t$ functions:
\begin{align}
\label{eqGamv2}
\gamma^v_t(C)&=(1 + \mathbf 1_{t=0} \lam)  \frac{ \sum_{S_1 \subseteq U, S_2 \subseteq W} \mathbf 1_{|S_1|=t} \cdot  Z_{-}(S_1,S_2)   }{ Z(C)} . 
\end{align}
We next compute the $\gamma^u_t$ functions:
\begin{align}
\label{eqGamu}
\gamma^u_0(C)&=  \frac{1}{3} \sum_{u\in U}  \frac{1}{Z_C}  \sum_{S_1 \subseteq U, S_2 \subseteq W} \mathbf 1_{|S_2 \cap N(u)|=0} \cdot  Z_{-}(S_1,S_2) ,
\end{align}
and for $t \ge 1$,
{\small
\begin{align}
\label{eqGamu2}
\gamma^u_t(C)&=  \frac{1}{3} \sum_{u\in U}  \frac{1}{Z(C)} \left[\sum_{S_2 \subseteq W} \mathbf 1_{|S_2 \cap N(u)|=t-1} \cdot Z_{+}(S_2)   +  \sum_{S_1 \subseteq U, S_2 \subseteq W} \mathbf 1_{|S_2 \cap N(u)|=t} \cdot  Z_{-}(S_1,S_2) \right ] .
\end{align}
}

\section{Extensions}
\label{sec:extensions}

\subsection{Conjectured extremal graphs}

Conjecture~\ref{conj:Moore} states that all Moore graphs are extremal for the quantity $\frac{1}{|V(G)|} \log | \mathcal I (G)|$ for a given regularity and minimum girth.  Here we give several specific instances of this conjecture that may be amenable to these methods, and give some strengthened conjectures on the level of the occupancy fraction. 

Our first conjecture extends Theorem~\ref{thm:d3g4} and is the occupancy fraction version of the conjecture proposed in~\cite{cutler2016minimum} (though neither implies the other).
\begin{conj}\label{conj:d3LB}
Let $\lam_3^*   \approx 1.84593$ be the largest root of the equation $ 21 \lam^4 - 50 \lam^2 - 36 \lam  =7$.   
Then for every $3$-regular triangle-free graph $G$,
\begin{enumerate}
\item If $\lam \in (0, \lam_3^*]$,
\[ \alpha_G(\lam) \ge \alpha_{P_{5,2}}(\lam);\]
\item If $\lam \ge \lam_3^*$,
\[ \alpha_G(\lam) \ge \alpha_{P_{7,2}}(\lam).\] 
\end{enumerate}
\end{conj}
Conjecture~\ref{conj:d3LB} is known in the limit as $\lam \to \infty$: all triangle-free graphs of maximum degree $3$ have independence ratio at least $5/14$ and this is achieved by $P_{7,2}$~\cite{staton1979some}.

\begin{center}
\begin{figure}[h!]
\caption{}
\begin{tabular}{ c c }
  
\begin{tikzpicture}
    \node[shape=circle,fill=black] (A) at (0,2) {};
   \node[shape=circle,fill=black] (B1) at (0.479,1.942) {};
    \node[shape=circle,fill=black] (B2) at (-0.479,1.942) {};
   \node[shape=circle,fill=black] (C1) at (0.929,1.771) {};
    \node[shape=circle,fill=black] (C2) at (-0.929,1.771) {};  
   \node[shape=circle,fill=black] (D1) at (1.326,1.497) {};
    \node[shape=circle,fill=black] (D2) at (-1.326,1.497) {};  
   \node[shape=circle,fill=black] (E1) at (1.646,1.136) {};
    \node[shape=circle,fill=black] (E2) at (-1.646,1.136) {};  
   \node[shape=circle,fill=black] (F1) at (1.87,0.709) {};
    \node[shape=circle,fill=black] (F2) at (-1.87,0.709) {};  
   \node[shape=circle,fill=black] (G1) at (1.985,0.241) {};
    \node[shape=circle,fill=black] (G2) at (-1.985,0.241) {};  
    
   \node[shape=circle,fill=black] (H1) at (1.985,-0.241) {};
    \node[shape=circle,fill=black] (H2) at (-1.985,-0.241) {};  
   \node[shape=circle,fill=black] (I1) at (1.87,-0.709) {};
    \node[shape=circle,fill=black] (I2) at (-1.87,-0.709) {};    
   \node[shape=circle,fill=black] (J1) at (1.646,-1.136) {};
    \node[shape=circle,fill=black] (J2) at (-1.646,-1.136) {};
   \node[shape=circle,fill=black] (K1) at (1.326,-1.497) {};
    \node[shape=circle,fill=black] (K2) at (-1.326,-1.497) {};
   \node[shape=circle,fill=black] (L1) at (0.929,-1.771) {};
    \node[shape=circle,fill=black] (L2) at (-0.929,-1.771) {};  
   \node[shape=circle,fill=black] (M1) at (0.479,-1.942) {};
    \node[shape=circle,fill=black] (M2) at (-0.479,-1.942) {};
     \node[shape=circle,fill=black] (N) at (0,-2) {};

\draw plot [smooth, tension=2] coordinates { (A) (B1) };
\draw plot [smooth, tension=2] coordinates { (B1) (C1) };
\draw plot [smooth, tension=2] coordinates { (C1) (D1) };
\draw plot [smooth, tension=2] coordinates { (D1) (E1) };
\draw plot [smooth, tension=2] coordinates { (E1) (F1) };
\draw plot [smooth, tension=2] coordinates { (F1) (G1) };
\draw plot [smooth, tension=2] coordinates { (G1) (H1) };
\draw plot [smooth, tension=2] coordinates { (H1) (I1) };
\draw plot [smooth, tension=2] coordinates { (I1) (J1) };
\draw plot [smooth, tension=2] coordinates { (J1) (K1) };
\draw plot [smooth, tension=2] coordinates { (K1) (L1) };
\draw plot [smooth, tension=2] coordinates { (L1) (M1) };
\draw plot [smooth, tension=2] coordinates { (M1) (N) };
\draw plot [smooth, tension=2] coordinates { (N) (M2) };
\draw plot [smooth, tension=2] coordinates { (M2) (L2) };
\draw plot [smooth, tension=2] coordinates { (L2) (K2) };
\draw plot [smooth, tension=2] coordinates { (K2) (J2) };
\draw plot [smooth, tension=2] coordinates { (J2) (I2) };
\draw plot [smooth, tension=2] coordinates { (I2) (H2) };
\draw plot [smooth, tension=2] coordinates { (H2) (G2) };
\draw plot [smooth, tension=2] coordinates { (G2) (F2) };
\draw plot [smooth, tension=2] coordinates { (F2) (E2) };
\draw plot [smooth, tension=2] coordinates { (E2) (D2) };
\draw plot [smooth, tension=2] coordinates { (D2) (C2) };
\draw plot [smooth, tension=2] coordinates { (C2) (B2) };
\draw plot [smooth, tension=2] coordinates { (B2) (A) };

\draw plot [smooth, tension=2] coordinates { (A) (L1) };
\draw plot [smooth, tension=2] coordinates { (A) (H1) };
\draw plot [smooth, tension=2] coordinates { (B1) (G2) };
\draw plot [smooth, tension=2] coordinates { (B1) (K2) };
\draw plot [smooth, tension=2] coordinates { (C1) (J1) };
\draw plot [smooth, tension=2] coordinates { (C1) (N) };
\draw plot [smooth, tension=2] coordinates { (D1) (E2) };
\draw plot [smooth, tension=2] coordinates { (D1) (I2) };
\draw plot [smooth, tension=2] coordinates { (E1) (L1) };
\draw plot [smooth, tension=2] coordinates { (E1) (L2) };
\draw plot [smooth, tension=2] coordinates { (F1) (C2) };
\draw plot [smooth, tension=2] coordinates { (F1) (G2) };
\draw plot [smooth, tension=2] coordinates { (G1) (N) };
\draw plot [smooth, tension=2] coordinates { (G1) (J2) };
\draw plot [smooth, tension=2] coordinates { (B2) (M2) };
\draw plot [smooth, tension=2] coordinates { (B2) (I2) };
\draw plot [smooth, tension=2] coordinates { (C2) (J1) };
\draw plot [smooth, tension=2] coordinates { (D2) (K2) };
\draw plot [smooth, tension=2] coordinates { (D2) (M1) };
\draw plot [smooth, tension=2] coordinates { (E2) (H1) };
\draw plot [smooth, tension=2] coordinates { (F2) (K1) };
\draw plot [smooth, tension=2] coordinates { (F2) (M2) };
\draw plot [smooth, tension=2] coordinates { (H2) (M1) };
\draw plot [smooth, tension=2] coordinates { (H2) (I1) };
\draw plot [smooth, tension=2] coordinates { (J2) (K1) };    
\draw plot [smooth, tension=2] coordinates { (L2) (I1) };    
\end{tikzpicture}
 
 &
 \hspace{2cm}
\begin{tikzpicture}
    \node[shape=circle,fill=black] (A1) at (0,2) {};
    \node[shape=circle,fill=black] (A2) at (0,-2) {};
    
    \node[shape=circle,fill=black] (B1) at (0.416,1.956) {};
    \node[shape=circle,fill=black] (B2) at (-0.416,1.956) {};
    \node[shape=circle,fill=black] (B3) at (0.416,-1.956) {};
    \node[shape=circle,fill=black] (B4) at (-0.416,-1.956) {};
     
    \node[shape=circle,fill=black] (C1) at (0.813,1.827) {};
    \node[shape=circle,fill=black] (C2) at (-0.813,1.827) {};
    \node[shape=circle,fill=black] (C3) at (0.813,-1.827) {};
    \node[shape=circle,fill=black] (C4) at (-0.813,-1.827) {};
    
    \node[shape=circle,fill=black] (D1) at (1.176,1.618) {};
    \node[shape=circle,fill=black] (D2) at (-1.176,1.618) {};
    \node[shape=circle,fill=black] (D3) at (1.176,-1.618) {};
    \node[shape=circle,fill=black] (D4) at (-1.176,-1.618) {};
    
    \node[shape=circle,fill=black] (E1) at (1.486,1.338) {};
    \node[shape=circle,fill=black] (E2) at (-1.486,1.338) {};
    \node[shape=circle,fill=black] (E3) at (1.486,-1.338) {};
    \node[shape=circle,fill=black] (E4) at (-1.486,-1.338) {};

    \node[shape=circle,fill=black] (F1) at (1.732,1) {};
    \node[shape=circle,fill=black] (F2) at (-1.732,1) {};
    \node[shape=circle,fill=black] (F3) at (1.732,-1) {};
    \node[shape=circle,fill=black] (F4) at (-1.732,-1) {};

    \node[shape=circle,fill=black] (G1) at (1.902,0.618) {};
    \node[shape=circle,fill=black] (G2) at (-1.902,0.618) {};
    \node[shape=circle,fill=black] (G3) at (1.902,-0.618) {};
    \node[shape=circle,fill=black] (G4) at (-1.902,-0.618) {};

    \node[shape=circle,fill=black] (H1) at (1.990,0.209) {};
    \node[shape=circle,fill=black] (H2) at (-1.990,0.209) {};
    \node[shape=circle,fill=black] (H3) at (1.990,-0.209) {};
    \node[shape=circle,fill=black] (H4) at (-1.990,-0.209) {};
     
     \draw plot [smooth, tension=2] coordinates { (A1) (B1) };
     \draw plot [smooth, tension=2] coordinates { (A1) (B2) };
     \draw plot [smooth, tension=2] coordinates { (B1) (C1) };
     \draw plot [smooth, tension=2] coordinates { (B2) (C2) };
     \draw plot [smooth, tension=2] coordinates { (C1) (D1) };
     \draw plot [smooth, tension=2] coordinates { (C2) (D2) };     
     \draw plot [smooth, tension=2] coordinates { (D1) (E1) };
     \draw plot [smooth, tension=2] coordinates { (D2) (E2) };     
     \draw plot [smooth, tension=2] coordinates { (E1) (F1) };
     \draw plot [smooth, tension=2] coordinates { (E2) (F2) };
     \draw plot [smooth, tension=2] coordinates { (F1) (G1) };
     \draw plot [smooth, tension=2] coordinates { (F2) (G2) };
     \draw plot [smooth, tension=2] coordinates { (G1) (H1) };
     \draw plot [smooth, tension=2] coordinates { (G2) (H2) };
     \draw plot [smooth, tension=2] coordinates { (H2) (H4) };
     \draw plot [smooth, tension=2] coordinates { (H1) (H3) };
      \draw plot [smooth, tension=2] coordinates { (A2) (B3) };
     \draw plot [smooth, tension=2] coordinates { (A2) (B4) };
     \draw plot [smooth, tension=2] coordinates { (B3) (C3) };
     \draw plot [smooth, tension=2] coordinates { (B4) (C4) };
     \draw plot [smooth, tension=2] coordinates { (C3) (D3) };
     \draw plot [smooth, tension=2] coordinates { (C4) (D4) };     
     \draw plot [smooth, tension=2] coordinates { (D3) (E3) };
     \draw plot [smooth, tension=2] coordinates { (D4) (E4) };     
     \draw plot [smooth, tension=2] coordinates { (E3) (F3) };
     \draw plot [smooth, tension=2] coordinates { (E4) (F4) };
     \draw plot [smooth, tension=2] coordinates { (F3) (G3) };
     \draw plot [smooth, tension=2] coordinates { (F4) (G4) };
     \draw plot [smooth, tension=2] coordinates { (G3) (H3) };
     \draw plot [smooth, tension=2] coordinates { (G4) (H4) };

     \draw plot [smooth, tension=2] coordinates { (A1) (G3) };
     \draw plot [smooth, tension=2] coordinates { (B1) (B3) };
     \draw plot [smooth, tension=2] coordinates { (C1) (E4) };
     \draw plot [smooth, tension=2] coordinates { (D1) (G2) };
     \draw plot [smooth, tension=2] coordinates { (E1) (E3) };
     \draw plot [smooth, tension=2] coordinates { (F1) (C2) };
     \draw plot [smooth, tension=2] coordinates { (G1) (A2) };
     \draw plot [smooth, tension=2] coordinates { (H1) (F4) };
     \draw plot [smooth, tension=2] coordinates { (H3) (F2) };
     \draw plot [smooth, tension=2] coordinates { (F3) (C4) };
     \draw plot [smooth, tension=2] coordinates { (D3) (G4) };
     \draw plot [smooth, tension=2] coordinates { (C3) (E2) };
     \draw plot [smooth, tension=2] coordinates { (D2) (D4) };
     \draw plot [smooth, tension=2] coordinates { (B2) (H4) };
     \draw plot [smooth, tension=2] coordinates { (B4) (H2) };

\end{tikzpicture}
 
\\
\\
(4,6)-cage  $H_{4,6}$& 
 \hspace{2cm}
(3,8)-cage  $H_{3,8}$
  \end{tabular}
  \label{fig:P52_H36} 
  \end{figure}
  \end{center}

We conjecture that the $(3,8)$-cage graph, $H_{3,8}$, also known as the Levi graph or the Tutte-Coxeter graph (see Figure~5), maximizes the occupancy fraction  for all $\lam>0$ over all $3$-regular graphs of girth at least $7$. 
\begin{conj}
For every $3$-regular graph $G$ of girth at least $7$, and for every $\lam >0$,
\[ \alpha_G(\lam) \le \alpha_{H_{3,8}}(\lam).\]
\end{conj}

We also conjecture that the analogous version of Theorem~\ref{thm:d3g6} holds for $4$-regular graphs; that is, the $(4,6)$-cage graph, $H_{4,6}$ (see Figure~5), maximizes the occupancy fraction for all $\lam>0$ over all $4$-regular graphs of girth at least $5$. 
\begin{conj}
For every $4$-regular graph $G$ of girth at least $5$, and for every $\lam >0$,
\[ \alpha_G(\lam) \le \alpha_{H_{4,6}}(\lam).\]
\end{conj}

For the lower bound for triangle-free, $4$-regular graphs, we conjecture that the minimum is attained by one of two graphs, depending on the value of $\lam$. Let $G_{ROB}$ be the Robertson graph and $CYC_{13}$ the Cyclotomic-$13$ graph  (see Figure~6).    
\begin{conj}
\label{conj:d4LB}
Let $\lam_4^*  \approx 1.77239$ be the largest real root of the equation 
\[ 90 \lam^6+729 \lam^5-188 \lam^4-1632 \lam^3-1247 \lam^2-363 \lam=38   .  \]
Then for every $4$-regular, triangle-free graph,
\begin{enumerate}
\item If $\lam \in (0, \lam_4^*]$,
\[ \alpha_G(\lam) \ge \alpha_{G_{ROB}}(\lam);\]
\item If $\lam \ge \lam_4^*$,
\[ \alpha_G(\lam) \ge \alpha_{CYC_{13}}(\lam).\]
\end{enumerate}
\end{conj}

\begin{center}
\begin{figure}[h!]\label{fig:P52_H36} 
\caption{}
\begin{tabular}{ c c }
  
\begin{tikzpicture}
    \node[shape=circle,fill=black] (A) at (0,2) {};
   \node[shape=circle,fill=black] (B1) at (0.929,1.771) {};
    \node[shape=circle,fill=black] (B2) at (-0.929,1.771) {};
   \node[shape=circle,fill=black] (C1) at (1.646,1.136) {};
    \node[shape=circle,fill=black] (C2) at (-1.646,1.136) {};  
   \node[shape=circle,fill=black] (D1) at (1.985,0.241) {};
    \node[shape=circle,fill=black] (D2) at (-1.985,0.241) {};  
   \node[shape=circle,fill=black] (E1) at (1.87,-0.709) {};
    \node[shape=circle,fill=black] (E2) at (-1.87,-0.709) {};  
   \node[shape=circle,fill=black] (F1) at (1.326,-1.497) {};
    \node[shape=circle,fill=black] (F2) at (-1.326,-1.497) {};  
   \node[shape=circle,fill=black] (G1) at (0.479,-1.942) {};
    \node[shape=circle,fill=black] (G2) at (-0.479,-1.942) {};  
    
\draw plot [smooth, tension=2] coordinates { (A) (B1) };
\draw plot [smooth, tension=2] coordinates { (B1) (C1) };
\draw plot [smooth, tension=2] coordinates { (C1) (D1) };
\draw plot [smooth, tension=2] coordinates { (D1) (E1) };
\draw plot [smooth, tension=2] coordinates { (E1) (F1) };
\draw plot [smooth, tension=2] coordinates { (F1) (G1) };
\draw plot [smooth, tension=2] coordinates { (G1) (G2) };
\draw plot [smooth, tension=2] coordinates { (G2) (F2) };
\draw plot [smooth, tension=2] coordinates { (F2) (E2) };
\draw plot [smooth, tension=2] coordinates { (E2) (D2) };
\draw plot [smooth, tension=2] coordinates { (D2) (C2) };
\draw plot [smooth, tension=2] coordinates { (C2) (B2) };
\draw plot [smooth, tension=2] coordinates { (B2) (A) };

\draw plot [smooth, tension=2] coordinates { (A) (F1) };
\draw plot [smooth, tension=2] coordinates { (A) (F2) };
\draw plot [smooth, tension=2] coordinates { (B1) (G1) };
\draw plot [smooth, tension=2] coordinates { (B1) (E2) };
\draw plot [smooth, tension=2] coordinates { (C1) (G2) };
\draw plot [smooth, tension=2] coordinates { (C1) (D2) };
\draw plot [smooth, tension=2] coordinates { (D1) (F2) };
\draw plot [smooth, tension=2] coordinates { (D1) (C2) };
\draw plot [smooth, tension=2] coordinates { (E1) (E2) };
\draw plot [smooth, tension=2] coordinates { (E1) (B2) };
\draw plot [smooth, tension=2] coordinates { (F1) (D2) };
\draw plot [smooth, tension=2] coordinates { (G1) (C2) };
\draw plot [smooth, tension=2] coordinates { (G2) (B2) };
\end{tikzpicture}
 
  &

 \hspace{0.5cm}
\begin{tikzpicture}
    \node[shape=circle,fill=black] (A) at (0,2) {};
   \node[shape=circle,fill=black] (B1) at (0.649,1.891) {};
    \node[shape=circle,fill=black] (B2) at (-0.649,1.891) {};
   \node[shape=circle,fill=black] (C1) at (1.228,1.578) {};
    \node[shape=circle,fill=black] (C2) at (-1.228,1.578) {};  
   \node[shape=circle,fill=black] (D1) at (1.679,1.094) {};
    \node[shape=circle,fill=black] (D2) at (-1.679,1.094) {};  
   \node[shape=circle,fill=black] (E1) at (1.939,0.491) {};
    \node[shape=circle,fill=black] (E2) at (-1.939,0.491) {};  
   \node[shape=circle,fill=black] (F1) at (1.993,-0.165) {};
    \node[shape=circle,fill=black] (F2) at (-1.993,-0.165) {};  
   \node[shape=circle,fill=black] (G1) at (1.831,-0.803) {};
    \node[shape=circle,fill=black] (G2) at (-1.831,-0.803) {}; 
   \node[shape=circle,fill=black] (H1) at (1.471,-1.354) {};
    \node[shape=circle,fill=black] (H2) at (-1.471,-1.354) {}; 
   \node[shape=circle,fill=black] (I1) at (0.952,-1.759) {};
    \node[shape=circle,fill=black] (I2) at (-0.952,-1.759) {}; 
   \node[shape=circle,fill=black] (J1) at (0.329,-1.972) {};
    \node[shape=circle,fill=black] (J2) at (-0.329,-1.972) {}; 

\draw plot [smooth, tension=2] coordinates { (A) (B1) };
\draw plot [smooth, tension=2] coordinates { (B1) (C1) };
\draw plot [smooth, tension=2] coordinates { (C1) (D1) };
\draw plot [smooth, tension=2] coordinates { (D1) (E1) };
\draw plot [smooth, tension=2] coordinates { (E1) (F1) };
\draw plot [smooth, tension=2] coordinates { (F1) (G1) };
\draw plot [smooth, tension=2] coordinates { (G1) (H1) };
\draw plot [smooth, tension=2] coordinates { (H1) (I1) };
\draw plot [smooth, tension=2] coordinates { (I1) (J1) };
\draw plot [smooth, tension=2] coordinates { (J1) (J2) };
\draw plot [smooth, tension=2] coordinates { (J2) (I2) };
\draw plot [smooth, tension=2] coordinates { (I2) (H2) };
\draw plot [smooth, tension=2] coordinates { (H2) (G2) };
\draw plot [smooth, tension=2] coordinates { (G2) (F2) };
\draw plot [smooth, tension=2] coordinates { (F2) (E2) };
\draw plot [smooth, tension=2] coordinates { (E2) (D2) };
\draw plot [smooth, tension=2] coordinates { (D2) (C2) };
\draw plot [smooth, tension=2] coordinates { (C2) (B2) };
\draw plot [smooth, tension=2] coordinates { (B2) (A) };

\draw plot [smooth, tension=2] coordinates { (A) (E1) };
\draw plot [smooth, tension=2] coordinates { (A) (F2) };
\draw plot [smooth, tension=2] coordinates { (B1) (J1) };
\draw plot [smooth, tension=2] coordinates { (B1) (H2) };
\draw plot [smooth, tension=2] coordinates { (C1) (G1) };
\draw plot [smooth, tension=2] coordinates { (C1) (C2) };
\draw plot [smooth, tension=2] coordinates { (D1) (I2) };
\draw plot [smooth, tension=2] coordinates { (D1) (E2) };
\draw plot [smooth, tension=2] coordinates { (E1) (I1) };
\draw plot [smooth, tension=2] coordinates { (F1) (H2) };
\draw plot [smooth, tension=2] coordinates { (F1) (D2) };

\draw plot [smooth, tension=2] coordinates { (G1) (J2) };
\draw plot [smooth, tension=2] coordinates { (H1) (B2) };
\draw plot [smooth, tension=2] coordinates { (H1) (E2) };
\draw plot [smooth, tension=2] coordinates { (I1) (G2) };
\draw plot [smooth, tension=2] coordinates { (J1) (D2) };
\draw plot [smooth, tension=2] coordinates { (J2) (F2) };
\draw plot [smooth, tension=2] coordinates { (I2) (B2) };
\draw plot [smooth, tension=2] coordinates { (G2) (C2) };
\end{tikzpicture}
\\
\\
Cyclotomic-13 graph $CYC_{13}$ 
&
\hspace{0.5cm}
Robertson Graph $G_{\text{ROB}}$
  \end{tabular}
  \end{figure}
  \end{center}
Conjecture~\ref{conj:d4LB} is known in the limit as $\lam \to \infty$: all triangle-free graphs of maximum degree $4$ have independence ratio at least $4/13$ and this is achieved by $CYC_{13}$~\cite{jones1984independence}.

\subsection{Graph homomorphisms}

We can also ask about generalizations from independent sets to graph homomorphisms. Let $\text{Hom}(G,H)$ denote the number of graph homomorphisms from $G$ into $H$; that is, the number of mappings $\phi: V(G) \to V(H)$ so that $(u,v) \in E(G) \Rightarrow (\phi(u), \phi(v)) \in E(H)$.  The number of independent sets of $G$ is $\text{Hom}(G, H_{ind})$ where $H_{ind}$ is a single edge with one looped vertex. The number of proper vertex $q$-colorings of $G$ is $\text{Hom}(G, K_q)$. Galvin and Tetali~\cite{galvin2004weighted} showed that for all $d$-regular bipartite $G$, and all $H$, 
\[ \text{Hom}(G,H)^{1/|V(G)|} \le  \text{Hom}(K_{d,d},H)^{1/2d} .\]
In \cite{cohen2016widom}, Cohen, Csikv{\'a}ri, Perkins, and Tetali conjectured that $G$ being triangle-free can replace the bipartite condition.   We ask (but don't dare conjecture) whether something analogous holds for graphs of larger girth. 
\begin{question}
Is it true that for all graphs $H$ and all cubic graphs $G$ of girth at least $6$,
\[ \text{Hom}(G,H)^{1/|V(G)|} \le \text{Hom}(H_{3,6},H)^{1/14} \, ?\]
\end{question}
There are counterexamples if we replace girth at least $6$ by girth at least $5$, e.g. $H$ is two looped vertices. 

\section*{Acknowledgements}
We thank Daniela K\"uhn and Deryk Osthus for fruitful discussions during the initial steps of this project.

\bibliography{girthBib}

\begin{thebibliography}{10}

\bibitem{cohen2016widom}
E.~Cohen, P.~Csikv{\'a}ri, W.~Perkins, and P.~Tetali.
\newblock The {W}idom--{R}owlinson model, the hard-core model and the
  extremality of the complete graph.
\newblock {\em European Journal of Combinatorics}, 62:70--76, 2017.

\bibitem{cohen2015widom}
E.~Cohen, W.~Perkins, and P.~Tetali.
\newblock On the {W}idom--{R}owlinson occupancy fraction in regular graphs.
\newblock {\em Combinatorics, Probability and Computing}, 26(2):183--194, 2017.

\bibitem{cutler2014maximum}
J.~Cutler and A.~Radcliffe.
\newblock The maximum number of complete subgraphs in a graph with given
  maximum degree.
\newblock {\em Journal of Combinatorial Theory, Series B}, 104:60--71, 2014.

\bibitem{cutler2016minimum}
J.~Cutler and A.~Radcliffe.
\newblock Minimizing the number of independent sets in triangle-free regular
  graphs.
\newblock {\em Discrete Mathematics}, 341(3):793--800, 2018.

\bibitem{davies2015independent}
E.~Davies, M.~Jenssen, W.~Perkins, and B.~Roberts.
\newblock Independent sets, matchings, and occupancy fractions.
\newblock {\em Journal of the London Mathematical Society}, 96(1):47--66, 2017.

\bibitem{davies2016averagens}
E.~Davies, M.~Jenssen, W.~Perkins, and B.~Roberts.
\newblock On the average size of independent sets in triangle-free graphs.
\newblock {\em Proceedings of the American Mathematical Society},
  146(1):111--124, 2018.

\bibitem{PottsExtremes}
E.~Davies, M.~Jenssen, W.~Perkins, and B.~Roberts.
\newblock Extremes of the internal energy of the {P}otts models on cubic
  graphs.
\newblock {\em Random Structures \& Algorithms}, to appear.

\bibitem{galvin2004weighted}
D.~Galvin and P.~Tetali.
\newblock On weighted graph homomorphisms.
\newblock {\em DIMACS Series in Discrete Mathematics and Theoretical Computer
  Science}, 63:97--104, 2004.

\bibitem{jones1984independence}
K.~F. Jones.
\newblock Independence in graphs with maximum degree four.
\newblock {\em Journal of Combinatorial Theory, Series B}, 37(3):254--269,
  1984.

\bibitem{kahn2001entropy}
J.~Kahn.
\newblock An entropy approach to the hard-core model on bipartite graphs.
\newblock {\em Combinatorics, Probability and Computing}, 10(03):219--237,
  2001.

\bibitem{staton1979some}
W.~Staton.
\newblock Some {R}amsey-type numbers and the independence ratio.
\newblock {\em Transactions of the American Mathematical Society},
  256:353--370, 1979.

\bibitem{wong1982cages}
P.-K. Wong.
\newblock Cages--a survey.
\newblock {\em Journal of Graph Theory}, 6(1):1--22, 1982.

\bibitem{zhao2010number}
Y.~Zhao.
\newblock The number of independent sets in a regular graph.
\newblock {\em Combinatorics, Probability and Computing}, 19(02):315--320,
  2010.

\bibitem{yufeiSurvey}
Y.~Zhao.
\newblock Extremal regular graphs: independent sets and graph homomorphisms.
\newblock {\em American Mathematical Monthly}, 124(9):827--843, 2017.

\end{thebibliography}
\bibliographystyle{abbrv}

\newpage
\appendix

\section{Proof that $\mathcal{C}_5$ is the set of all possible configurations for free second neighborhoods of cubic graphs with girth at least $5$.}\label{app:conf}

Let $u_1,u_2,u_3$ be the neighbors of $v$. For $i\in\{1,2,3\}$, let  $W_i=\{w_{i1},w_{21}\}$ be the neighbors of $u_i$ that are second neighbors of $v$, and let $W=W_1\cup W_2\cup W_3$.

Suppose that $G[W]$ contains a copy of a graph $H$, then we will write $\phi: V(H)\to W$ for the map that assigns each vertex of $H$ to its corresponding vertex in the copy of $H$ in $G[W]$. 
If $H$ is a graph that can be contained in $G[W]$, then its order is at most $6$, its maximum degree of $H$ is at most $2$ and it has girth at least $5$. Moreover, the endpoints of an edge from $H$ cannot be mapped to the same set $W_i$, since otherwise $W_i\cup \{u_i\}$ would induce a triangle in $G$.

Let $S_n$ denote the symmetric group of order $n$. Two configuration types are \emph{equal up to symmetries} if there exist $\sigma\in S_3$, $\tau_1,\tau_2,\tau_3\in S_2$  such that the bijection $\rho: W\to W$ defined by $\rho(w_{ij})=w_{\sigma(i) \tau_i(j)}$ transforms one configuration into the other one.

We now enumerate the different configurations for the second neighborhood of $v$, depending on the number of edges in $G[W]$:
\begin{itemize}
\item[0)] If $G[W]$ induces no edges, the configuration is of type $C_0$.

\item[1)] If $G[W]$ induces one edge, then, up to symmetries, $C_1$ is the only possible type of configuration.

\item[2)] If $G[W]$ induces two edges, then we have two cases for $H$:
\begin{itemize}
\item[2.a)] $H$ is a path $x_1x_2x_3$ of length $2$. Since no edges can be contained in any $W_i$, we  have that that $\phi(x_2)$ lies in a different set $W_i$ that the images of the other two vertices. We may assume $\phi(x_2)\in W_1$ and $\phi(x_1)\in W_2$. If $\phi(x_3)\in W_2$, then $G[W\cup u_2]$  contains a $C_4$. Therefore, $\phi(x_3)\in W_3$, and $C_2$ is the only possible configuration type.

\item[2.b)] $H$ is composed by two vertex-disjoint edges $x_1x_2$ and $y_1y_2$. We may assume that $W_1$ contains the images of two vertices. Since  no edge is included within $W_1$, we may assume $\phi(x_1)=w_{11}$ and $\phi(y_1)=w_{12}$. If $\phi(x_2)$ and $\phi(y_2)$ lie in the same set $W_i$, we may assume that $\phi(x_2)=w_{22}$ and $\phi(y_2)=w_{21}$, giving rise to a configuration of type $C_3$. Otherwise, up to symmetries, $\phi(x_2)=w_{21}$ and $\phi(y_2)=w_{31}$, giving rise to a configuration of type $C_4$.
\end{itemize} 

\item[3)] If $G[W]$ induces $3$ edges, then we have three cases for $H$:
\begin{itemize}
\item[3.a)] $H$ is composed by three vertex-disjoint edges $x_1x_2$, $y_1y_2$ and $z_1z_2$. Since edges cannot lie within the sets $W_i$, the type of configuration $C_5$ is the only possible one, up to symmetries.

\item[3.b)] $H$ is a path of length $3$, $x_1x_2x_3x_4$. We can assume it contains two vertices in $W_1$. Moreover, these should be $\phi(x_1)$ and $\phi(x_4)$, since otherwise $G[W\cup\{u_1\}]$ would contain a $C_4$. A similar argument shows that $\phi(x_2)$ and $\phi(x_3)$ lie in different sets $W_i$ and thus, up to symmetries, the only type of configuration is $C_6$.
	
\item[3.c)] $H$ is composed by the vertex-disjoint union of a path of length $2$, $x_1x_2x_3$, and an edge $y_1y_2$. As in the case where we were embedding only a path of length $2$, we may assume that $\phi(x_1)=w_{11}$, $\phi(x_2)=w_{21}$ and $\phi(x_3)=w_{31}$. Up to symmetries we may assume that $\phi(y_1)=w_{22}$. If $\phi(y_2)=w_{12}$, then we obtain a configuration of type $C_7$ and if $\phi(y_2)=w_{32}$, of type $C_8$.
\end{itemize}

\item[4)] If $G[W]$ induces $4$ edges, then we have three cases for $H$:
\begin{itemize}
\item[4.a)] $H$ is composed by two vertex disjoint paths of length $2$, $x_1x_2x_3$ and $y_1y_2y_3$. Then we may assume that $\phi(x_2)=w_{11}$.  If  $\phi(y_2)\neq w_{12}$, we may assume that $\phi(y_2)=w_{32}$ and, up to symmetries, we obtain a configuration of type $C_9$. If $\phi(y_2)=w_{12}$, up to symmetries, the only type of configuration is $C_{10}$.

\item[4.b)] $H$ is composed by the disjoint union of a path of length $3$, $x_1x_2x_3x_4$ and an edge $y_1y_2$. Then we may assume that $\phi(y_1)=w_{22}$ and $\phi(y_2)=w_{23}$ and, as in the case where we were embedding only a path of length $3$, we should have $\phi(x_1)=w_{11}$ and $\phi(x_4)=w_{12}$. Up to symmetries, this gives rise to a configuration of type $C_{11}$.

\item[4.c)] $H$ is composed by a path of length $4$, $x_1x_2x_3x_4x_5$. We can assume it contains only one vertex in $W_3$. Since $H$ contains a path of length $4$ as a subpath, we can also assume that $\phi(x_1)=w_{11}$, $\phi(x_2)=w_{21}$, $\phi(x_3)=w_{31}$ and $\phi(x_4)=w_{12}$. Up to symmetries, the only possible type of configuration is $C_{12}$.
\end{itemize}

\item[5)] If $G[W]$ induces $5$ edges, then $G[W]$ must be isomorphic to a path of length $5$. Note that if $G[W]$ would induce a $C_5$, one could find a set $W_i$ such that $w_{i1}$ and $w_{i2}$ have a common neighbor in $W$, creating a $C_4$ in $G[W\cup\{u_i\}]$. Let $H$ be a path of length $5$, $x_1x_2x_3x_4x_5x_6$. Since $H$ contains a path of length $4$, we can assume that  $\phi(x_1)=w_{11}$, $\phi(x_2)=w_{21}$, $\phi(x_3)=w_{31}$, $\phi(x_4)=w_{12}$ and $\phi(x_5)=w_{22}$. Then $\phi(x_6)=w_{32}$ and the only type of configuration is $C_{13}$.

\item[6)] If $G[W]$ induces $6$ edges, then $G[W]$ must be isomorphic to a cycle of length $6$. A similar argument as before gives $C_{14}$ as the only possible configuration.

\end{itemize}

\newpage
\section{Configurations for second free neighborhoods in triangle-free graphs}\label{app:C_4}
In this part of the appendix we display all the configurations in $\mathcal{C}_4$.

\subsection{}
The following configurations are obtained if the second neighborhood  of a vertex has size six. The variables $x_i\in\{0,1\}$ indicates whether the $i$-th vertex in the second neighborhood is externally uncovered ($x_i=1$) or not ($x_i=0$).

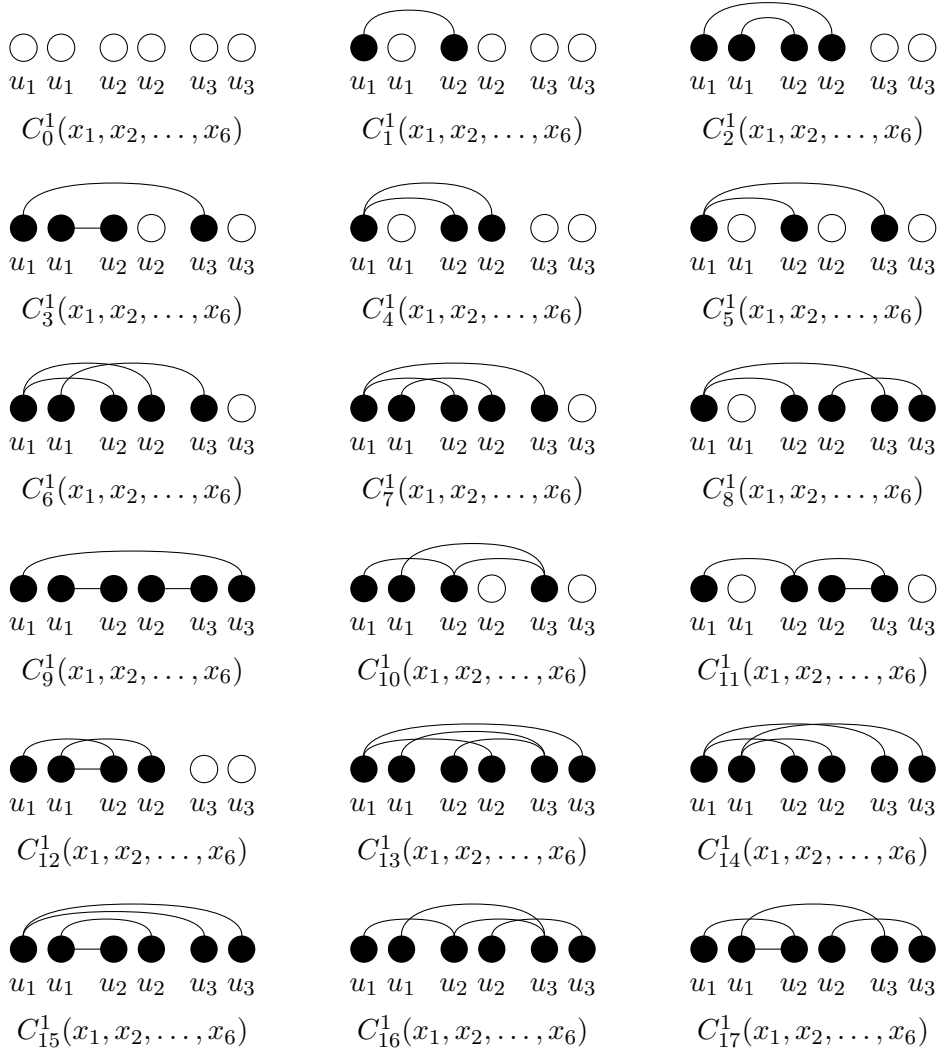
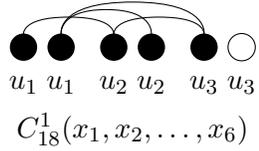
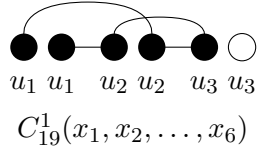
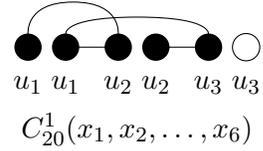
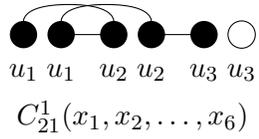
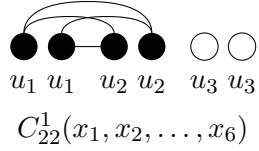
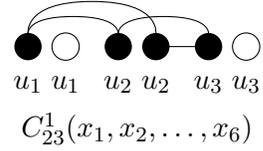
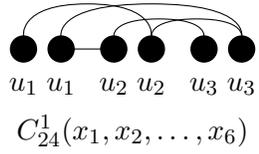
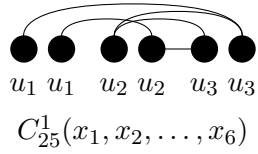
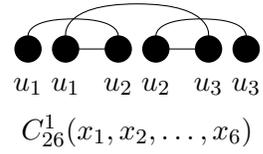
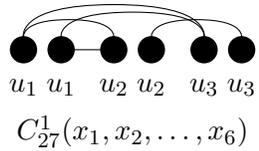
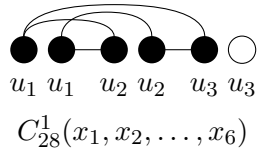
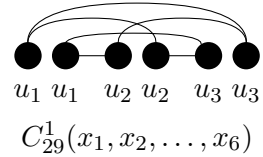
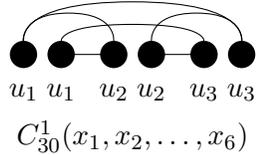
\begin{center}
\begin{tabular}{ c c c }
  \begin{tikzpicture}
    \node[shape=circle,draw=black] (A) at (0,0) {};
    \node[shape=circle,draw=black] (B) at (0.5,0) {};
    \node[shape=circle,draw=black] (C) at (1.2,0) {};
    \node[shape=circle,draw=black] (D) at (1.7,0) {};
    \node[shape=circle,draw=black] (E) at (2.4,0) {};
    \node[shape=circle,draw=black] (F) at (2.9,0) {} ;
    
    \node (G) at (0,-0.5) {$u_1$};
    \node (H) at (0.5,-0.5) {$u_1$};
    \node (I) at (1.2,-0.5) {$u_2$};
    \node (J) at (1.7,-0.5) {$u_2$};
    \node (K) at (2.4,-0.5) {$u_3$};
    \node (L) at (2.9,-0.5) {$u_3$} ;
\end{tikzpicture}
  &
  \hspace{0.5cm}
\begin{tikzpicture}
    \node[shape=circle,fill=black] (A) at (0,0) {};
    \node[shape=circle,draw=black] (B) at (0.5,0) {};
    \node[shape=circle,fill=black] (C) at (1.2,0) {};
    \node[shape=circle,draw=black] (D) at (1.7,0) {};
    \node[shape=circle,draw=black] (E) at (2.4,0) {};
    \node[shape=circle,draw=black] (F) at (2.9,0) {} ;
    
    \node (G) at (0,-0.5) {$u_1$};
    \node (H) at (0.5,-0.5) {$u_1$};
    \node (I) at (1.2,-0.5) {$u_2$};
    \node (J) at (1.7,-0.5) {$u_2$};
    \node (K) at (2.4,-0.5) {$u_3$};
    \node (L) at (2.9,-0.5) {$u_3$} ;

\draw plot [smooth, tension=2] coordinates { (0,0.15) (0.6,0.5) (1.2,0.15)};
\end{tikzpicture}
  &

  \hspace{0.5cm}
  \begin{tikzpicture}
    \node[shape=circle,fill=black] (A) at (0,0) {};
    \node[shape=circle,fill=black] (B) at (0.5,0) {};
    \node[shape=circle,fill=black] (C) at (1.2,0) {};
    \node[shape=circle,fill=black] (D) at (1.7,0) {};
    \node[shape=circle,draw=black] (E) at (2.4,0) {};
    \node[shape=circle,draw=black] (F) at (2.9,0) {} ;
    
     \node (G) at (0,-0.5) {$u_1$};
    \node (H) at (0.5,-0.5) {$u_1$};
    \node (I) at (1.2,-0.5) {$u_2$};
    \node (J) at (1.7,-0.5) {$u_2$};
    \node (K) at (2.4,-0.5) {$u_3$};
    \node (L) at (2.9,-0.5) {$u_3$} ;
    
\draw plot [smooth, tension=2] coordinates { (0,0.15) (0.85 ,0.6) (1.7,0.15)};
\draw plot [smooth, tension=2] coordinates { (0.5,0.15) (0.85,0.4) (1.2,0.15)};
\end{tikzpicture}
 \\
$C^1_0(x_1,x_2,\dots,x_6)$ & \hspace{0.5cm} $C^1_1(x_1,x_2,\dots,x_6)$ & \hspace{0.5cm} $C^1_2(x_1,x_2,\dots,x_6)$\\
 \\

\begin{tikzpicture}
    \node[shape=circle,fill=black] (A) at (0,0) {};
    \node[shape=circle,fill=black] (B) at (0.5,0) {};
    \node[shape=circle,fill=black] (C) at (1.2,0) {};
    \node[shape=circle,draw=black] (D) at (1.7,0) {};
    \node[shape=circle,fill=black] (E) at (2.4,0) {};
    \node[shape=circle,draw=black] (F) at (2.9,0) {} ;

     \node (G) at (0,-0.5) {$u_1$};
    \node (H) at (0.5,-0.5) {$u_1$};
    \node (I) at (1.2,-0.5) {$u_2$};
    \node (J) at (1.7,-0.5) {$u_2$};
    \node (K) at (2.4,-0.5) {$u_3$};
    \node (L) at (2.9,-0.5) {$u_3$} ;
    
\draw plot [smooth, tension=2] coordinates { (0,0.15) (1.2 ,0.6) (2.4,0.15)};
\draw plot [smooth, tension=2] coordinates { (0.5,0)(1.2,0)};
\end{tikzpicture}
&
 \hspace{0.5cm}
\begin{tikzpicture}
    \node[shape=circle,fill=black] (A) at (0,0) {};
    \node[shape=circle,draw=black] (B) at (0.5,0) {};
    \node[shape=circle,fill=black] (C) at (1.2,0) {};
    \node[shape=circle,fill=black] (D) at (1.7,0) {};
    \node[shape=circle,draw=black] (E) at (2.4,0) {};
    \node[shape=circle,draw=black] (F) at (2.9,0) {} ;
    
     \node (G) at (0,-0.5) {$u_1$};
    \node (H) at (0.5,-0.5) {$u_1$};
    \node (I) at (1.2,-0.5) {$u_2$};
    \node (J) at (1.7,-0.5) {$u_2$};
    \node (K) at (2.4,-0.5) {$u_3$};
    \node (L) at (2.9,-0.5) {$u_3$} ;

\draw plot [smooth, tension=2] coordinates { (0,0.15) (0.6,0.4) (1.2,0.15)};
\draw plot [smooth, tension=2] coordinates { (0,0.15) (0.85,0.6) (1.7,0.15)};
\end{tikzpicture}  

&
\hspace{0.5cm}
\begin{tikzpicture}
    \node[shape=circle,fill=black] (A) at (0,0) {};
    \node[shape=circle,draw=black] (B) at (0.5,0) {};
    \node[shape=circle,fill=black] (C) at (1.2,0) {};
    \node[shape=circle,draw=black] (D) at (1.7,0) {};
    \node[shape=circle,fill=black] (E) at (2.4,0) {};
    \node[shape=circle,draw=black] (F) at (2.9,0) {} ;
    
     \node (G) at (0,-0.5) {$u_1$};
    \node (H) at (0.5,-0.5) {$u_1$};
    \node (I) at (1.2,-0.5) {$u_2$};
    \node (J) at (1.7,-0.5) {$u_2$};
    \node (K) at (2.4,-0.5) {$u_3$};
    \node (L) at (2.9,-0.5) {$u_3$} ;

\draw plot [smooth, tension=2] coordinates { (0,0.15) (0.6,0.4) (1.2,0.15)};
\draw plot [smooth, tension=2] coordinates { (0,0.15) (1.2,0.6) (2.4,0.15)};
\end{tikzpicture}  
  
 \\
 $C^1_3(x_1,x_2,\dots,x_6)$ & \hspace{0.5cm} $C^1_4(x_1,x_2,\dots,x_6)$ & \hspace{0.5cm} $C^1_5(x_1,x_2,\dots,x_6)$\\
 \\

\begin{tikzpicture}
    \node[shape=circle,fill=black] (A) at (0,0) {};
    \node[shape=circle,fill=black] (B) at (0.5,0) {};
    \node[shape=circle,fill=black] (C) at (1.2,0) {};
    \node[shape=circle,fill=black] (D) at (1.7,0) {};
    \node[shape=circle,fill=black] (E) at (2.4,0) {};
    \node[shape=circle,draw=black] (F) at (2.9,0) {} ;

     \node (G) at (0,-0.5) {$u_1$};
    \node (H) at (0.5,-0.5) {$u_1$};
    \node (I) at (1.2,-0.5) {$u_2$};
    \node (J) at (1.7,-0.5) {$u_2$};
    \node (K) at (2.4,-0.5) {$u_3$};
    \node (L) at (2.9,-0.5) {$u_3$} ;
    
\draw plot [smooth, tension=2] coordinates { (0,0.15) (0.6 ,0.4) (1.2,0.15)};
\draw plot [smooth, tension=2] coordinates { (0,0.15) (0.85,0.6) (1.7,0.15)};
\draw plot [smooth, tension=2] coordinates { (0.5,0.15) (1.45,0.6) (2.4,0.15)};
\end{tikzpicture}

  &
  \hspace{0.5cm}
\begin{tikzpicture}
    \node[shape=circle,fill=black] (A) at (0,0) {};
    \node[shape=circle,fill=black] (B) at (0.5,0) {};
    \node[shape=circle,fill=black] (C) at (1.2,0) {};
    \node[shape=circle,fill=black] (D) at (1.7,0) {};
    \node[shape=circle,fill=black] (E) at (2.4,0) {};
    \node[shape=circle,draw=black] (F) at (2.9,0) {} ;

     \node (G) at (0,-0.5) {$u_1$};
    \node (H) at (0.5,-0.5) {$u_1$};
    \node (I) at (1.2,-0.5) {$u_2$};
    \node (J) at (1.7,-0.5) {$u_2$};
    \node (K) at (2.4,-0.5) {$u_3$};
    \node (L) at (2.9,-0.5) {$u_3$} ;
    
\draw plot [smooth, tension=2] coordinates { (0,0.15) (0.6 ,0.4) (1.2,0.15)};
\draw plot [smooth, tension=2] coordinates { (0.5,0.15) (1.1,0.4) (1.7,0.15)};
\draw plot [smooth, tension=2] coordinates { (0,0.15) (1.2,0.6) (2.4,0.15)};
\end{tikzpicture}
  & 
\hspace{0.5cm}
\begin{tikzpicture}
    \node[shape=circle,fill=black] (A) at (0,0) {};
    \node[shape=circle,draw=black] (B) at (0.5,0) {};
    \node[shape=circle,fill=black] (C) at (1.2,0) {};
    \node[shape=circle,fill=black] (D) at (1.7,0) {};
    \node[shape=circle,fill=black] (E) at (2.4,0) {};
    \node[shape=circle,fill=black] (F) at (2.9,0) {} ;

     \node (G) at (0,-0.5) {$u_1$};
    \node (H) at (0.5,-0.5) {$u_1$};
    \node (I) at (1.2,-0.5) {$u_2$};
    \node (J) at (1.7,-0.5) {$u_2$};
    \node (K) at (2.4,-0.5) {$u_3$};
    \node (L) at (2.9,-0.5) {$u_3$} ;
    
\draw plot [smooth, tension=2] coordinates { (0,0.15) (0.6 ,0.4) (1.2,0.15)};
\draw plot [smooth, tension=2] coordinates { (1.7,0.15) (2.3,0.4) (2.9,0.15)};
\draw plot [smooth, tension=2] coordinates { (0,0.15) (1.2,0.6) (2.4,0.15)};
\end{tikzpicture}  
 
  \\ 
 $C^1_6(x_1,x_2,\dots,x_6)$ & \hspace{0.5cm} $C^1_7(x_1,x_2,\dots,x_6)$ & \hspace{0.5cm} $C^1_8(x_1,x_2,\dots,x_6)$\\
 \\

\begin{tikzpicture}
    \node[shape=circle,fill=black] (A) at (0,0) {};
    \node[shape=circle,fill=black] (B) at (0.5,0) {};
    \node[shape=circle,fill=black] (C) at (1.2,0) {};
    \node[shape=circle,fill=black] (D) at (1.7,0) {};
    \node[shape=circle,fill=black] (E) at (2.4,0) {};
    \node[shape=circle,fill=black] (F) at (2.9,0) {} ;

     \node (G) at (0,-0.5) {$u_1$};
    \node (H) at (0.5,-0.5) {$u_1$};
    \node (I) at (1.2,-0.5) {$u_2$};
    \node (J) at (1.7,-0.5) {$u_2$};
    \node (K) at (2.4,-0.5) {$u_3$};
    \node (L) at (2.9,-0.5) {$u_3$} ;
    
\draw plot [smooth, tension=2] coordinates { (0.5,0) (1.2,0)};
\draw plot [smooth, tension=2] coordinates { (1.7,0) (2.4,0)};
\draw plot [smooth, tension=2] coordinates { (0,0.15) (1.45,0.5)  (2.9,0.15)};
\end{tikzpicture}  
 
 & 
 \hspace{0.5cm}
 
  \begin{tikzpicture}
    \node[shape=circle,fill=black] (A) at (0,0) {};
    \node[shape=circle,fill=black] (B) at (0.5,0) {};
    \node[shape=circle,fill=black] (C) at (1.2,0) {};
    \node[shape=circle,draw=black] (D) at (1.7,0) {};
    \node[shape=circle,fill=black] (E) at (2.4,0) {};
    \node[shape=circle,draw=black] (F) at (2.9,0) {} ;

     \node (G) at (0,-0.5) {$u_1$};
    \node (H) at (0.5,-0.5) {$u_1$};
    \node (I) at (1.2,-0.5) {$u_2$};
    \node (J) at (1.7,-0.5) {$u_2$};
    \node (K) at (2.4,-0.5) {$u_3$};
    \node (L) at (2.9,-0.5) {$u_3$} ;

\draw plot [smooth, tension=2] coordinates { (0,0.15) (0.6 ,0.4) (1.2,0.15)};
\draw plot [smooth, tension=2] coordinates { (0.5,0.15) (1.45,0.6) (2.4,0.15)};
\draw plot [smooth, tension=2] coordinates { (1.2,0.15) (1.8,0.4) (2.4,0.15)};

\end{tikzpicture}
& 

 \hspace{0.5cm}
 
  \begin{tikzpicture}
    \node[shape=circle,fill=black] (A) at (0,0) {};
    \node[shape=circle,draw=black] (B) at (0.5,0) {};
    \node[shape=circle,fill=black] (C) at (1.2,0) {};
    \node[shape=circle,fill=black] (D) at (1.7,0) {};
    \node[shape=circle,fill=black] (E) at (2.4,0) {};
    \node[shape=circle,draw=black] (F) at (2.9,0) {} ;

     \node (G) at (0,-0.5) {$u_1$};
    \node (H) at (0.5,-0.5) {$u_1$};
    \node (I) at (1.2,-0.5) {$u_2$};
    \node (J) at (1.7,-0.5) {$u_2$};
    \node (K) at (2.4,-0.5) {$u_3$};
    \node (L) at (2.9,-0.5) {$u_3$} ;

\draw plot [smooth, tension=2] coordinates { (0,0.15) (0.6 ,0.4) (1.2,0.15)};
\draw plot [smooth, tension=2] coordinates { (1.7,0) (2.4,0)};
\draw plot [smooth, tension=2] coordinates { (1.2,0.15) (1.8,0.4) (2.4,0.15)};

\end{tikzpicture}

  \\ 
$C^1_9(x_1,x_2,\dots,x_6)$ & \hspace{0.5cm} $C^1_{10}(x_1,x_2,\dots,x_6)$ & \hspace{0.5cm} $C^1_{11}(x_1,x_2,\dots,x_6)$\\
 \\
   \begin{tikzpicture}
    \node[shape=circle,fill=black] (A) at (0,0) {};
    \node[shape=circle,fill=black] (B) at (0.5,0) {};
    \node[shape=circle,fill=black] (C) at (1.2,0) {};
    \node[shape=circle,fill=black] (D) at (1.7,0) {};
    \node[shape=circle,draw=black] (E) at (2.4,0) {};
    \node[shape=circle,draw=black] (F) at (2.9,0) {} ;

     \node (G) at (0,-0.5) {$u_1$};
    \node (H) at (0.5,-0.5) {$u_1$};
    \node (I) at (1.2,-0.5) {$u_2$};
    \node (J) at (1.7,-0.5) {$u_2$};
    \node (K) at (2.4,-0.5) {$u_3$};
    \node (L) at (2.9,-0.5) {$u_3$} ;

\draw plot [smooth, tension=2] coordinates { (0,0.15) (0.6 ,0.4) (1.2,0.15)};
\draw plot [smooth, tension=2] coordinates { (0.5,0) (1.2,0)};
\draw plot [smooth, tension=2] coordinates { (0.5,0.15) (1.1,0.4) (1.7,0.15)};

\end{tikzpicture}
  
 &
 
  \hspace{0.5cm}
  \begin{tikzpicture}
    \node[shape=circle,fill=black] (A) at (0,0) {};
    \node[shape=circle,fill=black] (B) at (0.5,0) {};
    \node[shape=circle,fill=black] (C) at (1.2,0) {};
    \node[shape=circle,fill=black] (D) at (1.7,0) {};
    \node[shape=circle,fill=black] (E) at (2.4,0) {};
    \node[shape=circle,fill=black] (F) at (2.9,0) {} ;
    
     \node (G) at (0,-0.5) {$u_1$};
    \node (H) at (0.5,-0.5) {$u_1$};
    \node (I) at (1.2,-0.5) {$u_2$};
    \node (J) at (1.7,-0.5) {$u_2$};
    \node (K) at (2.4,-0.5) {$u_3$};
    \node (L) at (2.9,-0.5) {$u_3$} ;
    
\draw plot [smooth, tension=2] coordinates { (0,0.15) (0.85 ,0.4) (1.7,0.15)};
\draw plot [smooth, tension=2] coordinates { (0,0.15) (1.45 ,0.6) (2.9,0.15)};
\draw plot [smooth, tension=2] coordinates { (2.4,0.15) (1.45,0.5) (0.5,0.15)};
\draw plot [smooth, tension=2] coordinates { (2.4,0.15) (1.8,0.4) (1.2,0.15)};

\end{tikzpicture}

 &

  \hspace{0.5cm}
\begin{tikzpicture}
    \node[shape=circle,fill=black] (A) at (0,0) {};
    \node[shape=circle,fill=black] (B) at (0.5,0) {};
    \node[shape=circle,fill=black] (C) at (1.2,0) {};
    \node[shape=circle,fill=black] (D) at (1.7,0) {};
    \node[shape=circle,fill=black] (E) at (2.4,0) {};
    \node[shape=circle,fill=black] (F) at (2.9,0) {} ;

     \node (G) at (0,-0.5) {$u_1$};
    \node (H) at (0.5,-0.5) {$u_1$};
    \node (I) at (1.2,-0.5) {$u_2$};
    \node (J) at (1.7,-0.5) {$u_2$};
    \node (K) at (2.4,-0.5) {$u_3$};
    \node (L) at (2.9,-0.5) {$u_3$} ;
    
\draw plot [smooth, tension=2] coordinates { (0,0.15) (0.6 ,0.4) (1.2,0.15)};
\draw plot [smooth, tension=2] coordinates { (0,0.15) (1.2 ,0.6) (2.4,0.15)};
\draw plot [smooth, tension=2] coordinates { (0.5,0.15) (1.7,0.6) (2.9,0.15)};
\draw plot [smooth, tension=2] coordinates { (0.5,0.15) (1.1,0.4) (1.7,0.15)};
\end{tikzpicture}

  \\ 
$C^1_{12}(x_1,x_2,\dots,x_6)$ & \hspace{0.5cm} $C^1_{13}(x_1,x_2,\dots,x_6)$ & \hspace{0.5cm} $C^1_{14}(x_1,x_2,\dots,x_6)$\\
 \\

\begin{tikzpicture}
    \node[shape=circle,fill=black] (A) at (0,0) {};
    \node[shape=circle,fill=black] (B) at (0.5,0) {};
    \node[shape=circle,fill=black] (C) at (1.2,0) {};
    \node[shape=circle,fill=black] (D) at (1.7,0) {};
    \node[shape=circle,fill=black] (E) at (2.4,0) {};
    \node[shape=circle,fill=black] (F) at (2.9,0) {} ;

     \node (G) at (0,-0.5) {$u_1$};
    \node (H) at (0.5,-0.5) {$u_1$};
    \node (I) at (1.2,-0.5) {$u_2$};
    \node (J) at (1.7,-0.5) {$u_2$};
    \node (K) at (2.4,-0.5) {$u_3$};
    \node (L) at (2.9,-0.5) {$u_3$} ;
    
\draw plot [smooth, tension=2] coordinates { (0,0.15) (1.45 ,0.6) (2.9,0.15)};
\draw plot [smooth, tension=2] coordinates { (0,0.15) (1.2 ,0.5) (2.4,0.15)};
\draw plot [smooth, tension=2] coordinates { (0.5,0) (1.2,0)};
\draw plot [smooth, tension=2] coordinates { (0.5,0.15) (1.1,0.4) (1.7,0.15)};
\end{tikzpicture}
 
  &
 
\hspace{0.5cm}
\begin{tikzpicture}
    \node[shape=circle,fill=black] (A) at (0,0) {};
    \node[shape=circle,fill=black] (B) at (0.5,0) {};
    \node[shape=circle,fill=black] (C) at (1.2,0) {};
    \node[shape=circle,fill=black] (D) at (1.7,0) {};
    \node[shape=circle,fill=black] (E) at (2.4,0) {};
    \node[shape=circle,fill=black] (F) at (2.9,0) {} ;

     \node (G) at (0,-0.5) {$u_1$};
    \node (H) at (0.5,-0.5) {$u_1$};
    \node (I) at (1.2,-0.5) {$u_2$};
    \node (J) at (1.7,-0.5) {$u_2$};
    \node (K) at (2.4,-0.5) {$u_3$};
    \node (L) at (2.9,-0.5) {$u_3$} ;
    
\draw plot [smooth, tension=2] coordinates { (0,0.15) (0.6 ,0.4) (1.2,0.15)};
\draw plot [smooth, tension=2] coordinates { (0.5,0.15) (1.45 ,0.6) (2.4,0.15)};
\draw plot [smooth, tension=2] coordinates { (1.2,0.15) (1.8,0.4) (2.4,0.15)};
\draw plot [smooth, tension=2] coordinates { (1.7,0.15) (2.3,0.4) (2.9,0.15)};
\end{tikzpicture}  

&

\hspace{0.5cm}
\begin{tikzpicture}
    \node[shape=circle,fill=black] (A) at (0,0) {};
    \node[shape=circle,fill=black] (B) at (0.5,0) {};
    \node[shape=circle,fill=black] (C) at (1.2,0) {};
    \node[shape=circle,fill=black] (D) at (1.7,0) {};
    \node[shape=circle,fill=black] (E) at (2.4,0) {};
    \node[shape=circle,fill=black] (F) at (2.9,0) {} ;

     \node (G) at (0,-0.5) {$u_1$};
    \node (H) at (0.5,-0.5) {$u_1$};
    \node (I) at (1.2,-0.5) {$u_2$};
    \node (J) at (1.7,-0.5) {$u_2$};
    \node (K) at (2.4,-0.5) {$u_3$};
    \node (L) at (2.9,-0.5) {$u_3$} ;
    
\draw plot [smooth, tension=2] coordinates { (0,0.15) (0.6 ,0.4) (1.2,0.15)};
\draw plot [smooth, tension=2] coordinates { (0.5,0.15) (1.45 ,0.6) (2.4,0.15)};
\draw plot [smooth, tension=2] coordinates { (1.2,0)  (0.5,0)};
\draw plot [smooth, tension=2] coordinates { (1.7,0.15) (2.3,0.4) (2.9,0.15)};
\end{tikzpicture}

 \\ 
$C^1_{15}(x_1,x_2,\dots,x_6)$ & \hspace{0.5cm} $C^1_{16}(x_1,x_2,\dots,x_6)$ & \hspace{0.5cm} $C^1_{17}(x_1,x_2,\dots,x_6)$\\
 \\
   
 \end{tabular}
 
 \begin{tabular}{ c c c }

  \begin{tikzpicture}
    \node[shape=circle,fill=black] (A) at (0,0) {};
    \node[shape=circle,fill=black] (B) at (0.5,0) {};
    \node[shape=circle,fill=black] (C) at (1.2,0) {};
    \node[shape=circle,fill=black] (D) at (1.7,0) {};
    \node[shape=circle,fill=black] (E) at (2.4,0) {};
    \node[shape=circle,draw=black] (F) at (2.9,0) {} ;
    
     \node (G) at (0,-0.5) {$u_1$};
    \node (H) at (0.5,-0.5) {$u_1$};
    \node (I) at (1.2,-0.5) {$u_2$};
    \node (J) at (1.7,-0.5) {$u_2$};
    \node (K) at (2.4,-0.5) {$u_3$};
    \node (L) at (2.9,-0.5) {$u_3$} ;
    
\draw plot [smooth, tension=2] coordinates { (0,0.15) (0.6 ,0.4) (1.2,0.15)};
\draw plot [smooth, tension=2] coordinates { (0.5,0.15) (1.1 ,0.5) (1.7,0.15)};
\draw plot [smooth, tension=2] coordinates { (0.5,0.15) (1.45 ,0.6) (2.4,0.15)};
\draw plot [smooth, tension=2] coordinates { (1.2,0.15) (1.8,0.4) (2.4,0.15)};
\end{tikzpicture}
  &
   \hspace{0.5cm}
    \begin{tikzpicture}
    \node[shape=circle,fill=black] (A) at (0,0) {};
    \node[shape=circle,fill=black] (B) at (0.5,0) {};
    \node[shape=circle,fill=black] (C) at (1.2,0) {};
    \node[shape=circle,fill=black] (D) at (1.7,0) {};
    \node[shape=circle,fill=black] (E) at (2.4,0) {};
    \node[shape=circle,draw=black] (F) at (2.9,0) {} ;
    
     \node (G) at (0,-0.5) {$u_1$};
    \node (H) at (0.5,-0.5) {$u_1$};
    \node (I) at (1.2,-0.5) {$u_2$};
    \node (J) at (1.7,-0.5) {$u_2$};
    \node (K) at (2.4,-0.5) {$u_3$};
    \node (L) at (2.9,-0.5) {$u_3$} ;
    
\draw plot [smooth, tension=2] coordinates { (0,0.15) (0.85 ,0.6) (1.7,0.15)};
\draw plot [smooth, tension=2] coordinates { (0.5,0)  (1.2,0)};
\draw plot [smooth, tension=2] coordinates { (1.7,0)  (2.4,0)};
\draw plot [smooth, tension=2] coordinates { (1.2,0.15) (1.8,0.4) (2.4,0.15)};
\end{tikzpicture}
  
  &
   \hspace{0.5cm}
  \begin{tikzpicture}
    \node[shape=circle,fill=black] (A) at (0,0) {};
    \node[shape=circle,fill=black] (B) at (0.5,0) {};
    \node[shape=circle,fill=black] (C) at (1.2,0) {};
    \node[shape=circle,fill=black] (D) at (1.7,0) {};
    \node[shape=circle,fill=black] (E) at (2.4,0) {};
    \node[shape=circle,draw=black] (F) at (2.9,0) {} ;
    
     \node (G) at (0,-0.5) {$u_1$};
    \node (H) at (0.5,-0.5) {$u_1$};
    \node (I) at (1.2,-0.5) {$u_2$};
    \node (J) at (1.7,-0.5) {$u_2$};
    \node (K) at (2.4,-0.5) {$u_3$};
    \node (L) at (2.9,-0.5) {$u_3$} ;
    
\draw plot [smooth, tension=2] coordinates { (0,0.15) (0.6 ,0.6) (1.2,0.15)};
\draw plot [smooth, tension=2] coordinates { (0.5,0)  (1.2,0)};
\draw plot [smooth, tension=2] coordinates { (1.7,0)  (2.4,0)};
\draw plot [smooth, tension=2] coordinates { (0.5,0.15) (1.45,0.4) (2.4,0.15)};
\end{tikzpicture}
  
   \\ 
$C^1_{18}(x_1,x_2,\dots,x_6)$ & \hspace{0.5cm} $C^1_{19}(x_1,x_2,\dots,x_6)$ & \hspace{0.5cm} $C^1_{20}(x_1,x_2,\dots,x_6)$\\
 \\

  \begin{tikzpicture}
    \node[shape=circle,fill=black] (A) at (0,0) {};
    \node[shape=circle,fill=black] (B) at (0.5,0) {};
    \node[shape=circle,fill=black] (C) at (1.2,0) {};
    \node[shape=circle,fill=black] (D) at (1.7,0) {};
    \node[shape=circle,fill=black] (E) at (2.4,0) {};
    \node[shape=circle,draw=black] (F) at (2.9,0) {} ;
    
     \node (G) at (0,-0.5) {$u_1$};
    \node (H) at (0.5,-0.5) {$u_1$};
    \node (I) at (1.2,-0.5) {$u_2$};
    \node (J) at (1.7,-0.5) {$u_2$};
    \node (K) at (2.4,-0.5) {$u_3$};
    \node (L) at (2.9,-0.5) {$u_3$} ;
    
\draw plot [smooth, tension=2] coordinates { (0,0.15) (0.6 ,0.4) (1.2,0.15)};
\draw plot [smooth, tension=2] coordinates { (0.5,0)  (1.2,0)};
\draw plot [smooth, tension=2] coordinates { (1.7,0)  (2.4,0)};
\draw plot [smooth, tension=2] coordinates { (0.5,0.15) (1.1,0.4) (1.7,0.15)};
\end{tikzpicture}

  & \hspace{0.5cm}
    \begin{tikzpicture}
    \node[shape=circle,fill=black] (A) at (0,0) {};
    \node[shape=circle,fill=black] (B) at (0.5,0) {};
    \node[shape=circle,fill=black] (C) at (1.2,0) {};
    \node[shape=circle,fill=black] (D) at (1.7,0) {};
    \node[shape=circle,draw=black] (E) at (2.4,0) {};
    \node[shape=circle,draw=black] (F) at (2.9,0) {} ;
    
     \node (G) at (0,-0.5) {$u_1$};
    \node (H) at (0.5,-0.5) {$u_1$};
    \node (I) at (1.2,-0.5) {$u_2$};
    \node (J) at (1.7,-0.5) {$u_2$};
    \node (K) at (2.4,-0.5) {$u_3$};
    \node (L) at (2.9,-0.5) {$u_3$} ;

\draw plot [smooth, tension=2] coordinates { (0,0.15) (0.6 ,0.4) (1.2,0.15)};
\draw plot [smooth, tension=2] coordinates { (0.5,0) (1.2,0)};
\draw plot [smooth, tension=2] coordinates { (0.5,0.15) (1.1,0.4) (1.7,0.15)};
\draw plot [smooth, tension=2] coordinates { (0,0.15) (0.85 ,0.6) (1.7,0.15)};
\end{tikzpicture}
  &
   \hspace{0.5cm}
     \begin{tikzpicture}
    \node[shape=circle,fill=black] (A) at (0,0) {};
    \node[shape=circle,draw=black] (B) at (0.5,0) {};
    \node[shape=circle,fill=black] (C) at (1.2,0) {};
    \node[shape=circle,fill=black] (D) at (1.7,0) {};
    \node[shape=circle,fill=black] (E) at (2.4,0) {};
    \node[shape=circle,draw=black] (F) at (2.9,0) {} ;
    
     \node (G) at (0,-0.5) {$u_1$};
    \node (H) at (0.5,-0.5) {$u_1$};
    \node (I) at (1.2,-0.5) {$u_2$};
    \node (J) at (1.7,-0.5) {$u_2$};
    \node (K) at (2.4,-0.5) {$u_3$};
    \node (L) at (2.9,-0.5) {$u_3$} ;

\draw plot [smooth, tension=2] coordinates { (0,0.15) (0.6 ,0.4) (1.2,0.15)};
\draw plot [smooth, tension=2] coordinates { (2.4,0) (1.7,0)};
\draw plot [smooth, tension=2] coordinates { (2.4,0.15) (1.8,0.4) (1.2,0.15)};
\draw plot [smooth, tension=2] coordinates { (0,0.15) (0.85 ,0.6) (1.7,0.15)};
\end{tikzpicture}

     \\ 
$C^1_{21}(x_1,x_2,\dots,x_6)$ & \hspace{0.5cm} $C^1_{22}(x_1,x_2,\dots,x_6)$ & \hspace{0.5cm} $C^1_{23}(x_1,x_2,\dots,x_6)$\\
 \\
 
\begin{tikzpicture}
    \node[shape=circle,fill=black] (A) at (0,0) {};
    \node[shape=circle,fill=black] (B) at (0.5,0) {};
    \node[shape=circle,fill=black] (C) at (1.2,0) {};
    \node[shape=circle,fill=black] (D) at (1.7,0) {};
    \node[shape=circle,fill=black] (E) at (2.4,0) {};
    \node[shape=circle,fill=black] (F) at (2.9,0) {} ;

     \node (G) at (0,-0.5) {$u_1$};
    \node (H) at (0.5,-0.5) {$u_1$};
    \node (I) at (1.2,-0.5) {$u_2$};
    \node (J) at (1.7,-0.5) {$u_2$};
    \node (K) at (2.4,-0.5) {$u_3$};
    \node (L) at (2.9,-0.5) {$u_3$} ;
    
\draw plot [smooth, tension=2] coordinates { (0,0.15) (0.85 ,0.6) (1.7,0.15)};
\draw plot [smooth, tension=2] coordinates { (1.7,0.15) (2.3 ,0.4) (2.9,0.15)};
\draw plot [smooth, tension=2] coordinates { (2.9,0.15) (1.7,0.6) (0.5,0.15)};
\draw plot [smooth, tension=2] coordinates { (0.5,0)  (1.2,0)};
\draw plot [smooth, tension=2] coordinates { (1.2,0.15) (1.8,0.4) (2.4,0.15)};
\end{tikzpicture}

&

  \hspace{0.5cm}
  \begin{tikzpicture}
    \node[shape=circle,fill=black] (A) at (0,0) {};
    \node[shape=circle,fill=black] (B) at (0.5,0) {};
    \node[shape=circle,fill=black] (C) at (1.2,0) {};
    \node[shape=circle,fill=black] (D) at (1.7,0) {};
    \node[shape=circle,fill=black] (E) at (2.4,0) {};
    \node[shape=circle,fill=black] (F) at (2.9,0) {} ;

     \node (G) at (0,-0.5) {$u_1$};
    \node (H) at (0.5,-0.5) {$u_1$};
    \node (I) at (1.2,-0.5) {$u_2$};
    \node (J) at (1.7,-0.5) {$u_2$};
    \node (K) at (2.4,-0.5) {$u_3$};
    \node (L) at (2.9,-0.5) {$u_3$} ;
    
\draw plot [smooth, tension=2] coordinates { (0,0.15) (1.45 ,0.6) (2.9,0.15)};
\draw plot [smooth, tension=2] coordinates { (0.5,0.15) (1.1 ,0.4) (1.7,0.15)};
\draw plot [smooth, tension=2] coordinates { (1.2,0.15) (1.8,0.4) (2.4,0.15)};
\draw plot [smooth, tension=2] coordinates { (1.2,0.15) (2.05,0.5) (2.9,0.15)};
\draw plot [smooth, tension=2] coordinates { (1.7,0) (2.4,0)};
\end{tikzpicture}
  & 
  
    \hspace{0.5cm}
  \begin{tikzpicture}
    \node[shape=circle,fill=black] (A) at (0,0) {};
    \node[shape=circle,fill=black] (B) at (0.5,0) {};
    \node[shape=circle,fill=black] (C) at (1.2,0) {};
    \node[shape=circle,fill=black] (D) at (1.7,0) {};
    \node[shape=circle,fill=black] (E) at (2.4,0) {};
    \node[shape=circle,fill=black] (F) at (2.9,0) {} ;

     \node (G) at (0,-0.5) {$u_1$};
    \node (H) at (0.5,-0.5) {$u_1$};
    \node (I) at (1.2,-0.5) {$u_2$};
    \node (J) at (1.7,-0.5) {$u_2$};
    \node (K) at (2.4,-0.5) {$u_3$};
    \node (L) at (2.9,-0.5) {$u_3$} ;
    
\draw plot [smooth, tension=2] coordinates { (0,0.15) (0.6 ,0.4) (1.2,0.15)};
\draw plot [smooth, tension=2] coordinates { (1.7,0.15) (2.3 ,0.4) (2.9,0.15)};
\draw plot [smooth, tension=2] coordinates { (0.5,0.15) (1.45,0.6) (2.4,0.15)};
\draw plot [smooth, tension=2] coordinates { (0.5,0) (1.2,0)};
\draw plot [smooth, tension=2] coordinates { (1.7,0) (2.4,0)};
\end{tikzpicture}

       \\ 
$C^1_{24}(x_1,x_2,\dots,x_6)$ & \hspace{0.5cm} $C^1_{25}(x_1,x_2,\dots,x_6)$ & \hspace{0.5cm} $C^1_{26}(x_1,x_2,\dots,x_6)$\\
 \\

  \begin{tikzpicture}
    \node[shape=circle,fill=black] (A) at (0,0) {};
    \node[shape=circle,fill=black] (B) at (0.5,0) {};
    \node[shape=circle,fill=black] (C) at (1.2,0) {};
    \node[shape=circle,fill=black] (D) at (1.7,0) {};
    \node[shape=circle,fill=black] (E) at (2.4,0) {};
    \node[shape=circle,fill=black] (F) at (2.9,0) {} ;

     \node (G) at (0,-0.5) {$u_1$};
    \node (H) at (0.5,-0.5) {$u_1$};
    \node (I) at (1.2,-0.5) {$u_2$};
    \node (J) at (1.7,-0.5) {$u_2$};
    \node (K) at (2.4,-0.5) {$u_3$};
    \node (L) at (2.9,-0.5) {$u_3$} ;
    
\draw plot [smooth, tension=2] coordinates { (0,0.15) (0.6 ,0.4) (1.2,0.15)};
\draw plot [smooth, tension=2] coordinates { (0,0.15) (1.2 ,0.6) (2.4,0.15)};
\draw plot [smooth, tension=2] coordinates { (0.5,0.15) (1.45,0.5) (2.4,0.15)};
\draw plot [smooth, tension=2] coordinates { (1.7,0.15) (2.3,0.4) (2.9,0.15)};
\draw plot [smooth, tension=2] coordinates { (0.5,0) (1.2,0)};
\end{tikzpicture}
 
 &
      \hspace{0.5cm}
  \begin{tikzpicture}
    \node[shape=circle,fill=black] (A) at (0,0) {};
    \node[shape=circle,fill=black] (B) at (0.5,0) {};
    \node[shape=circle,fill=black] (C) at (1.2,0) {};
    \node[shape=circle,fill=black] (D) at (1.7,0) {};
    \node[shape=circle,fill=black] (E) at (2.4,0) {};
    \node[shape=circle,draw=black] (F) at (2.9,0) {} ;

     \node (G) at (0,-0.5) {$u_1$};
    \node (H) at (0.5,-0.5) {$u_1$};
    \node (I) at (1.2,-0.5) {$u_2$};
    \node (J) at (1.7,-0.5) {$u_2$};
    \node (K) at (2.4,-0.5) {$u_3$};
    \node (L) at (2.9,-0.5) {$u_3$} ;
    
\draw plot [smooth, tension=2] coordinates { (0,0.15) (0.6 ,0.4) (1.2,0.15)};
\draw plot [smooth, tension=2] coordinates { (0,0.15) (1.2 ,0.6) (2.4,0.15)};
\draw plot [smooth, tension=2] coordinates { (0.5,0.15) (1.1,0.5) (1.7,0.15)};
\draw plot [smooth, tension=2] coordinates { (0.5,0) (1.2,0)};
\draw plot [smooth, tension=2] coordinates { (1.7,0) (2.4,0)};
\end{tikzpicture}
 
 & 
 \hspace{0.5cm}
\begin{tikzpicture}
    \node[shape=circle,fill=black] (A) at (0,0) {};
    \node[shape=circle,fill=black] (B) at (0.5,0) {};
    \node[shape=circle,fill=black] (C) at (1.2,0) {};
    \node[shape=circle,fill=black] (D) at (1.7,0) {};
    \node[shape=circle,fill=black] (E) at (2.4,0) {};
    \node[shape=circle,fill=black] (F) at (2.9,0) {} ;

     \node (G) at (0,-0.5) {$u_1$};
    \node (H) at (0.5,-0.5) {$u_1$};
    \node (I) at (1.2,-0.5) {$u_2$};
    \node (J) at (1.7,-0.5) {$u_2$};
    \node (K) at (2.4,-0.5) {$u_3$};
    \node (L) at (2.9,-0.5) {$u_3$} ;

 \draw plot [smooth, tension=2] coordinates { (0.5,0) (1.2,0) };
     \draw plot [smooth, tension=2] coordinates { (1.7,0) (2.3,0) };
     \draw plot [smooth, tension=2] coordinates { (0,0.15) (1.45,0.7) (2.9,0.15) };
     \draw plot [smooth, tension=2] coordinates { (0,0.15) (0.85,0.5) (1.7,0.15) };
     \draw plot [smooth, tension=2] coordinates { (1.2,0.15) (2.05,0.5) (2.9,0.15) };
      \draw plot [smooth, tension=2] coordinates { (0.5,0.15) (1.4,0.3) (2.3,0.15) };    
\end{tikzpicture}  
 \\ 
 $C^1_{27}(x_1,x_2,\dots,x_6)$ & \hspace{0.5cm} $C^1_{28}(x_1,x_2,\dots,x_6)$ & \hspace{0.5cm} $C^1_{29}(x_1,x_2,\dots,x_6)$\\
 \\
 
 \begin{tikzpicture}
    \node[shape=circle,fill=black] (A) at (0,0) {};
    \node[shape=circle,fill=black] (B) at (0.5,0) {};
    \node[shape=circle,fill=black] (C) at (1.2,0) {};
    \node[shape=circle,fill=black] (D) at (1.7,0) {};
    \node[shape=circle,fill=black] (E) at (2.4,0) {};
    \node[shape=circle,fill=black] (F) at (2.9,0) {} ;

     \node (G) at (0,-0.5) {$u_1$};
    \node (H) at (0.5,-0.5) {$u_1$};
    \node (I) at (1.2,-0.5) {$u_2$};
    \node (J) at (1.7,-0.5) {$u_2$};
    \node (K) at (2.4,-0.5) {$u_3$};
    \node (L) at (2.9,-0.5) {$u_3$} ;
    
\draw plot [smooth, tension=2] coordinates { (0,0.15) (1.45 ,0.7) (2.9,0.15)};
\draw plot [smooth, tension=2] coordinates { (0,0.15) (0.6 ,0.55) (1.2,0.15)};
\draw plot [smooth, tension=2] coordinates { (1.7,0.15) (2.3,0.55) (2.9,0.15)};
\draw plot [smooth, tension=2] coordinates { (0.5,0.15) (1.45,0.4) (2.4,0.15)};
\draw plot [smooth, tension=2] coordinates { (0.5,0) (1.2,0)};
\draw plot [smooth, tension=2] coordinates { (1.7,0) (2.4,0)};
\end{tikzpicture}
 
 & 
 
 &
 \\ 
 $C^1_{30}(x_1,x_2,\dots,x_6)$ & &
 \\
\end{tabular}
\end{center}

\medskip

One can prove that these are all the possible type of configurations, up to symmetries, when the second neighborhood of $v$ has $6$ vertices by exhaustive case analysis, in a similar way as we did for $\mathcal C_5$ in Appendix~\ref{app:conf}. The main difference is that here we allow for cycles of length $4$ in $G$. While this implies that all configurations in $\mathcal C_5$ are also in $\mathcal C_4$, the set $\mathcal C_4$ also contains other configurations that either induce a $C_4$ in $G[W]$ or in $G[W\cup\{u_i\}]$, for some $i\in \{1,2,3\}$. Examples of the former are $C^1_{22}(x_1,x_2,\dots,x_6)$ and $C^1_{23}(x_1,x_2,\dots,x_6)$, and examples of the latter are $C^1_4(x_1,x_2,\dots,x_6)$ and $C^1_{15}(x_1,x_2,\dots,x_6)$.

\newpage
\subsection{}
The following configurations are obtained if the second neighborhood  of a vertex has size five.

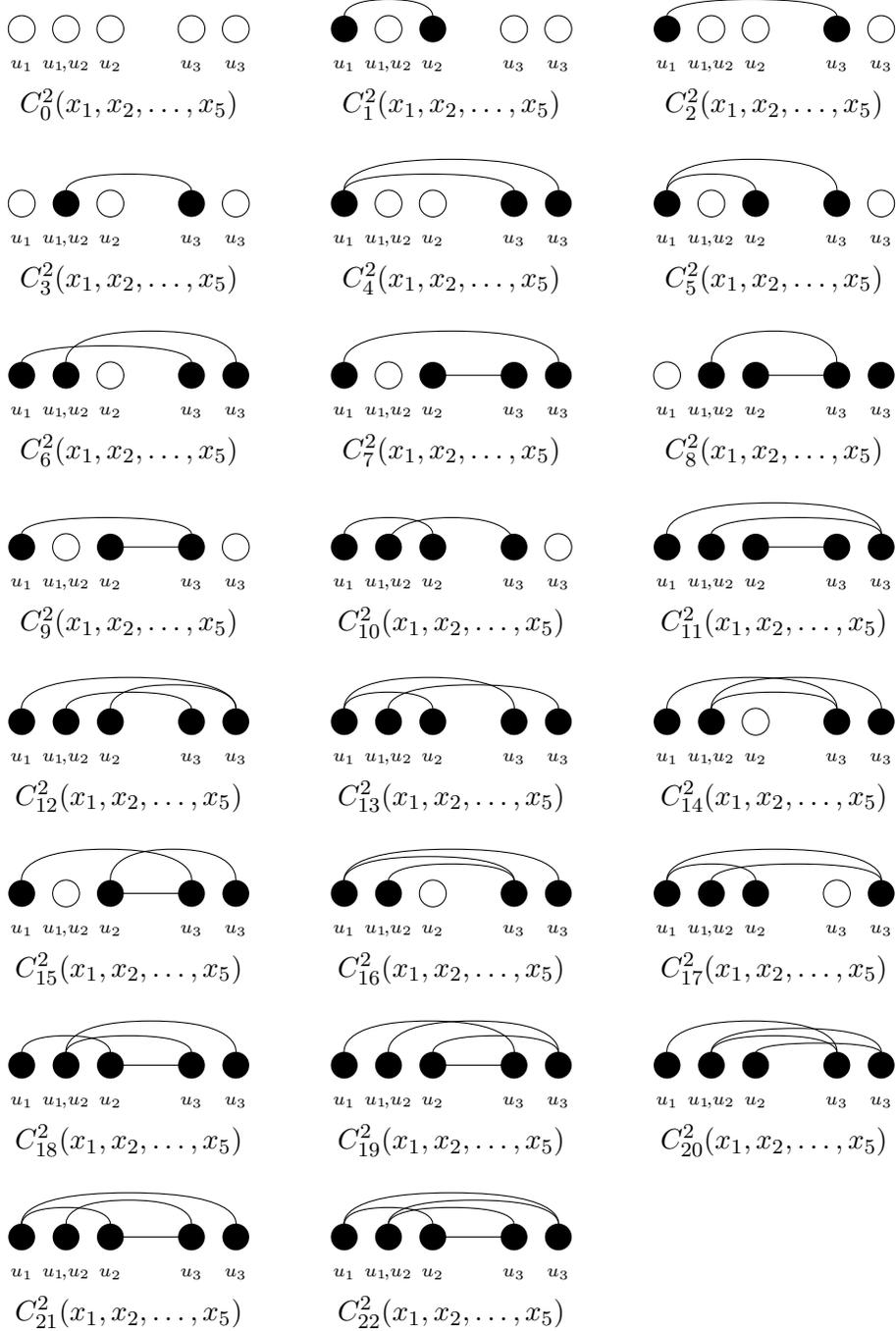
\begin{center}
\begin{tabular}{ c c c }

\begin{tikzpicture}
    \node[shape=circle,draw=black] (A) at (0,0) {};
    \node[shape=circle,draw=black] (B) at (0.6,0) {};
    \node[shape=circle,draw=black] (C) at (1.2,0) {};
    \node[shape=circle,draw=black] (D) at (2.3,0) {};
    \node[shape=circle,draw=black] (E) at (2.9,0) {} ;
    
    {\tiny
    \node (G) at (0,-0.5) {$u_1$};
    \node (H) at (0.6,-0.5) {$u_1\!,\!u_2$};
    \node (I) at (1.2,-0.5) {$u_2$};
    \node (J) at (2.3,-0.5) {$u_3$};
    \node (K) at (2.9,-0.5) {$u_3$};
    }
\end{tikzpicture}

 & 
   
  \hspace{0.5cm}
\begin{tikzpicture}
    \node[shape=circle,fill=black] (A) at (0,0) {};
    \node[shape=circle,draw=black] (B) at (0.6,0) {};
    \node[shape=circle,fill=black] (C) at (1.2,0) {};
    \node[shape=circle,draw=black] (D) at (2.3,0) {};
    \node[shape=circle,draw=black] (E) at (2.9,0) {} ;
    
    {\tiny
    \node (G) at (0,-0.5) {$u_1$};
    \node (H) at (0.6,-0.5) {$u_1\!,\!u_2$};
    \node (I) at (1.2,-0.5) {$u_2$};
    \node (J) at (2.3,-0.5) {$u_3$};
    \node (K) at (2.9,-0.5) {$u_3$};
    }

\draw plot [smooth, tension=2] coordinates { (0,0.15) (0.6,0.4) (1.2,0.15)};
\end{tikzpicture}

 &
   
  \hspace{0.5cm}
\begin{tikzpicture}
    \node[shape=circle,fill=black] (A) at (0,0) {};
    \node[shape=circle,draw=black] (B) at (0.6,0) {};
    \node[shape=circle,draw=black] (C) at (1.2,0) {};
    \node[shape=circle,fill=black] (D) at (2.3,0) {};
    \node[shape=circle,draw=black] (E) at (2.9,0) {} ;
    
    {\tiny
    \node (G) at (0,-0.5) {$u_1$};
    \node (H) at (0.6,-0.5) {$u_1\!,\!u_2$};
    \node (I) at (1.2,-0.5) {$u_2$};
    \node (J) at (2.3,-0.5) {$u_3$};
    \node (K) at (2.9,-0.5) {$u_3$};
    }

\draw plot [smooth, tension=2] coordinates { (0,0.15) (1.15,0.4) (2.3,0.15)};
\end{tikzpicture}

 \\ 
 $C^2_{0}(x_1,x_2,\dots,x_5)$ & \hspace{0.5cm} $C^2_{1}(x_1,x_2,\dots,x_5)$ & \hspace{0.5cm} $C^2_{2}(x_1,x_2,\dots,x_5)$\\
 \\

 \end{tabular}
 \begin{tabular}{ c c c }

\begin{tikzpicture}
    \node[shape=circle,draw=black] (A) at (0,0) {};
    \node[shape=circle,fill=black] (B) at (0.6,0) {};
    \node[shape=circle,draw=black] (C) at (1.2,0) {};
    \node[shape=circle,fill=black] (D) at (2.3,0) {};
    \node[shape=circle,draw=black] (E) at (2.9,0) {} ;
    
    {\tiny
    \node (G) at (0,-0.5) {$u_1$};
    \node (H) at (0.6,-0.5) {$u_1\!,\!u_2$};
    \node (I) at (1.2,-0.5) {$u_2$};
    \node (J) at (2.3,-0.5) {$u_3$};
    \node (K) at (2.9,-0.5) {$u_3$};
    }
    
\draw plot [smooth, tension=2] coordinates { (0.6,0.15) (1.45,0.4) (2.3,0.15)};
\end{tikzpicture}

 & 
   
  \hspace{0.5cm}
\begin{tikzpicture}
    \node[shape=circle,fill=black] (A) at (0,0) {};
    \node[shape=circle,draw=black] (B) at (0.6,0) {};
    \node[shape=circle,draw=black] (C) at (1.2,0) {};
    \node[shape=circle,fill=black] (D) at (2.3,0) {};
    \node[shape=circle,fill=black] (E) at (2.9,0) {} ;
    
    {\tiny
    \node (G) at (0,-0.5) {$u_1$};
    \node (H) at (0.6,-0.5) {$u_1\!,\!u_2$};
    \node (I) at (1.2,-0.5) {$u_2$};
    \node (J) at (2.3,-0.5) {$u_3$};
    \node (K) at (2.9,-0.5) {$u_3$};
    }

\draw plot [smooth, tension=2] coordinates { (0,0.15) (1.15,0.4) (2.3,0.15)};
\draw plot [smooth, tension=2] coordinates { (0,0.15) (1.45,0.6) (2.9,0.15)};
\end{tikzpicture}

 &
   
  \hspace{0.5cm}
\begin{tikzpicture}
    \node[shape=circle,fill=black] (A) at (0,0) {};
    \node[shape=circle,draw=black] (B) at (0.6,0) {};
    \node[shape=circle,fill=black] (C) at (1.2,0) {};
    \node[shape=circle,fill=black] (D) at (2.3,0) {};
    \node[shape=circle,draw=black] (E) at (2.9,0) {} ;
    
    {\tiny
    \node (G) at (0,-0.5) {$u_1$};
    \node (H) at (0.6,-0.5) {$u_1\!,\!u_2$};
    \node (I) at (1.2,-0.5) {$u_2$};
    \node (J) at (2.3,-0.5) {$u_3$};
    \node (K) at (2.9,-0.5) {$u_3$};
    }

\draw plot [smooth, tension=2] coordinates { (0,0.15) (1.15,0.6) (2.3,0.15)};
\draw plot [smooth, tension=2] coordinates { (0,0.15) (0.6,0.4) (1.2,0.15)};
\end{tikzpicture}

 \\ 
 $C^2_{3}(x_1,x_2,\dots,x_5)$ & \hspace{0.5cm} $C^2_{4}(x_1,x_2,\dots,x_5)$ & \hspace{0.5cm} $C^2_{5}(x_1,x_2,\dots,x_5)$\\
 \\

\begin{tikzpicture}
    \node[shape=circle,fill=black] (A) at (0,0) {};
    \node[shape=circle,fill=black] (B) at (0.6,0) {};
    \node[shape=circle,draw=black] (C) at (1.2,0) {};
    \node[shape=circle,fill=black] (D) at (2.3,0) {};
    \node[shape=circle,fill=black] (E) at (2.9,0) {} ;
    
    {\tiny
    \node (G) at (0,-0.5) {$u_1$};
    \node (H) at (0.6,-0.5) {$u_1\!,\!u_2$};
    \node (I) at (1.2,-0.5) {$u_2$};
    \node (J) at (2.3,-0.5) {$u_3$};
    \node (K) at (2.9,-0.5) {$u_3$};
    }

\draw plot [smooth, tension=2] coordinates { (0,0.15) (1.15,0.4) (2.3,0.15)};
\draw plot [smooth, tension=2] coordinates { (0.6,0.15) (1.75,0.6) (2.9,0.15)};
\end{tikzpicture}

 & 
   
  \hspace{0.5cm}
\begin{tikzpicture}
    \node[shape=circle,fill=black] (A) at (0,0) {};
    \node[shape=circle,draw=black] (B) at (0.6,0) {};
    \node[shape=circle,fill=black] (C) at (1.2,0) {};
    \node[shape=circle,fill=black] (D) at (2.3,0) {};
    \node[shape=circle,fill=black] (E) at (2.9,0) {} ;
    
    {\tiny
    \node (G) at (0,-0.5) {$u_1$};
    \node (H) at (0.6,-0.5) {$u_1\!,\!u_2$};
    \node (I) at (1.2,-0.5) {$u_2$};
    \node (J) at (2.3,-0.5) {$u_3$};
    \node (K) at (2.9,-0.5) {$u_3$};
    }

\draw plot [smooth, tension=2] coordinates { (0,0.15) (1.45,0.6) (2.9,0.15)};
\draw plot [smooth, tension=2] coordinates { (1.2,0)  (2.3,0)};
\end{tikzpicture}

 &
   
  \hspace{0.5cm}
\begin{tikzpicture}
    \node[shape=circle,draw=black] (A) at (0,0) {};
    \node[shape=circle,fill=black] (B) at (0.6,0) {};
    \node[shape=circle,fill=black] (C) at (1.2,0) {};
    \node[shape=circle,fill=black] (D) at (2.3,0) {};
    \node[shape=circle,fill=black] (E) at (2.9,0) {} ;
    
    {\tiny
    \node (G) at (0,-0.5) {$u_1$};
    \node (H) at (0.6,-0.5) {$u_1\!,\!u_2$};
    \node (I) at (1.2,-0.5) {$u_2$};
    \node (J) at (2.3,-0.5) {$u_3$};
    \node (K) at (2.9,-0.5) {$u_3$};
    }

\draw plot [smooth, tension=2] coordinates { (0.6,0.15) (1.45,0.6) (2.3,0.15)};
\draw plot [smooth, tension=2] coordinates { (1.2,0)  (2.3,0)};
\end{tikzpicture}

 \\ 
 $C^2_{6}(x_1,x_2,\dots,x_5)$ & \hspace{0.5cm} $C^2_{7}(x_1,x_2,\dots,x_5)$ & \hspace{0.5cm} $C^2_{8}(x_1,x_2,\dots,x_5)$\\
 \\

\begin{tikzpicture}
    \node[shape=circle,fill=black] (A) at (0,0) {};
    \node[shape=circle,draw=black] (B) at (0.6,0) {};
    \node[shape=circle,fill=black] (C) at (1.2,0) {};
    \node[shape=circle,fill=black] (D) at (2.3,0) {};
    \node[shape=circle,draw=black] (E) at (2.9,0) {} ;
    
    {\tiny
    \node (G) at (0,-0.5) {$u_1$};
    \node (H) at (0.6,-0.5) {$u_1\!,\!u_2$};
    \node (I) at (1.2,-0.5) {$u_2$};
    \node (J) at (2.3,-0.5) {$u_3$};
    \node (K) at (2.9,-0.5) {$u_3$};
    }

\draw plot [smooth, tension=2] coordinates { (0,0.15) (1.15,0.4) (2.3,0.15)};
\draw plot [smooth, tension=2] coordinates { (1.2,0)  (2.3,0)};
\end{tikzpicture}

 & 
   
  \hspace{0.5cm}
\begin{tikzpicture}
    \node[shape=circle,fill=black] (A) at (0,0) {};
    \node[shape=circle,fill=black] (B) at (0.6,0) {};
    \node[shape=circle,fill=black] (C) at (1.2,0) {};
    \node[shape=circle,fill=black] (D) at (2.3,0) {};
    \node[shape=circle,draw=black] (E) at (2.9,0) {} ;
    
    {\tiny
    \node (G) at (0,-0.5) {$u_1$};
    \node (H) at (0.6,-0.5) {$u_1\!,\!u_2$};
    \node (I) at (1.2,-0.5) {$u_2$};
    \node (J) at (2.3,-0.5) {$u_3$};
    \node (K) at (2.9,-0.5) {$u_3$};
    }

\draw plot [smooth, tension=2] coordinates { (0,0.15) (0.6,0.4) (1.2,0.15)};
\draw plot [smooth, tension=2] coordinates { (0.6,0) (1.45,0.4) (2.3,0)};
\end{tikzpicture}

 &
   
  \hspace{0.5cm}
\begin{tikzpicture}
    \node[shape=circle,fill=black] (A) at (0,0) {};
    \node[shape=circle,fill=black] (B) at (0.6,0) {};
    \node[shape=circle,fill=black] (C) at (1.2,0) {};
    \node[shape=circle,fill=black] (D) at (2.3,0) {};
    \node[shape=circle,fill=black] (E) at (2.9,0) {} ;
    
    {\tiny
    \node (G) at (0,-0.5) {$u_1$};
    \node (H) at (0.6,-0.5) {$u_1\!,\!u_2$};
    \node (I) at (1.2,-0.5) {$u_2$};
    \node (J) at (2.3,-0.5) {$u_3$};
    \node (K) at (2.9,-0.5) {$u_3$};
    }

\draw plot [smooth, tension=2] coordinates { (0,0.15) (1.45,0.6) (2.9,0.15)};
\draw plot [smooth, tension=2] coordinates { (0.6,0.15) (1.75,0.4) (2.9,0.15)};
\draw plot [smooth, tension=2] coordinates { (1.2,0)  (2.3,0)};
\end{tikzpicture}

 \\ 
 $C^2_{9}(x_1,x_2,\dots,x_5)$ & \hspace{0.5cm} $C^2_{10}(x_1,x_2,\dots,x_5)$ & \hspace{0.5cm} $C^2_{11}(x_1,x_2,\dots,x_5)$\\
 \\

 \end{tabular}
 
 \begin{tabular}{ c c c }

\begin{tikzpicture}
    \node[shape=circle,fill=black] (A) at (0,0) {};
    \node[shape=circle,fill=black] (B) at (0.6,0) {};
    \node[shape=circle,fill=black] (C) at (1.2,0) {};
    \node[shape=circle,fill=black] (D) at (2.3,0) {};
    \node[shape=circle,fill=black] (E) at (2.9,0) {} ;
    
    {\tiny
    \node (G) at (0,-0.5) {$u_1$};
    \node (H) at (0.6,-0.5) {$u_1\!,\!u_2$};
    \node (I) at (1.2,-0.5) {$u_2$};
    \node (J) at (2.3,-0.5) {$u_3$};
    \node (K) at (2.9,-0.5) {$u_3$};
    }

\draw plot [smooth, tension=2] coordinates { (0,0.15) (1.45,0.6) (2.9,0.15)};
\draw plot [smooth, tension=2] coordinates { (1.2,0.15) (2.05,0.5) (2.9,0.15)};
\draw plot [smooth, tension=2] coordinates { (0.6,0.15) (1.45,0.4) (2.3,0.15)};
\end{tikzpicture}

 & 
   
  \hspace{0.5cm}
\begin{tikzpicture}
    \node[shape=circle,fill=black] (A) at (0,0) {};
    \node[shape=circle,fill=black] (B) at (0.6,0) {};
    \node[shape=circle,fill=black] (C) at (1.2,0) {};
    \node[shape=circle,fill=black] (D) at (2.3,0) {};
    \node[shape=circle,fill=black] (E) at (2.9,0) {} ;
    
    {\tiny
    \node (G) at (0,-0.5) {$u_1$};
    \node (H) at (0.6,-0.5) {$u_1\!,\!u_2$};
    \node (I) at (1.2,-0.5) {$u_2$};
    \node (J) at (2.3,-0.5) {$u_3$};
    \node (K) at (2.9,-0.5) {$u_3$};
    }

\draw plot [smooth, tension=2] coordinates { (0,0.15) (0.6,0.4) (1.2,0.15)};
\draw plot [smooth, tension=2] coordinates { (0,0.15) (1.15,0.6) (2.3,0.15)};
\draw plot [smooth, tension=2] coordinates { (0.6,0.15) (1.75,0.5) (2.9,0.15)};
\end{tikzpicture}

 &
   
  \hspace{0.5cm}
\begin{tikzpicture}
    \node[shape=circle,fill=black] (A) at (0,0) {};
    \node[shape=circle,fill=black] (B) at (0.6,0) {};
    \node[shape=circle,draw=black] (C) at (1.2,0) {};
    \node[shape=circle,fill=black] (D) at (2.3,0) {};
    \node[shape=circle,fill=black] (E) at (2.9,0) {} ;
    
    {\tiny
    \node (G) at (0,-0.5) {$u_1$};
    \node (H) at (0.6,-0.5) {$u_1\!,\!u_2$};
    \node (I) at (1.2,-0.5) {$u_2$};
    \node (J) at (2.3,-0.5) {$u_3$};
    \node (K) at (2.9,-0.5) {$u_3$};
    }

\draw plot [smooth, tension=2] coordinates { (0,0.15) (1.15,0.6) (2.3,0.15)};
\draw plot [smooth, tension=2] coordinates { (0.6,0.15) (1.75,0.6) (2.9,0.15)};
\draw plot [smooth, tension=2] coordinates { (0.6,0.15) (1.45,0.4) (2.3,0.15)};
\end{tikzpicture}

 \\ 
 $C^2_{12}(x_1,x_2,\dots,x_5)$ & \hspace{0.5cm} $C^2_{13}(x_1,x_2,\dots,x_5)$ & \hspace{0.5cm} $C^2_{14}(x_1,x_2,\dots,x_5)$\\
 \\

\begin{tikzpicture}
    \node[shape=circle,fill=black] (A) at (0,0) {};
    \node[shape=circle,draw=black] (B) at (0.6,0) {};
    \node[shape=circle,fill=black] (C) at (1.2,0) {};
    \node[shape=circle,fill=black] (D) at (2.3,0) {};
    \node[shape=circle,fill=black] (E) at (2.9,0) {} ;
    
    {\tiny
    \node (G) at (0,-0.5) {$u_1$};
    \node (H) at (0.6,-0.5) {$u_1\!,\!u_2$};
    \node (I) at (1.2,-0.5) {$u_2$};
    \node (J) at (2.3,-0.5) {$u_3$};
    \node (K) at (2.9,-0.5) {$u_3$};
    }

\draw plot [smooth, tension=2] coordinates { (0,0.15) (1.15,0.6) (2.3,0.15)};
\draw plot [smooth, tension=2] coordinates { (1.2,0.15) (2.05,0.6) (2.9,0.15)};
\draw plot [smooth, tension=2] coordinates { (1.2,0)(2.3,0)};
\end{tikzpicture}

 & 
   
  \hspace{0.5cm}
\begin{tikzpicture}
    \node[shape=circle,fill=black] (A) at (0,0) {};
    \node[shape=circle,fill=black] (B) at (0.6,0) {};
    \node[shape=circle,draw=black] (C) at (1.2,0) {};
    \node[shape=circle,fill=black] (D) at (2.3,0) {};
    \node[shape=circle,fill=black] (E) at (2.9,0) {} ;
    
    {\tiny
    \node (G) at (0,-0.5) {$u_1$};
    \node (H) at (0.6,-0.5) {$u_1\!,\!u_2$};
    \node (I) at (1.2,-0.5) {$u_2$};
    \node (J) at (2.3,-0.5) {$u_3$};
    \node (K) at (2.9,-0.5) {$u_3$};
    }

\draw plot [smooth, tension=2] coordinates { (0.6,0.15) (1.45,0.4) (2.3,0.15)};
\draw plot [smooth, tension=2] coordinates { (0,0.15) (1.45,0.6) (2.9,0.15)};
\draw plot [smooth, tension=2] coordinates { (0,0.15) (1.15,0.5) (2.3,0.15)};
\end{tikzpicture}

 &
   
  \hspace{0.5cm}
\begin{tikzpicture}
    \node[shape=circle,fill=black] (A) at (0,0) {};
    \node[shape=circle,fill=black] (B) at (0.6,0) {};
    \node[shape=circle,fill=black] (C) at (1.2,0) {};
    \node[shape=circle,draw=black] (D) at (2.3,0) {};
    \node[shape=circle,fill=black] (E) at (2.9,0) {} ;
    
    {\tiny
    \node (G) at (0,-0.5) {$u_1$};
    \node (H) at (0.6,-0.5) {$u_1\!,\!u_2$};
    \node (I) at (1.2,-0.5) {$u_2$};
    \node (J) at (2.3,-0.5) {$u_3$};
    \node (K) at (2.9,-0.5) {$u_3$};
    }

\draw plot [smooth, tension=2] coordinates { (0,0.15) (1.45,0.6) (2.9,0.15)};
\draw plot [smooth, tension=2] coordinates { (0,0.15) (0.6,0.4) (1.2,0.15)};
\draw plot [smooth, tension=2] coordinates { (0.6,0.15) (1.75,0.4) (2.9,0.15)};
\end{tikzpicture}

 \\ 
 $C^2_{15}(x_1,x_2,\dots,x_5)$ & \hspace{0.5cm} $C^2_{16}(x_1,x_2,\dots,x_5)$ & \hspace{0.5cm} $C^2_{17}(x_1,x_2,\dots,x_5)$\\
 \\

\begin{tikzpicture}
    \node[shape=circle,fill=black] (A) at (0,0) {};
    \node[shape=circle,fill=black] (B) at (0.6,0) {};
    \node[shape=circle,fill=black] (C) at (1.2,0) {};
    \node[shape=circle,fill=black] (D) at (2.3,0) {};
    \node[shape=circle,fill=black] (E) at (2.9,0) {} ;
    
    {\tiny
    \node (G) at (0,-0.5) {$u_1$};
    \node (H) at (0.6,-0.5) {$u_1\!,\!u_2$};
    \node (I) at (1.2,-0.5) {$u_2$};
    \node (J) at (2.3,-0.5) {$u_3$};
    \node (K) at (2.9,-0.5) {$u_3$};
    }

\draw plot [smooth, tension=2] coordinates { (0,0.15) (0.6,0.4) (1.2,0.15)};
\draw plot [smooth, tension=2] coordinates { (1.2,0) (2.3,0)};
\draw plot [smooth, tension=2] coordinates { (0.6,0.15) (1.75,0.6) (2.9,0.15)};
\draw plot [smooth, tension=2] coordinates { (0.6,0.15) (1.45,0.4) (2.3,0.15)};
\end{tikzpicture}

 & 
   
  \hspace{0.5cm}
\begin{tikzpicture}
    \node[shape=circle,fill=black] (A) at (0,0) {};
    \node[shape=circle,fill=black] (B) at (0.6,0) {};
    \node[shape=circle,fill=black] (C) at (1.2,0) {};
    \node[shape=circle,fill=black] (D) at (2.3,0) {};
    \node[shape=circle,fill=black] (E) at (2.9,0) {} ;
    
    {\tiny
    \node (G) at (0,-0.5) {$u_1$};
    \node (H) at (0.6,-0.5) {$u_1\!,\!u_2$};
    \node (I) at (1.2,-0.5) {$u_2$};
    \node (J) at (2.3,-0.5) {$u_3$};
    \node (K) at (2.9,-0.5) {$u_3$};
    }

\draw plot [smooth, tension=2] coordinates { (1.2,0) (2.3,0)};
\draw plot [smooth, tension=2] coordinates { (0.6,0.15) (1.75,0.6) (2.9,0.15)};
\draw plot [smooth, tension=2] coordinates { (0,0.15) (1.15,0.6) (2.3,0.15)};
\draw plot [smooth, tension=2] coordinates { (1.2,0.15) (2.05,0.4) (2.9,0.15)};
\end{tikzpicture}

 &
   
  \hspace{0.5cm}
\begin{tikzpicture}
    \node[shape=circle,fill=black] (A) at (0,0) {};
    \node[shape=circle,fill=black] (B) at (0.6,0) {};
    \node[shape=circle,fill=black] (C) at (1.2,0) {};
    \node[shape=circle,fill=black] (D) at (2.3,0) {};
    \node[shape=circle,fill=black] (E) at (2.9,0) {} ;
    
    {\tiny
    \node (G) at (0,-0.5) {$u_1$};
    \node (H) at (0.6,-0.5) {$u_1\!,\!u_2$};
    \node (I) at (1.2,-0.5) {$u_2$};
    \node (J) at (2.3,-0.5) {$u_3$};
    \node (K) at (2.9,-0.5) {$u_3$};
    }
    
\draw plot [smooth, tension=2] coordinates { (0,0.15) (1.15,0.6) (2.3,0.15)};
\draw plot [smooth, tension=2] coordinates { (0.6,0.15) (1.75,0.5) (2.9,0.15)};
\draw plot [smooth, tension=2] coordinates { (0.6,0.15) (1.45,0.4) (2.3,0.15)};
\draw plot [smooth, tension=2] coordinates { (1.2,0.15) (2.05,0.3) (2.9,0.15)};
\end{tikzpicture}

 \\ 
 $C^2_{18}(x_1,x_2,\dots,x_5)$ & \hspace{0.5cm} $C^2_{19}(x_1,x_2,\dots,x_5)$ & \hspace{0.5cm} $C^2_{20}(x_1,x_2,\dots,x_5)$\\
 \\

\begin{tikzpicture}
    \node[shape=circle,fill=black] (A) at (0,0) {};
    \node[shape=circle,fill=black] (B) at (0.6,0) {};
    \node[shape=circle,fill=black] (C) at (1.2,0) {};
    \node[shape=circle,fill=black] (D) at (2.3,0) {};
    \node[shape=circle,fill=black] (E) at (2.9,0) {} ;
    
    {\tiny
    \node (G) at (0,-0.5) {$u_1$};
    \node (H) at (0.6,-0.5) {$u_1\!,\!u_2$};
    \node (I) at (1.2,-0.5) {$u_2$};
    \node (J) at (2.3,-0.5) {$u_3$};
    \node (K) at (2.9,-0.5) {$u_3$};
    }

\draw plot [smooth, tension=2] coordinates { (0,0.15) (1.45,0.6) (2.9,0.15)};
\draw plot [smooth, tension=2] coordinates { (0,0.15) (0.6,0.4) (1.2,0.15)};
\draw plot [smooth, tension=2] coordinates { (0.6,0.15) (1.45,0.5) (2.3,0.15)};
\draw plot [smooth, tension=2] coordinates { (1.2,0) (2.3,0)};
\end{tikzpicture}

 & 
   
  \hspace{0.5cm}
\begin{tikzpicture}
    \node[shape=circle,fill=black] (A) at (0,0) {};
    \node[shape=circle,fill=black] (B) at (0.6,0) {};
    \node[shape=circle,fill=black] (C) at (1.2,0) {};
    \node[shape=circle,fill=black] (D) at (2.3,0) {};
    \node[shape=circle,fill=black] (E) at (2.9,0) {} ;
    
    {\tiny
    \node (G) at (0,-0.5) {$u_1$};
    \node (H) at (0.6,-0.5) {$u_1\!,\!u_2$};
    \node (I) at (1.2,-0.5) {$u_2$};
    \node (J) at (2.3,-0.5) {$u_3$};
    \node (K) at (2.9,-0.5) {$u_3$};
    }

\draw plot [smooth, tension=2] coordinates { (0,0.15) (1.45,0.6) (2.9,0.15)};
\draw plot [smooth, tension=2] coordinates { (0,0.15) (0.6,0.4) (1.2,0.15)};
\draw plot [smooth, tension=2] coordinates { (0.6,0.15) (1.45,0.4) (2.3,0.15)};
\draw plot [smooth, tension=2] coordinates { (0.6,0.15) (1.75,0.5) (2.9,0.15)};
\draw plot [smooth, tension=2] coordinates { (1.2,0) (2.3,0)};

\end{tikzpicture}

 &

 \\ 
 $C^2_{21}(x_1,x_2,\dots,x_5)$ & \hspace{0.5cm} $C^2_{22}(x_1,x_2,\dots,x_5)$ & 
 \\

\end{tabular}
\end{center}

\medskip

To prove that these are the only configurations in $\mathcal C_4$ when the second neighborhood of $v$ has $5$ vertices one can proceed again by case analysis. Now, less connections are allowed within the second neighborhood but the set of symmetries among the vertices in $W$ is also smaller. For instance, there are three different type of configurations with one edge in $G[W]$: $C_1^2$, $C_2^2$ and $C_3^2$.

\subsection{}
The following configurations are obtained if the second neighborhood  of a vertex has size four. There are two possibilities for the edges between the first and the second neighborhood of $v$, corresponding to the types $C^{3}_j$ and $C^4_j$.

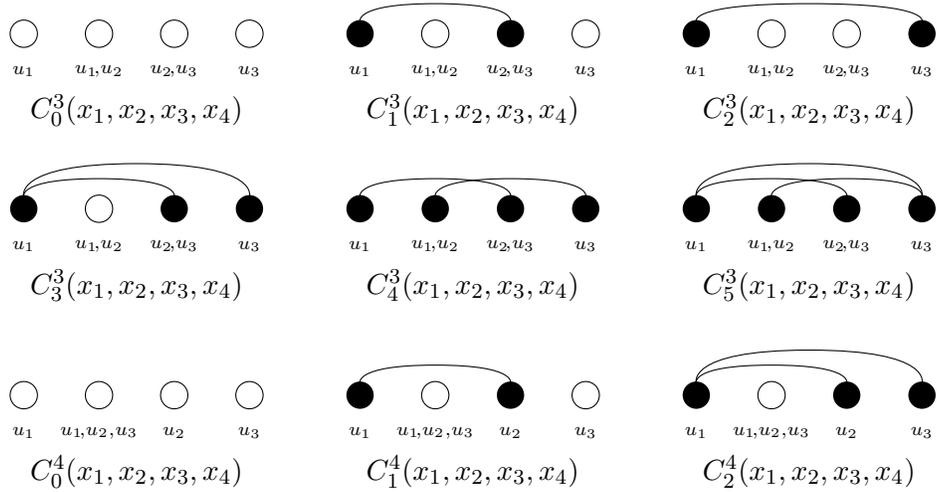
\begin{center}
\begin{tabular}{ c c c }

\begin{tikzpicture}
    \node[shape=circle,draw=black] (A) at (0,0) {};
    \node[shape=circle,draw=black] (B) at (1,0) {};
    \node[shape=circle,draw=black] (C) at (2,0) {};
    \node[shape=circle,draw=black] (D) at (3,0) {};
    
    {\tiny
    \node (G) at (0,-0.5) {$u_1$};
    \node (H) at (1,-0.5) {$u_1\!,\!u_2$};
    \node (I) at (2,-0.5) {$u_2\!,\!u_3$};
    \node (J) at (3,-0.5) {$u_3$};
    }
    
\end{tikzpicture}

 & 
   
  \hspace{0.5cm}
\begin{tikzpicture}
    \node[shape=circle,fill=black] (A) at (0,0) {};
    \node[shape=circle,draw=black] (B) at (1,0) {};
    \node[shape=circle,fill=black] (C) at (2,0) {};
    \node[shape=circle,draw=black] (D) at (3,0) {};
    
    {\tiny
    \node (G) at (0,-0.5) {$u_1$};
    \node (H) at (1,-0.5) {$u_1\!,\!u_2$};
    \node (I) at (2,-0.5) {$u_2\!,\!u_3$};
    \node (J) at (3,-0.5) {$u_3$};
    }

\draw plot [smooth, tension=2] coordinates { (0,0.15) (1,0.4) (2,0.15)};
\end{tikzpicture}

 &
   
  \hspace{0.5cm}
\begin{tikzpicture}
    \node[shape=circle,fill=black] (A) at (0,0) {};
    \node[shape=circle,draw=black] (B) at (1,0) {};
    \node[shape=circle,draw=black] (C) at (2,0) {};
    \node[shape=circle,fill=black] (D) at (3,0) {};
    
    {\tiny
    \node (G) at (0,-0.5) {$u_1$};
    \node (H) at (1,-0.5) {$u_1\!,\!u_2$};
    \node (I) at (2,-0.5) {$u_2\!,\!u_3$};
    \node (J) at (3,-0.5) {$u_3$};
    }

\draw plot [smooth, tension=2] coordinates { (0,0.15) (1.5,0.4) (3,0.15)};
\end{tikzpicture}

 \\ 
 $C^3_{0}(x_1,x_2,x_3,x_4)$ & \hspace{0.5cm} $C^3_{1}(x_1,x_2,x_3,x_4)$ & \hspace{0.5cm} $C^3_{2}(x_1,x_2,x_3,x_4)$\\
 \\

\begin{tikzpicture}
    \node[shape=circle,fill=black] (A) at (0,0) {};
    \node[shape=circle,draw=black] (B) at (1,0) {};
    \node[shape=circle,fill=black] (C) at (2,0) {};
    \node[shape=circle,fill=black] (D) at (3,0) {};
    
    {\tiny
    \node (G) at (0,-0.5) {$u_1$};
    \node (H) at (1,-0.5) {$u_1\!,\!u_2$};
    \node (I) at (2,-0.5) {$u_2\!,\!u_3$};
    \node (J) at (3,-0.5) {$u_3$};
    }
    
\draw plot [smooth, tension=2] coordinates { (0,0.15) (1,0.4) (2,0.15)};
\draw plot [smooth, tension=2] coordinates { (0,0.15) (1.5,0.6) (3,0.15)};
\end{tikzpicture}

 & 
   
  \hspace{0.5cm}
\begin{tikzpicture}
    \node[shape=circle,fill=black] (A) at (0,0) {};
    \node[shape=circle,fill=black] (B) at (1,0) {};
    \node[shape=circle,fill=black] (C) at (2,0) {};
    \node[shape=circle,fill=black] (D) at (3,0) {};
    
    {\tiny
    \node (G) at (0,-0.5) {$u_1$};
    \node (H) at (1,-0.5) {$u_1\!,\!u_2$};
    \node (I) at (2,-0.5) {$u_2\!,\!u_3$};
    \node (J) at (3,-0.5) {$u_3$};
    }

\draw plot [smooth, tension=2] coordinates { (0,0.15) (1,0.4) (2,0.15)};
\draw plot [smooth, tension=2] coordinates { (1,0.15) (2,0.4) (3,0.15)};
\end{tikzpicture}

 &
   
  \hspace{0.5cm}
\begin{tikzpicture}
    \node[shape=circle,fill=black] (A) at (0,0) {};
    \node[shape=circle,fill=black] (B) at (1,0) {};
    \node[shape=circle,fill=black] (C) at (2,0) {};
    \node[shape=circle,fill=black] (D) at (3,0) {};
    
    {\tiny
    \node (G) at (0,-0.5) {$u_1$};
    \node (H) at (1,-0.5) {$u_1\!,\!u_2$};
    \node (I) at (2,-0.5) {$u_2\!,\!u_3$};
    \node (J) at (3,-0.5) {$u_3$};
    }

\draw plot [smooth, tension=2] coordinates { (0,0.15) (1,0.4) (2,0.15)};
\draw plot [smooth, tension=2] coordinates { (1,0.15) (2,0.4) (3,0.15)};
\draw plot [smooth, tension=2] coordinates { (0,0.15) (1.5,0.6) (3,0.15)};
\end{tikzpicture}

 \\ 
 $C^3_{3}(x_1,x_2,x_3,x_4)$ & \hspace{0.5cm} $C^3_{4}(x_1,x_2,x_3,x_4)$ & \hspace{0.5cm} $C^3_{5}(x_1,x_2,x_3,x_4)$\\
 \\
 
\end{tabular}
\end{center}

\begin{center}
\begin{tabular}{ c c c }

\begin{tikzpicture}
    \node[shape=circle,draw=black] (A) at (0,0) {};
    \node[shape=circle,draw=black] (B) at (1,0) {};
    \node[shape=circle,draw=black] (C) at (2,0) {};
    \node[shape=circle,draw=black] (D) at (3,0) {};
    
    {\tiny
    \node (G) at (0,-0.5) {$u_1$};
    \node (H) at (1,-0.5) {$u_1\!,\!u_2,\!u_3$};
    \node (I) at (2,-0.5) {$u_2$};
    \node (J) at (3,-0.5) {$u_3$};
    }
    
\end{tikzpicture}

 & 
   
  \hspace{0.5cm}
\begin{tikzpicture}
    \node[shape=circle,fill=black] (A) at (0,0) {};
    \node[shape=circle,draw=black] (B) at (1,0) {};
    \node[shape=circle,fill=black] (C) at (2,0) {};
    \node[shape=circle,draw=black] (D) at (3,0) {};
    
        {\tiny
    \node (G) at (0,-0.5) {$u_1$};
    \node (H) at (1,-0.5) {$u_1\!,\!u_2,\!u_3$};
    \node (I) at (2,-0.5) {$u_2$};
    \node (J) at (3,-0.5) {$u_3$};
    }

\draw plot [smooth, tension=2] coordinates { (0,0.15) (1,0.4) (2,0.15)};
\end{tikzpicture}

 &
   
  \hspace{0.5cm}
\begin{tikzpicture}
    \node[shape=circle,fill=black] (A) at (0,0) {};
    \node[shape=circle,draw=black] (B) at (1,0) {};
    \node[shape=circle,fill=black] (C) at (2,0) {};
    \node[shape=circle,fill=black] (D) at (3,0) {};
    
        {\tiny
    \node (G) at (0,-0.5) {$u_1$};
    \node (H) at (1,-0.5) {$u_1\!,\!u_2,\!u_3$};
    \node (I) at (2,-0.5) {$u_2$};
    \node (J) at (3,-0.5) {$u_3$};
    }

\draw plot [smooth, tension=2] coordinates { (0,0.15) (1,0.4) (2,0.15)};
\draw plot [smooth, tension=2] coordinates { (0,0.15) (1.5,0.6) (3,0.15)};
\end{tikzpicture}

 \\ 
 $C^4_{0}(x_1,x_2,x_3,x_4)$ & \hspace{0.5cm} $C^4_{1}(x_1,x_2,x_3,x_4)$ & \hspace{0.5cm} $C^4_{2}(x_1,x_2,x_3,x_4)$
 \\

\end{tabular}
\end{center}

\subsection{}
The following configurations are obtained if the second neighborhood  of a vertex has size three. There are two possibilities for the edges between the first and the second neighborhood of $v$, corresponding to the types $C^{5}_j$ and $C^6_j$.

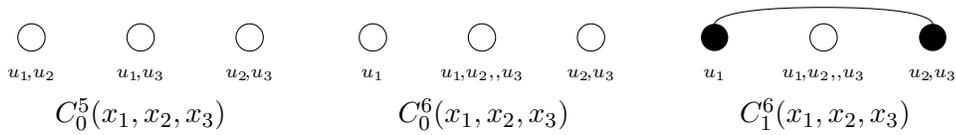
\begin{center}
\begin{tabular}{ c c c }

\begin{tikzpicture}
    \node[shape=circle,draw=black] (A) at (0,0) {};
    \node[shape=circle,draw=black] (B) at (1.45,0) {};
    \node[shape=circle,draw=black] (C) at (2.9,0) {};
    
    {\tiny
    \node (G) at (0,-0.5) {$u_1\!,\! u_2$};
    \node (H) at (1.45,-0.5) {$u_1\!,\!u_3$};
    \node (I) at (2.9,-0.5) {$u_2\!,\!u_3$};
    }
    
\end{tikzpicture}

&
    
  \hspace{0.5cm}
\begin{tikzpicture}
    \node[shape=circle,draw=black] (A) at (0,0) {};
    \node[shape=circle,draw=black] (B) at (1.45,0) {};
    \node[shape=circle,draw=black] (C) at (2.9,0) {};
    
    {\tiny
    \node (G) at (0,-0.5) {$u_1$};
    \node (H) at (1.45,-0.5) {$u_1\!,\!u_2,\!,\!u_3$};
    \node (I) at (2.9,-0.5) {$u_2\!,\!u_3$};
    }
    
\end{tikzpicture}

 & 
   
  \hspace{0.5cm}
\begin{tikzpicture}
    \node[shape=circle,fill=black] (A) at (0,0) {};
    \node[shape=circle,draw=black] (B) at (1.45,0) {};
    \node[shape=circle,fill=black] (C) at (2.9,0) {};
    
    {\tiny
    \node (G) at (0,-0.5) {$u_1$};
    \node (H) at (1.45,-0.5) {$u_1\!,\!u_2,\!,\!u_3$};
    \node (I) at (2.9,-0.5) {$u_2\!,\!u_3$};
    }

\draw plot [smooth, tension=2] coordinates { (0,0.15) (1.45,0.4) (2.9,0.15)};
\end{tikzpicture}

 \\ 
 $C^5_{0}(x_1,x_2,x_3)$ &  \hspace{0.5cm}$C^6_{0}(x_1,x_2,x_3)$ &  \hspace{0.5cm}$C^6_{1}(x_1,x_2,x_3)$ 
 \\

\end{tabular}
\end{center}

\subsection{}
The following configurations are obtained if the second neighborhood  of a vertex has size 2. In this case, the vertex belongs to a copy of $K_{3,3}$.
\medskip

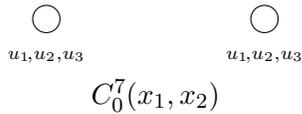
\begin{center}
\begin{tabular}{ c c c }
\begin{tikzpicture}
    \node[shape=circle,draw=black] (A) at (0,0) {};
    \node[shape=circle,draw=black] (C) at (2.9,0) {};
    
    {\tiny
    \node (H) at (0,-0.5) {$u_1\!,\!u_2,\!u_3$};
    \node (H) at (2.9,-0.5) {$u_1\!,\!u_2,\!u_3$};
    }
    
\end{tikzpicture}

 & 
 
 &
 \\ 
 $C^7_{0}(x_1,x_2)$ & & 
 \\

\end{tabular}
\end{center}

\newpage
\section{Proof that unions of $P_{5,2}$ are the only graphs that minimize the occupancy fraction for every $\lambda\in (0,1]$.}
\label{sec:Unqiue52}

In this appendix we prove that unions of  $P_{5,2}$ are the only graphs that attain the minimum in Theorem~\ref{thm:d3g4}. As in the proof of the Theorem~\ref{thm:d3g4}, we will split the proof into $4$ cases. For each case, there is an assignment of $\Lambda_0^*(\lambda)$, $\Lambda_1^*(\lambda)$ and $\Lambda_2^*(\lambda)$ such that $\text{SLACK}_{\mathrm{min}}(\lambda,\Lambda_0^*(\lambda),\Lambda_2^*(\lambda),\Lambda_2^*(\lambda),C)$ defined as in~\eqref{eq:g4slack} is non-negative for every $C\in\mathcal C_4$ and is $0$ for a subset of configurations $\hat C$ that include $C^1_{29}(1,1,1,1,1,1)$, corresponding to the Petersen graph. We need to show that the only solutions induced by graphs and supported in $\hat C$ are, in fact, only supported in $C^1_{29}(1,1,1,1,1,1)$. It follows that unions of $P_{5,2}$ are the only graphs that attain the minimum.
\medskip

For $\lambda\in (0,3/16]$ and $\Lambda_1^*(\lambda), \Lambda_2^*(\lambda)$ as in Claim~\ref{cla:31}, there are $3$ configurations $C$ for which $\text{SLACK}_{\mathrm{min}}(\lambda,0, \Lambda_1^*(\lambda), \Lambda_2^*(\lambda),C)=0$:
\begin{center}
\begin{tabular}{ c c c }

 \begin{tikzpicture}
    \node[shape=circle,fill=black] (A) at (0,0) {};
    \node[shape=circle,fill=black] (B) at (0.5,0) {};
    \node[shape=circle,fill=black] (C) at (1.2,0) {};
    \node[shape=circle,fill=black] (D) at (1.7,0) {};
    \node[shape=circle,fill=black] (E) at (2.4,0) {};
    \node[shape=circle,draw=black] (F) at (2.9,0) {} ;

     \node (G) at (0,-0.5) {$u_1$};
    \node (H) at (0.5,-0.5) {$u_1$};
    \node (I) at (1.2,-0.5) {$u_2$};
    \node (J) at (1.7,-0.5) {$u_2$};
    \node (K) at (2.4,-0.5) {$u_3$};
    \node (L) at (2.9,-0.5) {$u_3$} ;
    
\draw plot [smooth, tension=2] coordinates { (0,0.15) (0.6 ,0.4) (1.2,0.15)};
\draw plot [smooth, tension=2] coordinates { (0,0.15) (1.2 ,0.6) (2.4,0.15)};
\draw plot [smooth, tension=2] coordinates { (0.5,0.15) (1.1,0.5) (1.7,0.15)};
\draw plot [smooth, tension=2] coordinates { (0.5,0) (1.2,0)};
\draw plot [smooth, tension=2] coordinates { (1.7,0) (2.4,0)};
\end{tikzpicture}
 &
\hspace{0.5cm}

\begin{tikzpicture}
    \node[shape=circle,fill=black] (A) at (0,0) {};
    \node[shape=circle,fill=black] (B) at (0.5,0) {};
    \node[shape=circle,fill=black] (C) at (1.2,0) {};
    \node[shape=circle,fill=black] (D) at (1.7,0) {};
    \node[shape=circle,fill=black] (E) at (2.4,0) {};
    \node[shape=circle,fill=black] (F) at (2.9,0) {} ;

     \node (G) at (0,-0.5) {$u_1$};
    \node (H) at (0.5,-0.5) {$u_1$};
    \node (I) at (1.2,-0.5) {$u_2$};
    \node (J) at (1.7,-0.5) {$u_2$};
    \node (K) at (2.4,-0.5) {$u_3$};
    \node (L) at (2.9,-0.5) {$u_3$} ;

 \draw plot [smooth, tension=2] coordinates { (0.5,0) (1.2,0) };
     \draw plot [smooth, tension=2] coordinates { (1.7,0) (2.3,0) };
     \draw plot [smooth, tension=2] coordinates { (0,0.15) (1.45,0.7) (2.9,0.15) };
     \draw plot [smooth, tension=2] coordinates { (0,0.15) (0.85,0.5) (1.7,0.15) };
     \draw plot [smooth, tension=2] coordinates { (1.2,0.15) (2.05,0.5) (2.9,0.15) };
      \draw plot [smooth, tension=2] coordinates { (0.5,0.15) (1.4,0.3) (2.3,0.15) };    
\end{tikzpicture}  
&
\hspace{0.5cm}

 \begin{tikzpicture}
    \node[shape=circle,fill=black] (A) at (0,0) {};
    \node[shape=circle,fill=black] (B) at (0.5,0) {};
    \node[shape=circle,fill=black] (C) at (1.2,0) {};
    \node[shape=circle,fill=black] (D) at (1.7,0) {};
    \node[shape=circle,fill=black] (E) at (2.4,0) {};
    \node[shape=circle,fill=black] (F) at (2.9,0) {} ;

     \node (G) at (0,-0.5) {$u_1$};
    \node (H) at (0.5,-0.5) {$u_1$};
    \node (I) at (1.2,-0.5) {$u_2$};
    \node (J) at (1.7,-0.5) {$u_2$};
    \node (K) at (2.4,-0.5) {$u_3$};
    \node (L) at (2.9,-0.5) {$u_3$} ;
    
\draw plot [smooth, tension=2] coordinates { (0,0.15) (1.45 ,0.7) (2.9,0.15)};
\draw plot [smooth, tension=2] coordinates { (0,0.15) (0.6 ,0.55) (1.2,0.15)};
\draw plot [smooth, tension=2] coordinates { (1.7,0.15) (2.3,0.55) (2.9,0.15)};
\draw plot [smooth, tension=2] coordinates { (0.5,0.15) (1.45,0.4) (2.4,0.15)};
\draw plot [smooth, tension=2] coordinates { (0.5,0) (1.2,0)};
\draw plot [smooth, tension=2] coordinates { (1.7,0) (2.4,0)};
\end{tikzpicture}
  \\  
   $C^1_{28}(1,1,1,1,1,0)$
  & 
\hspace{0.5cm}

  $C^1_{29}(1,1,1,1,1,1)$
  &
\hspace{0.5cm}

 $C^1_{30}(1,1,1,1,1,1)$
 \\
\end{tabular}
\end{center}

To prove that $P_{5,2}$ is the unique minimizer in this range, we need to show that a graph $G$ attaining the minimum is supported  neither in $C_{28}^1(1,1,1,1,1,0)$ nor in $C_{30}^1(1,1,1,1,1,1)$. Choose $I\in \mathcal I(G)$ according to the hard-core model and $v\in V(G)$ uniformly at random, and let $C$ be the random configuration obtained looking at the second neighborhood of $v$. Suppose first that $\Pr(C=C_{28}^1(1,1,1,1,1,0))>0$, then there exists a vertex whose second neighborhood has type $C_{28}^1$. If we select $I=\emptyset$, then this implies $\Pr(C=C_{28}^1(1,1,1,1,1,1))>0$, obtaining a contradiction. Suppose now that $\Pr(C=C_{30}^1(1,1,1,1,1,1))>0$, then there exists a vertex $v$ whose second neighborhood has type $C_{30}^1$. Note that $u_1$ is contained in a cycle of length $4$, thus, its second neighbourhood has type $C_j^i$ for some $i\neq 1$, obtaining a contradiction since all these types appear with probability $0$. We conclude that for every value of $\lambda\in (0, 3/16]$, any graph $G$ that attains the minimum is only supported in $C_{29}^1(1,1,1,1,1,1)$. In other words, $G$ is a union of $P_{5,2}$s.
\medskip


For $\lambda\in [3/16,11/20]$ and $\Lambda_1^*(\lambda), \Lambda_2^*(\lambda)$ as in Claim~\ref{cla:32}, there are $4$ configurations $C$ for which $\text{SLACK}_{\mathrm{min}}(\lambda,0, \Lambda_1^*(\lambda), \Lambda_2^*(\lambda),C)=0$:
\begin{center}
\begin{tabular}{ c c c}

\begin{tikzpicture}
    \node[shape=circle,fill=black] (A) at (0,0) {};
    \node[shape=circle,fill=black] (B) at (0.5,0) {};
    \node[shape=circle,fill=black] (C) at (1.2,0) {};
    \node[shape=circle,draw=black] (D) at (1.7,0) {};
    \node[shape=circle,fill=black] (E) at (2.4,0) {};
    \node[shape=circle,draw=black] (F) at (2.9,0) {} ;

     \node (G) at (0,-0.5) {$u_1$};
    \node (H) at (0.5,-0.5) {$u_1$};
    \node (I) at (1.2,-0.5) {$u_2$};
    \node (J) at (1.7,-0.5) {$u_2$};
    \node (K) at (2.4,-0.5) {$u_3$};
    \node (L) at (2.9,-0.5) {$u_3$} ;
    
\draw plot [smooth, tension=2] coordinates { (0,0.15) (1.2 ,0.6) (2.4,0.15)};
\draw plot [smooth, tension=2] coordinates { (0.5,0)(1.2,0)};
\end{tikzpicture}
&
\hspace{0.5cm}

 \begin{tikzpicture}
 
    \node[shape=circle,fill=black] (A) at (0,0) {};
    \node[shape=circle,fill=black] (B) at (0.5,0) {};
    \node[shape=circle,fill=black] (C) at (1.2,0) {};
    \node[shape=circle,fill=black] (D) at (1.7,0) {};
    \node[shape=circle,fill=black] (E) at (2.4,0) {};
    \node[shape=circle,draw=black] (F) at (2.9,0) {} ;

     \node (G) at (0,-0.5) {$u_1$};
    \node (H) at (0.5,-0.5) {$u_1$};
    \node (I) at (1.2,-0.5) {$u_2$};
    \node (J) at (1.7,-0.5) {$u_2$};
    \node (K) at (2.4,-0.5) {$u_3$};
    \node (L) at (2.9,-0.5) {$u_3$} ;
    
\draw plot [smooth, tension=2] coordinates { (0,0.15) (0.6 ,0.4) (1.2,0.15)};
\draw plot [smooth, tension=2] coordinates { (0,0.15) (1.2 ,0.6) (2.4,0.15)};
\draw plot [smooth, tension=2] coordinates { (0.5,0.15) (1.1,0.5) (1.7,0.15)};
\draw plot [smooth, tension=2] coordinates { (0.5,0) (1.2,0)};
\draw plot [smooth, tension=2] coordinates { (1.7,0) (2.4,0)};
\end{tikzpicture}
 &
 \hspace{0.5cm}

\begin{tikzpicture}
    \node[shape=circle,fill=black] (A) at (0,0) {};
    \node[shape=circle,fill=black] (B) at (0.5,0) {};
    \node[shape=circle,fill=black] (C) at (1.2,0) {};
    \node[shape=circle,fill=black] (D) at (1.7,0) {};
    \node[shape=circle,fill=black] (E) at (2.4,0) {};
    \node[shape=circle,fill=black] (F) at (2.9,0) {} ;

     \node (G) at (0,-0.5) {$u_1$};
    \node (H) at (0.5,-0.5) {$u_1$};
    \node (I) at (1.2,-0.5) {$u_2$};
    \node (J) at (1.7,-0.5) {$u_2$};
    \node (K) at (2.4,-0.5) {$u_3$};
    \node (L) at (2.9,-0.5) {$u_3$} ;

 \draw plot [smooth, tension=2] coordinates { (0.5,0) (1.2,0) };
     \draw plot [smooth, tension=2] coordinates { (1.7,0) (2.3,0) };
     \draw plot [smooth, tension=2] coordinates { (0,0.15) (1.45,0.7) (2.9,0.15) };
     \draw plot [smooth, tension=2] coordinates { (0,0.15) (0.85,0.5) (1.7,0.15) };
     \draw plot [smooth, tension=2] coordinates { (1.2,0.15) (2.05,0.5) (2.9,0.15) };
      \draw plot [smooth, tension=2] coordinates { (0.5,0.15) (1.4,0.3) (2.3,0.15) };    
\end{tikzpicture}  
\\
     $C^1_{3}(1,1,1,0,1,0)$
     
&
\hspace{0.5cm}

   $C^1_{28}(1,1,1,1,1,0)$
  &
  \hspace{0.5cm}
 
  $C^1_{29}(1,1,1,1,1,1)$
 \\
\\
\\
\begin{tikzpicture}
    \node[shape=circle,fill=black] (A) at (0,0) {};
    \node[shape=circle,draw=black] (B) at (0.6,0) {};
    \node[shape=circle,fill=black] (C) at (1.2,0) {};
    \node[shape=circle,fill=black] (D) at (2.3,0) {};
    \node[shape=circle,fill=black] (E) at (2.9,0) {} ;
    
    {\tiny
    \node (G) at (0,-0.5) {$u_1$};
    \node (H) at (0.6,-0.5) {$u_1\!,\!u_2$};
    \node (I) at (1.2,-0.5) {$u_2$};
    \node (J) at (2.3,-0.5) {$u_3$};
    \node (K) at (2.9,-0.5) {$u_3$};
    }

\draw plot [smooth, tension=2] coordinates { (0,0.15) (1.45,0.6) (2.9,0.15)};
\draw plot [smooth, tension=2] coordinates { (1.2,0)  (2.3,0)};
\end{tikzpicture}
&&
\\

 $C^2_{7}(1,0,1,1,1)$ & &

  \\  
\end{tabular}
\end{center}

To prove that $P_{5,2}$ is the unique minimizer in this range, as before, just observe that if any configuration $C^{i}_j(x_1,x_2,\dots,x_s)$ has positive probability to appear, then $C^{i}_{j'}(1,1,\dots,1)$ must also have positive probability to appear, for some $j'$ such that $C^{i}_j$ is \emph{isomorphic} to an \emph{induced subgraph} of $C^{i}_{j'}$. This is not the case for the configurations  $C^1_{3}(1,1,1,1,0,1,0)$, $C^1_{28}(1,1,1,1,1,1,0)$  and $C^2_{7}(1,0,1,1,1)$.   Therefore, any graph attaining the minimum is a union of $P_{5,2}$.
\medskip


For $\lambda\in [11/20,\sqrt{3/5}]$ and $\Lambda_1^*(\lambda), \Lambda_2^*(\lambda)$ as in Claim~\ref{cla:33}, there are $6$ configurations $C$ for which $\text{SLACK}_{\mathrm{min}}(\lambda,0, \Lambda_1^*(\lambda), \Lambda_2^*(\lambda),C)=0$:
\begin{center}
\begin{tabular}{ c c c}

 \begin{tikzpicture}
    \node[shape=circle,fill=black] (A) at (0,0) {};
    \node[shape=circle,draw=black] (B) at (0.5,0) {};
    \node[shape=circle,fill=black] (C) at (1.2,0) {};
    \node[shape=circle,draw=black] (D) at (1.7,0) {};
    \node[shape=circle,fill=black] (E) at (2.4,0) {};
    \node[shape=circle,draw=black] (F) at (2.9,0) {} ;
    
    \node (G) at (0,-0.5) {$u_1$};
    \node (H) at (0.5,-0.5) {$u_1$};
    \node (I) at (1.2,-0.5) {$u_2$};
    \node (J) at (1.7,-0.5) {$u_2$};
    \node (K) at (2.4,-0.5) {$u_3$};
    \node (L) at (2.9,-0.5) {$u_3$} ;
\end{tikzpicture}
&
\hspace{0.5cm}

\begin{tikzpicture}
    \node[shape=circle,fill=black] (A) at (0,0) {};
    \node[shape=circle,fill=black] (B) at (0.5,0) {};
    \node[shape=circle,fill=black] (C) at (1.2,0) {};
    \node[shape=circle,draw=black] (D) at (1.7,0) {};
    \node[shape=circle,fill=black] (E) at (2.4,0) {};
    \node[shape=circle,draw=black] (F) at (2.9,0) {} ;

     \node (G) at (0,-0.5) {$u_1$};
    \node (H) at (0.5,-0.5) {$u_1$};
    \node (I) at (1.2,-0.5) {$u_2$};
    \node (J) at (1.7,-0.5) {$u_2$};
    \node (K) at (2.4,-0.5) {$u_3$};
    \node (L) at (2.9,-0.5) {$u_3$} ;
    
\draw plot [smooth, tension=2] coordinates { (0,0.15) (1.2 ,0.6) (2.4,0.15)};
\draw plot [smooth, tension=2] coordinates { (0.5,0)(1.2,0)};
\end{tikzpicture}
&
\hspace{0.5cm}

\begin{tikzpicture}
    \node[shape=circle,fill=black] (A) at (0,0) {};
    \node[shape=circle,fill=black] (B) at (0.5,0) {};
    \node[shape=circle,fill=black] (C) at (1.2,0) {};
    \node[shape=circle,fill=black] (D) at (1.7,0) {};
    \node[shape=circle,fill=black] (E) at (2.4,0) {};
    \node[shape=circle,fill=black] (F) at (2.9,0) {} ;

     \node (G) at (0,-0.5) {$u_1$};
    \node (H) at (0.5,-0.5) {$u_1$};
    \node (I) at (1.2,-0.5) {$u_2$};
    \node (J) at (1.7,-0.5) {$u_2$};
    \node (K) at (2.4,-0.5) {$u_3$};
    \node (L) at (2.9,-0.5) {$u_3$} ;

 \draw plot [smooth, tension=2] coordinates { (0.5,0) (1.2,0) };
     \draw plot [smooth, tension=2] coordinates { (1.7,0) (2.3,0) };
     \draw plot [smooth, tension=2] coordinates { (0,0.15) (1.45,0.7) (2.9,0.15) };
     \draw plot [smooth, tension=2] coordinates { (0,0.15) (0.85,0.5) (1.7,0.15) };
     \draw plot [smooth, tension=2] coordinates { (1.2,0.15) (2.05,0.5) (2.9,0.15) };
      \draw plot [smooth, tension=2] coordinates { (0.5,0.15) (1.4,0.3) (2.3,0.15) };    
\end{tikzpicture}  
\\
 $C^1_{0}(1,0,1,0,1,0)$
    &\hspace{0.5cm}

    $C^1_{3}(1,1,1,0,1,0)$
  & \hspace{0.5cm}

  $C^1_{29}(1,1,1,1,1,1)$
\\
\\
\\
  
\begin{tikzpicture}
    \node[shape=circle,fill=black] (A) at (0,0) {};
    \node[shape=circle,draw=black] (B) at (0.6,0) {};
    \node[shape=circle,fill=black] (C) at (1.2,0) {};
    \node[shape=circle,draw=black] (D) at (2.3,0) {};
    \node[shape=circle,fill=black] (E) at (2.9,0) {} ;
    
    {\tiny
    \node (G) at (0,-0.5) {$u_1$};
    \node (H) at (0.6,-0.5) {$u_1\!,\!u_2$};
    \node (I) at (1.2,-0.5) {$u_2$};
    \node (J) at (2.3,-0.5) {$u_3$};
    \node (K) at (2.9,-0.5) {$u_3$};
    }
\end{tikzpicture}

&\hspace{0.5cm}

\begin{tikzpicture}
    \node[shape=circle,fill=black] (A) at (0,0) {};
    \node[shape=circle,draw=black] (B) at (0.6,0) {};
    \node[shape=circle,fill=black] (C) at (1.2,0) {};
    \node[shape=circle,fill=black] (D) at (2.3,0) {};
    \node[shape=circle,fill=black] (E) at (2.9,0) {} ;
    
    {\tiny
    \node (G) at (0,-0.5) {$u_1$};
    \node (H) at (0.6,-0.5) {$u_1\!,\!u_2$};
    \node (I) at (1.2,-0.5) {$u_2$};
    \node (J) at (2.3,-0.5) {$u_3$};
    \node (K) at (2.9,-0.5) {$u_3$};
    }

\draw plot [smooth, tension=2] coordinates { (0,0.15) (1.45,0.6) (2.9,0.15)};
\draw plot [smooth, tension=2] coordinates { (1.2,0)  (2.3,0)};
\end{tikzpicture}
&
\hspace{0.5cm}

\begin{tikzpicture}
    \node[shape=circle,fill=black] (A) at (0,0) {};
    \node[shape=circle,draw=black] (B) at (1,0) {};
    \node[shape=circle,fill=black] (C) at (2,0) {};
    \node[shape=circle,fill=black] (D) at (3,0) {};
    
    {\tiny
    \node (G) at (0,-0.5) {$u_1$};
    \node (H) at (1,-0.5) {$u_1\!,\!u_2,\!u_3$};
    \node (I) at (2,-0.5) {$u_2$};
    \node (J) at (3,-0.5) {$u_3$};
    }
    
\end{tikzpicture}

  \\  

  $C^2_{0}(1,0,1,1,0)$
  &\hspace{0.5cm}

 $C^2_{7}(1,0,1,1,1)$
  &\hspace{0.5cm}

 $C^4_{0}(1,0,1,1)$
 \\
\end{tabular}
\end{center}

The argument used previously, proves that the minimizer should be supported in
$C^1_{0}(1,0,1,0,1,0)$ and $C^1_{29}(1,1,1,1,1,1)$. However $C^1_{0}(1,0,1,0,1,0)$ cannot appear with positive probability since otherwise there is a configuration $C^1_j$ with $j\neq 29$ such that $C_j^1(1,1,1,1,1,1)$ appears with positive probability. Therefore, any graph attaining the minimum is a union of $P_{5,2}$.

\medskip

For $\lambda\in [\sqrt{3/5},1]$ and $\Lambda_0^*(\lambda),\Lambda_1^*(\lambda), \Lambda_2^*(\lambda)$ as in Claim~\ref{cla:34}, there are 11 configurations $C$ for which $\text{SLACK}_C(\lambda,\Lambda_0^*(\lambda), \Lambda_1^*(\lambda), \Lambda_2^*(\lambda))=0$:
\begin{center}
\begin{tabular}{ c c c}

 \begin{tikzpicture}
    \node[shape=circle,fill=black] (A) at (0,0) {};
    \node[shape=circle,draw=black] (B) at (0.5,0) {};
    \node[shape=circle,fill=black] (C) at (1.2,0) {};
    \node[shape=circle,draw=black] (D) at (1.7,0) {};
    \node[shape=circle,fill=black] (E) at (2.4,0) {};
    \node[shape=circle,draw=black] (F) at (2.9,0) {} ;
    
    \node (G) at (0,-0.5) {$u_1$};
    \node (H) at (0.5,-0.5) {$u_1$};
    \node (I) at (1.2,-0.5) {$u_2$};
    \node (J) at (1.7,-0.5) {$u_2$};
    \node (K) at (2.4,-0.5) {$u_3$};
    \node (L) at (2.9,-0.5) {$u_3$} ;
\end{tikzpicture}
&
\hspace{0.5cm}

\begin{tikzpicture}
    \node[shape=circle,fill=black] (A) at (0,0) {};
    \node[shape=circle,fill=black] (B) at (0.5,0) {};
    \node[shape=circle,fill=black] (C) at (1.2,0) {};
    \node[shape=circle,fill=black] (D) at (1.7,0) {};
    \node[shape=circle,fill=black] (E) at (2.4,0) {};
    \node[shape=circle,fill=black] (F) at (2.9,0) {} ;

     \node (G) at (0,-0.5) {$u_1$};
    \node (H) at (0.5,-0.5) {$u_1$};
    \node (I) at (1.2,-0.5) {$u_2$};
    \node (J) at (1.7,-0.5) {$u_2$};
    \node (K) at (2.4,-0.5) {$u_3$};
    \node (L) at (2.9,-0.5) {$u_3$} ;
    
\draw plot [smooth, tension=2] coordinates { (0.5,0) (1.2,0)};
\draw plot [smooth, tension=2] coordinates { (1.7,0) (2.4,0)};
\draw plot [smooth, tension=2] coordinates { (0,0.15) (1.45,0.5)  (2.9,0.15)};
\end{tikzpicture}  
&

\hspace{0.5cm}
\begin{tikzpicture}
    \node[shape=circle,fill=black] (A) at (0,0) {};
    \node[shape=circle,fill=black] (B) at (0.5,0) {};
    \node[shape=circle,fill=black] (C) at (1.2,0) {};
    \node[shape=circle,fill=black] (D) at (1.7,0) {};
    \node[shape=circle,fill=black] (E) at (2.4,0) {};
    \node[shape=circle,fill=black] (F) at (2.9,0) {} ;

     \node (G) at (0,-0.5) {$u_1$};
    \node (H) at (0.5,-0.5) {$u_1$};
    \node (I) at (1.2,-0.5) {$u_2$};
    \node (J) at (1.7,-0.5) {$u_2$};
    \node (K) at (2.4,-0.5) {$u_3$};
    \node (L) at (2.9,-0.5) {$u_3$} ;
    
\draw plot [smooth, tension=2] coordinates { (0,0.15) (0.6 ,0.4) (1.2,0.15)};
\draw plot [smooth, tension=2] coordinates { (0.5,0.15) (1.45 ,0.6) (2.4,0.15)};
\draw plot [smooth, tension=2] coordinates { (1.2,0.15) (1.8,0.4) (2.4,0.15)};
\draw plot [smooth, tension=2] coordinates { (1.7,0.15) (2.3,0.4) (2.9,0.15)};
\end{tikzpicture}  
\\
 $C^1_{0}(1,0,1,0,1,0)$
    &\hspace{0.5cm}

    $C^1_{9}(1,1,1,0,1,0)$
  & \hspace{0.5cm}

  $C^1_{16}(1,1,1,1,1,1)$
\\
\\
\\
  
\begin{tikzpicture}
    \node[shape=circle,fill=black] (A) at (0,0) {};
    \node[shape=circle,fill=black] (B) at (0.5,0) {};
    \node[shape=circle,fill=black] (C) at (1.2,0) {};
    \node[shape=circle,fill=black] (D) at (1.7,0) {};
    \node[shape=circle,fill=black] (E) at (2.4,0) {};
    \node[shape=circle,fill=black] (F) at (2.9,0) {} ;

     \node (G) at (0,-0.5) {$u_1$};
    \node (H) at (0.5,-0.5) {$u_1$};
    \node (I) at (1.2,-0.5) {$u_2$};
    \node (J) at (1.7,-0.5) {$u_2$};
    \node (K) at (2.4,-0.5) {$u_3$};
    \node (L) at (2.9,-0.5) {$u_3$} ;
    
\draw plot [smooth, tension=2] coordinates { (0,0.15) (0.85 ,0.6) (1.7,0.15)};
\draw plot [smooth, tension=2] coordinates { (1.7,0.15) (2.3 ,0.4) (2.9,0.15)};
\draw plot [smooth, tension=2] coordinates { (2.9,0.15) (1.7,0.6) (0.5,0.15)};
\draw plot [smooth, tension=2] coordinates { (0.5,0)  (1.2,0)};
\draw plot [smooth, tension=2] coordinates { (1.2,0.15) (1.8,0.4) (2.4,0.15)};
\end{tikzpicture}

&

  \hspace{0.5cm}
  
  \begin{tikzpicture}
    \node[shape=circle,fill=black] (A) at (0,0) {};
    \node[shape=circle,fill=black] (B) at (0.5,0) {};
    \node[shape=circle,fill=black] (C) at (1.2,0) {};
    \node[shape=circle,fill=black] (D) at (1.7,0) {};
    \node[shape=circle,fill=black] (E) at (2.4,0) {};
    \node[shape=circle,fill=black] (F) at (2.9,0) {} ;

     \node (G) at (0,-0.5) {$u_1$};
    \node (H) at (0.5,-0.5) {$u_1$};
    \node (I) at (1.2,-0.5) {$u_2$};
    \node (J) at (1.7,-0.5) {$u_2$};
    \node (K) at (2.4,-0.5) {$u_3$};
    \node (L) at (2.9,-0.5) {$u_3$} ;

 \draw plot [smooth, tension=2] coordinates { (0.5,0) (1.2,0) };
     \draw plot [smooth, tension=2] coordinates { (1.7,0) (2.3,0) };
     \draw plot [smooth, tension=2] coordinates { (0,0.15) (1.45,0.7) (2.9,0.15) };
     \draw plot [smooth, tension=2] coordinates { (0,0.15) (0.85,0.5) (1.7,0.15) };
     \draw plot [smooth, tension=2] coordinates { (1.2,0.15) (2.05,0.5) (2.9,0.15) };
      \draw plot [smooth, tension=2] coordinates { (0.5,0.15) (1.4,0.3) (2.3,0.15) };    
\end{tikzpicture}  
&

  \hspace{0.5cm}
  
\begin{tikzpicture}
    \node[shape=circle,fill=black] (A) at (0,0) {};
    \node[shape=circle,draw=black] (B) at (0.6,0) {};
    \node[shape=circle,fill=black] (C) at (1.2,0) {};
    \node[shape=circle,fill=black] (D) at (2.3,0) {};
    \node[shape=circle,draw=black] (E) at (2.9,0) {} ;
    
    {\tiny
    \node (G) at (0,-0.5) {$u_1$};
    \node (H) at (0.6,-0.5) {$u_1\!,\!u_2$};
    \node (I) at (1.2,-0.5) {$u_2$};
    \node (J) at (2.3,-0.5) {$u_3$};
    \node (K) at (2.9,-0.5) {$u_3$};
    }
\end{tikzpicture}

  \\  

  $C^1_{24}(1,1,1,1,1,1)$
  &\hspace{0.5cm}

 $C^1_{29}(1,1,1,1,1,1)$
  &\hspace{0.5cm}

 $C^2_{0}(1,0,1,1,0)$
 \\
\\
\\

\begin{tikzpicture}
    \node[shape=circle,draw=black] (A) at (0,0) {};
    \node[shape=circle,fill=black] (B) at (0.6,0) {};
    \node[shape=circle,draw=black] (C) at (1.2,0) {};
    \node[shape=circle,draw=black] (D) at (2.3,0) {};
    \node[shape=circle,draw=black] (E) at (2.9,0) {} ;
    
    {\tiny
    \node (G) at (0,-0.5) {$u_1$};
    \node (H) at (0.6,-0.5) {$u_1\!,\!u_2$};
    \node (I) at (1.2,-0.5) {$u_2$};
    \node (J) at (2.3,-0.5) {$u_3$};
    \node (K) at (2.9,-0.5) {$u_3$};
    }
\end{tikzpicture}
&

\hspace{0.5cm}

\begin{tikzpicture}
    \node[shape=circle,fill=black] (A) at (0,0) {};
    \node[shape=circle,draw=black] (B) at (1,0) {};
    \node[shape=circle,fill=black] (C) at (2,0) {};
    \node[shape=circle,fill=black] (D) at (3,0) {};
    
    {\tiny
    \node (G) at (0,-0.5) {$u_1$};
    \node (H) at (1,-0.5) {$u_1\!,\!u_2$};
    \node (I) at (2,-0.5) {$u_2\!,\!u_3$};
    \node (J) at (3,-0.5) {$u_3$};
    }
    
\end{tikzpicture}

&
\hspace{0.5cm}

\begin{tikzpicture}
    \node[shape=circle,draw=black] (A) at (0,0) {};
    \node[shape=circle,fill=black] (B) at (1,0) {};
    \node[shape=circle,draw=black] (C) at (2,0) {};
    \node[shape=circle,draw=black] (D) at (3,0) {};
    
    {\tiny
    \node (G) at (0,-0.5) {$u_1$};
    \node (H) at (1,-0.5) {$u_1\!,\!u_2,\!u_3$};
    \node (I) at (2,-0.5) {$u_2$};
    \node (J) at (3,-0.5) {$u_3$};
    }
\end{tikzpicture}
\\  

  $C^2_{0}(0,1,0,0,0)$
  &\hspace{0.5cm}

$C^3_{0}(0,1,0,0)$
  &\hspace{0.5cm}

 $C^4_{0}(1,0,1,1)$
 \\
\\
\\

\begin{tikzpicture}
    \node[shape=circle,fill=black] (A) at (0,0) {};
    \node[shape=circle,draw=black] (B) at (1.45,0) {};
    \node[shape=circle,draw=black] (C) at (2.9,0) {};
    
    {\tiny
    \node (G) at (0,-0.5) {$u_1\!,\! u_2$};
    \node (H) at (1.45,-0.5) {$u_1\!,\!u_3$};
    \node (I) at (2.9,-0.5) {$u_2\!,\!u_3$};
    }
    
\end{tikzpicture}

&
    
  \hspace{0.5cm}
\begin{tikzpicture}
    \node[shape=circle,draw=black] (A) at (0,0) {};
    \node[shape=circle,draw=black] (B) at (1.45,0) {};
    \node[shape=circle,fill=black] (C) at (2.9,0) {};
    
    {\tiny
    \node (G) at (0,-0.5) {$u_1$};
    \node (H) at (1.45,-0.5) {$u_1\!,\!u_2,\!,\!u_3$};
    \node (I) at (2.9,-0.5) {$u_2\!,\!u_3$};
    }
    
\end{tikzpicture}
&
\\
 $C^5_{0}(1,0,0)$
  &\hspace{0.5cm}

$C^6_{0}(0,0,1)$
  &   
 
\end{tabular}
\end{center}
We can discard the configurations $C^2_{0}(1,0,1,1,0)$, $C^2_{0}(0,1,0,0,0)$, $C^3_{0}(0,1,0,0)$, $C^4_{0}(1,0,1,1)$,  $C^5_{0}(1,0,0)$ and $C^6_{0}(0,0,1)$ using the same argument as before. 

Suppose that $C_{24}^1(1,1,1,1,1,1)$ appears with positive probability.  Let $v$ be a vertex such that there exists an $I\in \mathcal I(G)$ that induces the configuration $C_{24}^1(1,1,1,1,1,1)$ in its second neighborhood. As usual, let $w_{ij}$ for $i\in\{1,2,3\}$ and $j\in\{1,2\}$, denote the second neighbors of $v$, and let $W=\{w_{ij}|\; i\in\{1,2,3\},\, j\in\{1,2\}\}$. The vertices $w_{11}$ and $w_{31}$ have degree $1$ in $G[W]$. The other vertices in $W$ have degree $2$ in $G[W]$ and thus have no neighbor at distance $3$ from $v$. Let $z$ be unique neighbor of $w_{11}$ at distance $3$ from $v$. If $z$ is also a neighbor of $w_{31}$, then the independent set $I=\{z\}$ induces a configuration equivalent up to symmetries to  $C_{10}^1(1,1,1,0,1,0)$. If $z$ is not a neighbor of $w_{31}$ then  the independent set $I=\{z\}$ induces a configuration equivalent up to symmetries to  $C_{18}^1(1,1,1,1,1,0)$. In both cases we obtain a contradiction since these two configurations appear with probability $0$.

A similar argument also discards the configurations $C_{16}^1(1,1,1,1,1,1)$ and $C_{9}^1(1,1,1,1,1,1)$.

Finally, we can discard $C_{0}^1(1,0,1,0,1,0)$ in the same way we did it in the case where $\lambda\in[11/20,\sqrt{3/5}]$.
We conclude that any graph attaining the minimum is a union of $P_{5,2}$.

\end{document}